\documentclass[11pt]{amsart}
\usepackage{amsmath,amsthm,amssymb,latexsym,color}

\usepackage{hyperref}

\numberwithin{equation}{section}

\theoremstyle{plain}
\newtheorem{theor}{Theorem}[section]
\newtheorem{prop}[theor]{Proposition}

\newtheorem{cor}[theor]{Corollary}
\newtheorem{lemma}[theor]{Lemma}

\theoremstyle{remark}
\newtheorem{rem}[theor]{Remark}

\def\R{{\mathbb R}}
\def\C{{\mathbb C}}
\def\P{{\mathbb P}}
\def\N{{\mathbb N}}

\def\col{{\rm col}}
\def\row{{\rm row}}
\def\Prob{{\mathbb P}}
\def\Exp{{\mathbb E}}

\def\Event{{\mathcal E}}
\def\Im{{\rm Im}}
\def\Re{{\rm Re}}

\def\Proj{{\rm Proj}}
\def\spn{{\rm span}}
\def\Id{{\rm Id}}
\def\cf{{\mathcal L}}

\def\ker{{\rm ker}}
\def\Var{{\rm Var}}

\def\Am{{\mathcal A}}
\def\Ael{{\mathcal C}}
\def\ChainsSet{{\mathcal W}}
\def\IsoChains{{\mathcal I}}

\def\Gr{{\mathcal G}}
\def\inneigh{\partial_{in}}
\def\outneigh{\partial_{out}}
\def\grc{{\mathbf G}}
\def\grch{{\bar{\mathbf G}}}
\def\indicator{{\bf 1}}

\newcommand{\csubm}[2]
{
{#1}^{#2}_{col}
}

\newcommand{\subgr}[2]
{
{#1}^{#2}
}

\newcommand{\E} {\mathbb{E}}

\newcommand{\smallrefer}[1]{{\tiny\text{\ref{#1}}}}

\newcommand{\etc} {,\ldots,}

\newcommand{\EE} {\mathcal{E}}
\newcommand{\FF} {\mathcal{F}}
\newcommand{\LL} {\mathcal{L}}
\newcommand{\NN} {\mathcal{N}}

\renewcommand{\a} {\alpha}
\newcommand{\e} {\varepsilon}
\renewcommand{\d} {\delta}

\renewcommand{\t} {\tau}
\renewcommand{\l} {\lambda}

\DeclareMathOperator{\dist }{dist}
\DeclareMathOperator{\Span }{span}

\DeclareMathOperator{\supp}{supp}
\DeclareMathOperator{\Max}{Max}
\DeclareMathOperator*{\tr}{tr}

\newcommand{\norm}[1]{\left \| #1 \right \|}

\title
{
The sparse circular law under minimal assumptions
}
\author{Mark Rudelson}
\address[Mark Rudelson]{University of Michigan}
\email{rudelson@umich.edu}

\author{Konstantin Tikhomirov}
\address[Konstantin Tikhomirov]{Princeton University}
\email{kt12@math.princeton.edu}

\begin{document}

\maketitle

\begin{abstract}
The circular law asserts that the empirical distribution of eigenvalues
of appropriately normalized $n\times n$ matrix with i.i.d.\ entries
converges to the uniform measure on the unit disc as the dimension $n$ grows to infinity.
Consider an $n\times n$ matrix $A_n=(\delta_{ij}^{(n)}\xi_{ij}^{(n)})$, where $\xi_{ij}^{(n)}$
are copies of a real random variable of unit variance,
variables $\delta_{ij}^{(n)}$ are Bernoulli ($0/1$) with $\Prob\{\delta_{ij}^{(n)}=1\}=p_n$,
and $\delta_{ij}^{(n)}$ and $\xi_{ij}^{(n)}$, $i,j\in[n]$, are jointly
independent.
In order for the circular law to hold for the sequence $\big(\frac{1}{\sqrt{p_n n}}A_n\big)$,
one has to assume that $p_n n\to \infty$.
We derive the circular law under this minimal assumption.
\end{abstract}

\tableofcontents

\section{Introduction}

For any $n\times n$ matrix $B$, denote by $\mu_n(B)$ the spectral measure of $B$, that is,
the probability measure
$$\mu_n(B):=\frac{1}{n}\sum\limits_{i=1}^n \delta_{\lambda_i(B)},$$
where $\lambda_1(B),\dots,\lambda_n(B)$ are eigenvalues of $B$.

Let $(A_n)$ be a sequence of random matrices where for each $n$, the matrix $A_n$
has i.i.d.\ entries equidistributed with a real or complex random variable $\xi$ of unit variance.
The {\it circular law} for $(A_n)$ asserts that the sequence of spectral measures $\mu_n(\frac{1}{\sqrt{n}}A_n)$
converges weakly (in probability and almost surely) to the uniform measure on the unit disc of the complex plane \cite{Tao-Vu circ}.

The paper \cite{Tao-Vu circ} is a culmination of a line of research which includes works
\cite{Ginibre,Edelman,Girko, Bai, GoTi, PanZhou, Tao-Vu circ 2+e}, where the circular law was established under additional assumptions
on the distribution of the entries.
The case of Gaussian matrices, when an explicit formula for joint distribution of the matrix eigenvalues
is available, was treated in \cite{Ginibre,Edelman}.
For general distributions of entries, the known proofs of the circular law are based on the {\it Hermitization}
strategy introduced by Girko \cite{Girko} (see Section~\ref{s: overview} below).
Following the strategy, Bai \cite{Bai} established the law when the matrix entries have a uniformly bounded density
and satisfy some additional moment conditions.
The assumption of bounded density, which allows to easily overcome the problem of singularity for shifted
matrices $A_n-z\,\Id$, was removed in \cite{GoTi,PanZhou,Tao-Vu circ 2+e},
following a rapid progress in understanding invertibility of non-Hermitian random matrices
\cite{Rudelson annmath, Tao-Vu annmath,RV invertibility}.
We refer to survey \cite{BC circ} for further information on the history of the circular law.

A natural counterpart of the above setting are sparse non-Hermitian random matrices.
Now, for each $n$ let $A_n$ be a random $n\times n$ matrix with entries of the form $\delta_{ij}^{(n)}\xi_{ij}^{(n)}$,
where $\xi_{ij}^{(n)}$ are independent copies of a random variable with unit variance, and
$\delta_{ij}^{(n)}$ are Bernoulli random variables jointly independent with $\xi_{ij}^{(n)}$,
with $\Prob\{\delta_{ij}^{(n)}=1\}=p_n$,
for some numbers $(p_n)_{n=1}^\infty$.
Under some additional moment assumptions and under the condition that for some
(arbitrary) fixed $\varepsilon>0$ the sequence $p_n$ satisfies $p_n n\geq n^\varepsilon$,
the circular law for the sequence of spectral measures of
matrices $\frac{1}{\sqrt{p_n n}}A_n$ was established in \cite{Tao-Vu circ 2+e, GoTi, Wood}.
Compared to the dense regime, additional difficulties in the sparse setting
arise when bounding from below the smallest and ``smallish'' singular values of shifted matrices $A_n-z\,\Id$
(see Section~\ref{s: overview} for further discussion).
For $p_n n\geq C\log n$, strong lower bounds on $s_{\min}(A_n-z\,\Id)$ were obtained
in \cite{BR inv,BR circ}, which allowed to prove the circular law for $(\frac{1}{\sqrt{p_n n}}A_n)$
when $p_n n$ is at least polylogarithmic in dimension and $\xi$ is subgaussian of zero mean \cite{BR circ}.

Let us note that for a special model of random matrices -- adjacency matrices of random $d$--regular directed graphs
--- the circular law
was recently established in papers \cite{Cook circ, BCZ} (for degree $d$ at least polylogarithmic in dimension)
and in \cite{LLTTY circ} (for $d$ slowly growing to infinity with $n$).
Paper \cite{LLTTY circ} was the first to treat the case of non-Hermitian matrices with a sublogarithmic number of non-zero elements
in rows and columns.
One of key elements of the proof in \cite{LLTTY circ} is a lower bound on the smallest singular value of a shifted adjacency
matrix, derived in \cite{Cook circ,LLTTY smin}.
It was shown in \cite{LLTTY smin}
that, for $d$ slowly growing to infinity with $n$, the smallest singular value of the uniform random $d$--regular matrix
is bounded below by a constant (negative) power of $n$ with probability going to one as
$n\to\infty$.

Analogous assertion for $s_{\min}(A_n)$ is false for the sparse model with i.i.d.\ entries discussed above.
Indeed, if $p_n=\Prob\{\delta_{ij}^{(n)}=1\}\leq\log n/n$ then with constant (non-zero) probability the matrix $A_n$ is singular.
The presence of a large number of zero rows and columns in the very sparse regime
requires a completely different approach to studying invertibility of the shifted matrices $A_n-z\,\Id$, compared with \cite{Tao-Vu circ 2+e, GoTi, Wood,BR circ}
(see Section~\ref{s: overview}).
For a real random variable $\psi$, define the concentration function $\cf(\psi,t):=\sup_{r\in\R}\Prob\{|\psi-r|\leq t\}$, $t\geq 0$.
One of the main technical results of this paper is
\begin{theor}[{Bound on $s_{\min}$ of a shifted matrix; Theorem~\ref{th: bound on smin}}]\label{th: smin}
For any $\alpha> 1$ there are $C_\alpha,c_\alpha>0$ depending only on $\alpha$ with the following property.
Let $n\geq C_\alpha$, $p\in(0,1]$ with $C_\alpha\leq pn\leq n^{1/8}$,
and let $A=(a_{ij})$ be a random $n\times n$ matrix with i.i.d.\ real valued entries $a_{ij}=\delta_{ij}\xi_{ij}$,
where $\delta_{ij}$ is the Bernoulli ($0/1$) variable with $\Prob\{\delta_{ij}=1\}=p$ and $\xi_{ij}$ is a variable
of unit variance
independent of $\delta_{ij}$ and such that $\cf(\xi_{ij}, 1/\alpha)\leq 1-1/\alpha$.
Further, assume that $z\in\C$ is such that
$|z|\leq pn$ and $|\Im(z)|\geq 1/\alpha$.
Then
$$\Prob\big\{s_{\min}(A-z\,\Id)\leq e^{-C_\alpha\log^3 n}\big\}\leq (pn)^{-c_\alpha}.$$
\end{theor}

The above theorem, together with estimates of intermediate singular values, allows to prove the main result of the paper:
\begin{theor}\label{th: main}
Let $\xi$ be a real random variable with unit variance.
For each $n\geq 1$, let $p_n$ satisfy $p_n n\leq n^{1/8}$,
and assume additionally that $\lim\limits_{n\to\infty}p_n n=\infty$. Further, for every $n$ let
$A_n$ be an $n\times n$ random matrix with i.i.d.\ entries $a_{ij}=\delta_{ij}\,\xi_{ij}$, where $\delta_{ij}$
is a Bernoulli ($0/1$) random variable with $\Prob\{\delta_{ij}=1\}=p_n$ and $\xi_{ij}$
are i.i.d.\ random variables equidistributed with $\xi$ (and mutually independent with $\delta_{ij}$).
Then, as $n$ converges to infinity, the empirical spectral distribution of $\frac{1}{\sqrt{p_n n}}A_n$
converges weakly in probability to the uniform measure on the unit disc of the complex plane.
\end{theor}

Note that for finite $p_n n$, the multiplicity of zero eigenvalue is bounded from below by a constant proportion of $n$
with a large probability, so convergence to the uniform distribution on the disc does not hold. In that respect,
our theorem is proved under the minimal assumptions on the sparsity.

%


\subsection{Acknowledgement}

Part of this research was performed while the authors were in residence at the Mathematical Sciences Research Institute  (MSRI) in Berkeley, California, during the Fall semester of 2017,
and at the Institute for Pure and Applied Mathematics (IPAM) in Los Angeles, California, during May and June of 2018.
Both institutions are supported by the National Science Foundation.
Part of this research was performed while the first author visited Weizmann Institute of Science in Rehovot, Israel, where he held Rosy and Max Varon Professorship. We are grateful to all these institutes for their hospitality and for creating an excellent work environment.

The research of the first author was supported in part by the NSF grant DMS 1464514 and by a fellowship from the Simons Foundation.
The research of the second named author was supported by the Viterbi postdoctoral fellowship while in residence at the Mathematical Sciences Research Institute.

\section{Overview of the proof}\label{s: overview}

The circular law was initially proved by Ginibre \cite{Ginibre} for matrices with i.i.d. complex normal entries, and by Edelman \cite{Edelman} for i.i.d. real normal entries.
All known proofs of the circular law for more general classes of random matrices rely on the strategy put forward by Girko \cite{Girko}.
This strategy is based on using the logarithmic potentials of the empirical measures of the eigenvalues.
Namely, let $B_n, \ n \in \N$ be a sequence of matrices, and let
\[
 \mu_n= \frac{1}{n} \sum_{j=1}^n \delta_{\lambda_j(B_n)}
\]
be the empirical measures of their eigenvalues. The measures $\mu_n$ converge to a deterministic measure $\mu$ weakly in probability if for any bounded continuous function $f: \C \to \C$, $\int f \, d \mu_n \to \int f \, d \mu$ in probability.
To establish this convergence, it is enough to show that the logarithmic potentials of $\mu_n$,
\[
 F_n(z)= \int_{\C} \log |z-w| \, d \mu_n(w)
\]
converge to the logarithmic potential of $\mu$ a.e. The logarithmic potential can be rewritten as
\[
 F_n(z)= \frac{1}{n} \log | \text{det} (B_n- z \Id_n)| = \frac{1}{n} \sum_{j=1}^n \log |\lambda_j(B_n)-z|
 = \frac{1}{2n} \sum_{j=1}^{2n} \log |\lambda_j(H_n(z))|,
\]
where $H_n(z)=\left( \begin{smallmatrix} 0 & (B_n-z \Id_n) \\ (B_n-z \Id_n)^* & 0 \end{smallmatrix} \right)$ is a Hermitian matrix. The eigenvalues of $H_n(z)$ are the singular values of $B_n-z \Id_n$ and their negatives.
Denoting the empirical measures of the singular values of $B_n-z \Id_n$ by $\nu_{n,z}$, we have to establish the convergence of $\int_0^{\infty} \log x \,  d \nu_{n,z}(x)$ for almost any $z \in \C$.
This argument allows to pass from the empirical measures of the eigenvalues to more stable empirical measures of the singular values.
To establish the convergence of logarithmic potentials in the latter case, it would be sufficient to prove the weak convergence of the  measures $\nu_{n,z}$ to some limit measure as well as the uniform integrability of the function $\log x$ with respect to $\nu_{n,z}$.
The last step is needed as the function $\log x$ is unbounded at $0$ and $\infty$.

To prove the circular law for sparse random matrices, we set $B_n=(1/\sqrt{p_n n})A_n$.
In this case, the weak convergence of the measures $\nu_{n,z}$ can be derived following the methods already existing in the literature.
The main problem therefore is establishing the uniform integrability of the logarithm. It splits in two parts: checking the uniform integrability at $\infty$ and at $0$. The first one turns out to be simple due to the fact that the Hilbert-Schmidt norm of $(1/\sqrt{p_n n})A_n$ has a finite second moment:
\[
 \E \norm{(1/\sqrt{p_n n})A_n}_{HS}^2=\frac{1}{p_n n} \sum_{i,j=1}^n  \E |(A_n)_{i,j}|^2 
 < \infty.
\]
Thus, the derivation of the circular law reduces to checking the uniform integrability of $\log x$ at $0$ with respect to  the measures $\nu_{n,z}$.
Although this looks like a minor technical issue, this  was a main step in the proof in all other
settings where the circular law was established.
Attacking it required developing a number of different methods ranging from additive combinatorics and harmonic analysis to measure concentration and convex geometry.
Yet, checking the uniform integrability for very sparse matrices present multiple new challenges which cannot be handled by these techniques.
This means that although Girko's strategy can be used in proving the circular law for very sparse random matrices, its implementation requires new ideas at each step.

Let us discuss these challenges in more details. The first, and usually the most difficult step is obtaining a lower bound for the smallest singular value of $\widetilde{A}_{n,z}:=(1/\sqrt{p_n n})A_n- z \Id_n$. Such bound frequently comes in the form $\P (s_n(\widetilde{A}_{n,z}) < n^{-c}) = o(1)$ for some absolute constant $c>0$. If proved, this bound allows to estimate $m(n)=o(n/\log n)$ smallest singular values of $\widetilde{A}_{n,z}$ by the minimal one and conclude that $(1/n)\sum_{j=n-m(n)}^n \log s_j(\widetilde{A}_{n,z}) =o(1)$ with probability $1-o(1)$.
In \cite{GoTi, PanZhou, Tao-Vu circ 2+e, Tao-Vu circ, Wood, BR circ,Cook circ, LLTTY smin}, the bound on the smallest singular value was uniform over $z$.
If we consider the range of sparsity $p_n n < \log n$, such uniform bound cannot hold as the matrix $\widetilde{A}_{n,z}$ contains a zero row with high probability whenever $z=0$.
It may seem that this problem has an easy fix. Since we have to bound the smallest singular value for a.e. $z \in \C$, we can  assume that $z \neq 0$. This would ensure the absence of entirely zero rows.
However, the zero rows is not the only obstacle we have to tackle to bound the smallest singular value. Consider, for example, the matrix of the form
\begin{equation}  \label{eq: Jordan}
B_n=
 \begin{pmatrix}
 Z_k & V_{n-k} \\
 0 & W_{n-k}
 \end{pmatrix},
  \quad \text{with } Z_k=
 \begin{pmatrix}
 z & 2z & 0  & \cdots & 0  \\
 0 & z & 2z  & \cdots & 0 \\
 & \cdots &   & \cdots&  \\
  0 & & \ldots & &  z  \\
 \end{pmatrix},
\end{equation}
where $Z_k$ is a $k \times k$ matrix, and $V_{n-k}, \ W_{n-k}$ is any $k \times (n-k)$ and $(n-k) \times (n-k)$ matrices respectively.
We can choose $V_{n-k}$ and $W_{n-k}$ so that all rows and columns of the matrix $B_n$ are non-zero. However, regardless of the value of $z$ and the choice of $V_{n-k}, \ W_{n-k}$, the smallest singular value of $B_n$ satisfies $s_n(B_n) \le 2^{-k+1} $. To be able to bound the smallest singular value from below in our setting, we have to identify the ``almost singular'' sparse deterministic matrices and show that the matrix $(1/\sqrt{p_n n})A_n- z I_n$ cannot be of this type with any significant probability.

On a more technical level, the estimates on the smallest singular value obtained in the papers mentioned above rely on discretization of the sphere $S^{n-1}$ using $\varepsilon$-nets and approximation of certain subsets of the sphere using these nets. A uniform estimate of $\norm{B_n x}_2$ over $x$ from the $\varepsilon$-net uses the union bound. However,
if we know only that $p_n n \to \infty$ without any prescribed rate, the union bound becomes largely unavailable.
A related problem appeared in \cite{LLTTY smin} where a lower bound on the smallest singular value of random $d$-regular matrices is derived for $d\to\infty$ with arbitrarily slow convergence,
however, absence of zero rows/columns and dependencies make that setting completely different.
These obstacles show that obtaining a smallest singular value bound would require developing a new method taking into account the structure of the non-zero entries of $A_n$ as well as replacing the classical $\varepsilon$-net argument with a more delicate discretization approach. We will discuss the details of our method below.

Besides the smallest singular value of the matrix $B_n$, the uniform integrability of the logarithmic potential requires the bound on the smallish ones. More precisely, we have to show that the contribution of these singular values
$(1/n)\sum \log s_j(\widetilde{A}_{n,z}) $, where the sum is taken over $j \le n-m(n)$ with $s_j(\widetilde{A}_{n,z}) \le \delta$ can be made arbitrarily small by choosing an appropriate $\delta$ independently of $n$.
In some previously considered settings this was a relatively easier step.
Following \cite{Tao-Vu circ}, one can use the negative second moment identity to obtain such bound. This identity allows to bound the singular value $s_{n-j}(B_n)$ of an $n \times n$ matrix $B_n$ in terms of the distances between one row of $B_n$ and the linear span of $n-j$ other rows. To obtain a small ball probability estimate for such distance, one uses the measure concentration. Then passing from the estimate of a single distance to the negative second moment uses the union bound over rows. Yet, as before, for very small $p_n$, the measure concentration estimate we can obtain this way is too weak to be combined with the union bound. Moreover, to guarantee the uniform integrability, we have to bound many intermediate singular values at once. In previous papers this was also achieved through using the union bound. In \cite{LLTTY circ}, which deals with the circular law for adjacency matrices of $d$--regular graphs in the very sparse regime (with $d\to\infty$ arbitrarily slowly)
a similar problem was resolved by deriving strong small ball probability bounds for those distances, however, the techniques are tailored to the $d$--regular setting and cannot be applied in our context.
In short, the weak probability estimates which preclude using the union bound is the main challenge in considering sub-logarithmic values of $p_n n$.

To overcome this obstacle, we introduce a new method. We will define a global event $\Event_{good}$ which occurs with probability $1-o(1)$. This event would reflect both the structure of the non-zero entries of the matrix and the magnitudes of the entries. The aim of this construction is to ensure that conditioned on $\Event_{good}$, we can obtain much better probability estimates allowing us to use the union bound whenever necessary.
 Construction of this event $\Event_{good}$ occupies a significant part of this paper.

To determine the obstacles to a good smallest singular value estimate, let us look at example \eqref{eq: Jordan} again. It immediately points to one of the possible problems, namely, the presence of the columns having a small support. Furthermore, if we replace the coefficient $2$ in this example, say, by $1/2$, the upper estimate for $s_n(B_n)$ is no longer true in general. This means that we have to pay a special attention to the support of the columns as well as to the distribution of entries having large absolute values. To account for both phenomena we associate to the random matrix $A_n=(\delta_{ij}\xi_{ij})$
a random directed bipartite graph $\Gr$ defined as follows.
The vertex set of the graph is $[n]\sqcup[n]$ (the union of {\it left} and {\it right} vertex sets).
For every left vertex $i$ and a right vertex $j$, there is a directed edge from $i$ to $j$ ($i\to j$)
if and only of $\delta_{ij}=1$, and a directed edge $i\leftarrow j$ iff $|\delta_{ij}\xi_{ij}|\geq 1/\alpha$,
where $\alpha>0$ is a parameter.
Alternatively, the graph can be described by introducing an auxiliary collection of i.i.d random Bernoulli variables $(\mu_{ij})$
mutually independent with $\delta_{ij}$, such that $\Prob\{\mu_{ij}=1\}=\Prob\{\xi_{ij}\geq 1/\alpha\}$.
Then $i\leftarrow j$ iff $\delta_{ij}\mu_{ij}=1$. The graph can be analyzed independently of the matrix $A_n$.

A column having too few large entries will correspond to a right vertex of a small out-degree in this encoding. We will regard these vertices as exceptional. After removing the exceptional vertices and all their left neighbors, we will get a subgraph, some of the right vertices of which can have a small out-degree. We will add these to the exceptional vertices and continue the process iteratively. The precise definition  of the set of exceptional vertices appears in Subsection \ref{subs: types}, where they are called vertices of a finite type. We analyze this set in Subsection \ref{subs: cardinality} and show that with probability close to $1$, the set of exceptional vertices has cardinality at most $\exp(-c p_n n) \cdot n$.

Note that for $z \neq 0$, the graph associated to the matrix $\widetilde{A}_{n,z}$ has all horizontal edges $j \to j$ and $j \leftarrow j, \ j \in [n]$.
After identifying the exceptional vertices, we will identify paths in the graph presence of which may result in a small least singular value. Here, we can also take guidance from example \eqref{eq: Jordan}, where the matrix $Z_k$ gives rise to a zig-zag path of the length $2k$ 
whose edges going from right to left are horizontal. Such special paths called chains are introduced and studied in Subsection \ref{subs: chains}. In this subsection we prove that with high probability, the associated graph has no long bad (self-balancing) chains, and estimate  the number of short ones.

Subsection \ref{subs: shells} defines a notion of a shell which is crucial in connecting the properties of the matrix to the geometry of the associated graph. Roughly speaking, an $M$-shell $\Am=(\Ael_\ell)_{\ell=0}^{d}$  is a sequence of subsets of right vertices such that each vertex in each layer $\Ael_{\ell+1}$ is reachable from the previous layer $\Ael_{\ell}$ by a path of length $2$ avoiding some set $M$ of  left vertices. Using previously established properties of chains, we prove that the shells possess an expansion property and that the union of the first few layers contains many right vertices which are not exceptional.  These results are used in Section \ref{s: very sparse} to show that with high probability, almost null vectors of the matrix $\widetilde{A}_{n,z}$ cannot be $\lceil c p_n^{-1} \rceil$-sparse.

Section \ref{s: moderately sparse} is devoted to proving that with high probability, almost null vectors cannot be $\lceil cn/\log (p_n n) \rceil$-sparse. The strategy in this section is different and relies on nets instead of graphs. Because of Section \ref{s: very sparse},  we can assume at this point that at least $\lceil c p_n^{-1} \rceil$ coordinates of a vector $x \in S^{n-1}$ we consider are non-negligible. This means that for any row $i \in [n]$,
$\langle\row_i(\widetilde{A}_{n,z}),\bar x\rangle$ is non-negligible with probability bounded away from zero.
A standard tensorization argument yields that the probability that $\|\widetilde{A}_{n,z} x\|_2$ is small is at most $\exp(-cn)$. Yet, as we do not have a good control of $\|\widetilde{A}_{n,z}\|_2$, we cannot combine this with a straightforward $\e$-net argument. Instead, we introduce a new method based on approximating the vector restricted to the set of its small coordinates in the $\ell_{\infty}$-norm and dealing with each product $\langle\row_i(\widetilde{A}_{n,z}),\bar x\rangle$ separately.
At this step, we turn the sparsity of the matrix from a difficulty to an advantage which allows us to disregard the large coordinates of $x$.

Section \ref{s: moderately sparse} yields that to estimate $s_{\min}(\widetilde{A}_{n,z})$, it is enough to bound  $\|(\widetilde{A}_{n,z})x\|_2$ over the set of spread vectors. Such bound is obtained in Section \ref{sec: smallest} using the random normal method of \cite{RV invertibility} and the L\'evy--Kolmogorov--Rogozin--Esseen inequality.

Now, we are passing to the estimates of the intermediate singular values. Compared to the least one, the difficulty here is twofold. First, to derive uniform integrability of the logarithmic potential, the bound has to be significantly more precise. Second, the probability estimate has to be strong enough to allow taking the union bound. In Section \ref{sec: restricted inv}, we relate the bound on the $(n-k+1)$-th singular value to the magnitude of projection of columns of our matrix onto a subspace orthogonal to $n-k$ other columns. To be able to derive a lower bound for these magnitudes, we need to know that the projections  have sufficiently many vectors in their kernels. This should be done simultaneously for many submatrices since we cannot rely on the union bound at this point. To this end, we introduce a special operation --  a compression of the matrix and its associated graph. These compressions are introduced in Subsection \ref{subs: matrix comp} and used in Subsection \ref{subs: shells} and Section \ref{s: very sparse} to derive the required property. After this is done, getting a strong probability bound is based on randomized restricted invertibility. Restricted invertibility is a well-studied topic going back to the classical theorem of Bourgain and Tzafriri \cite{BT}. It is known that this theorem may not hold for a random submatrix with any significant probability. However, in Section \ref{sec: restricted inv} we show that it holds with a non-negligible, albeit exponentially small probability which turns out to be sufficient for our purposes. Restricted invertibility has been used in random matrix context in \cite{Cook smallest singular} and \cite{Nguyen}, but our approach is significantly different.
The combination of compressions and randomized restricted invertibility allows to obtain a good lower bound for all intermediate singular values. This is done in Section \ref{sec: intermediate}.

We derive the uniform integrability and complete the proof of the circular law in Section \ref{sec: proof circular}. To this end, we use the estimate of the least singular value obtained in Section \ref{sec: smallest} as well as that of the intermediate singular values obtained in Section \ref{sec: intermediate}. However, it turns out that we can use the estimate for $s_{n-k}((1/\sqrt{p_n n})A_n- z \Id_n)$ only for $k < \frac{n}{\log^C(p_n n)}$. For larger $k$, we need a tighter bound.
To this end, we use the idea of \cite{Cook circ}  based on the comparison of Stieltjes transforms of our matrix and some reference random matrix having nice properties. This reference random matrix is often chosen to be Gaussian.
However, in our case, the comparison with the Gaussian matrix does not seem to be feasible. Instead, we introduce a new random matrix obtained by replacing relatively small values of $(1/\sqrt{p_n n})A_n- z \Id_n$ by i.i.d. $N(0,1)$ variables.
This requires bounding the Stieltjes transform of 
Gaussian matrices with partially frozen entries. Such bound is obtained in Subsection \ref{sec: shifted}. The uniform integrability is established in Subsection \ref{sec: uniform}. Finally, in Subsection \ref{sec: completion}, we  complete the proof of Theorem \ref{th: main}.

\section{Preliminaries}

Let us start with notation.
The complex conjugate of a complex number $z$ is denoted by $\overline{z}$.
Given a vector $x=(x_1,\dots,x_n)$ in $\C^n$ or $\R^n$, denote by $x^*$ the non-increasing
rearrangement of the vector of absolute values
$(|x_1|,\dots,|x_n|)$. Further, by $\supp(x)$ we denote the support of $x$.
For a real number $a$, by $\lfloor a\rfloor$ we denote the largest integer not exceeding $a$, and
by $\lceil a\rceil$ --- the smallest integer greater or equal to $a$.
Given a finite set $I$, let $|I|$ denote its cardinality.

The standard inner product in $\C^n$ and $\R^n$ is denoted by $\langle \cdot,\cdot\rangle$, and the standard unit vectors
--- by $e_1,e_2,\dots,e_n$.
For a $k\times m$ matrix $B$, let $\col_j(B)$, $j\leq m$ and $\row_i(B)$, $i\leq k$, be its columns and rows, respectively.
By $\|B\|_{HS}$ we denote the Hilbert--Schmidt norm of $B=(b_{ij})$, i.e.\
$\|B\|_{HS}=\sqrt{\sum_{i,j}|b_{ij}|^2}$.

For a random variable $\xi$ (real or complex), define its {\it L\'evy anti-concentration function} by
$$\cf(\xi,t):=\sup\limits_{\tau\in\C}\Prob\big\{|\xi-\tau|\leq t\big\},\quad t\geq0.$$

\medskip

Let $k,m$ be any positive integers. We introduce a collection $\grc_{k,m}$ of directed bipartite graphs
having $k$ left and $m$ right vertices, and with the property that $i\leftarrow j$ {\it only if} $i\to j$
(for any $i\in[k], j\in[m]$).

For a subset $I$ of the right vertices of $G\in\grc_{k,m}$, define in-neighbors of $I$ --- $\inneigh(I)$ ---
as the set of all left vertices of $i$ of $G$ such that
there is an edge emanating from $i$ and landing in $I$.
Similarly, the set of out-neighbors $\outneigh(I)$ is the collection of left vertices $i$ such that
$i\leftarrow j$ for some $j\in I$.
For a one-element set $\{j\}$, we will write $\inneigh(j)$, $\outneigh(j)$ instead of $\inneigh(\{j\})$, $\outneigh(\{j\})$.

Sets of in- and out-neighbors for collections of left vertices of $G$ are defined along the same lines.
In situations where confusion may arise, we specify explicitly if the vertices are left or right, by adding corresponding superscript:
$j^L$ stands for the left vertex $j$ of $G$, and $j^R$ --- the right vertex.
We use the same convention for sets of vertices.

\subsection{Assumptions on distributions and parameters}\label{subs: assump}

As the crucial step in the proof of the main result, we will derive estimates on the smallest and ``smallish''
singular values of shifted random matrices $A-z\,\Id$,
assuming that conditions \eqref{Asmp on p weak}-\eqref{Asmp on A}--\eqref{Asmp on z}
stated below are satisfied.
First, we fix a global parameter $\alpha\geq 1$
and let the dimension $n$ and sparsity parameter $p$ satisfy
\begin{equation}\tag{{\bf A1}}\label{Asmp on p weak}
\mbox{$C_{\alpha}n^{-1}\leq p\leq n^{-7/8}$,}
\end{equation}
where $C_{\alpha}>0$ depends only on $\alpha$ and is assumed to be sufficiently large
(the value of $C_\alpha$ could be computed explicitly but we prefer to reduce the amount of technical details).
We will consider random square matrices $A$ satisfying

\begin{equation}\tag{{\bf A2}}\label{Asmp on A}
\begin{split}
&\mbox{$A$ is $n\times n$, with i.i.d.\ entries $a_{ij}=\delta_{ij}\,\xi_{ij}$, where $\delta_{ij}$  is a Bernoulli random}\\
&\mbox{variable with $\Prob\{\delta_{ij}=1\}=p$;
$\xi_{ij}$ is nowhere zero complex variable with zero}\\
&\mbox{mean and unit variance independent from $\delta_{ij}$ such that
$\cf(\xi_{ij},1/\alpha)\leq 1-1/\alpha$.}
\end{split}
\end{equation}

The complex shift $z\in\C$ will be chosen so that
\begin{equation}\tag{{\bf A3}}\label{Asmp on z}
\mbox{$|z|\leq pn$\quad and\quad $|a_{ij}-z|
\geq 1/\alpha$ almost surely.}
\end{equation}

The assumption that $\xi_{ij}$'s are nowhere zero does not affect our estimates on the singular values
and can be discarded with help of a standard approximation argument.

As was already mentioned in the introduction, a considerable part of the paper is devoted to the study of
the random bipartite graph associated with our random matrix. Let us recall the definition.

\begin{equation}\tag{{\bf B1}}\label{Asmp on G}
\begin{split}
&\mbox{Let $(\delta_{ij})$ and $(\mu_{ij})$ be two collections of jointly independent Bernoulli random}\\
&\mbox{variables where $\Prob\{\delta_{ij}=1\}=p$ and $\Prob\{\mu_{ij}=1\}\geq 1/\alpha$, where $p$ satisfies \eqref{Asmp on p weak}.}\\
&\mbox{Then the directed bipartite graph $\Gr$ with the vertex set $[n]\sqcup[n]$ is defined by}\\
&\mbox{$i\to j$ iff ($\delta_{ij}=1$ or $i=j$) and $i\leftarrow j$ iff ($\delta_{ij}\mu_{ij}=1$ or $i=j$).}
\end{split}
\end{equation}

\subsection{Classical inequalities}

Let us recall the classical Bernstein inequality for sums of Bernoulli variables:
\begin{lemma}[Bernstein's inequality]\label{l: bernsteins}
Let $m$ be any positive integer, and let $\eta_1,\dots,\eta_m$ be i.i.d.\ Bernoulli ($0/1$) random variables
with $\Prob\{\eta_i=1\}=p$ for some $p\in[0,1]$. Then for any $t>0$ we have
$$\Prob\Big\{\sum\limits_{i=1}^m\eta_i\geq pm+t \Big\}\leq \exp\big(-c_{\smallrefer{l: bernsteins}}t^2/(pm+t)\big)$$
for a universal constant $c_{\smallrefer{l: bernsteins}}>0$.
\end{lemma}

The next lemma (with certain variations) is due to L\'evy, Kolmogorov, Rogozin, Esseen:
\begin{lemma}[{L\'evy--Kolmogorov--Rogozin--Esseen, \cite{Kolmogorov, Rogozin, Esseen}}]\label{l: rogozin}
Let $m\in\N$ and let $\xi_1,\xi_2,\dots,\xi_m$ be independent complex random variables.
Then for any $t>0$ we have
$$\cf\Big(\sum\limits_{i=1}^m\xi_i,t\Big)\leq \frac{C_{\smallrefer{l: rogozin}}}{\big(\sum_{i=1}^m (1-\cf(\xi_i,t))\big)^{1/2}},$$
where $C_{\smallrefer{l: rogozin}}>0$ is a universal constant.
\end{lemma}

\subsection{Basic concentration and expansion properties of $A$ and $\Gr$}

In the following elementary statements we summarize some typical properties of
the matrix $A$ and the graph $\Gr$,
specifically, expansion (Proposition~\ref{p: expansion}),
statistics of in- and out-degrees of vertices in $\Gr$ (Proposition~\ref{p: supports}
and Lemma~\ref{l: union of supp}),
magnitude of $\ell_1$--norms of rows and columns of $A$
(Proposition~\ref{p: ell one norm}).
All statements and their proofs are elementary; the proofs are provided for
Reader's convenience.
We refer to \cite{LLTTY JMAA,Cook graphs} for some related results in the setting of random directed $d$--regular graphs,
and to \cite[Section~2.1]{Cook graphs} for the directed Erd\H os--Renyi setting.

\begin{prop}[Expansion in $\Gr$]\label{p: expansion}
For any $\varepsilon\in(0,1]$ there are $C_\varepsilon,c_\varepsilon>0$
depending only on $\varepsilon$ with the following property.
Let $p,n,\Gr$ be as in \eqref{Asmp on G}, and, additionally, assume $pn\geq C_\varepsilon$.
Then for each $k$ in the interval $2\leq k\leq c_\varepsilon/p$,
with probability at least $1-(\frac{n}{k})^{-k}$ we have
$$\big|\inneigh(I)\big|
\geq \sum_{i\in I}\big|\inneigh(i)\big|- \varepsilon pn\,|I|
\quad\mbox{for any set of right vertices $I$ with $|I|=k$.}
$$
In particular, the event
\begin{align*}
\Event_{\smallrefer{p: expansion}}(\varepsilon):=
\Big\{
&\big|\inneigh(I)\big|
\geq \sum_{i\in I}\big|\inneigh(i)\big|- \varepsilon pn\,|I|
\;\mbox{for every set of right vertices $I$,\,$2\leq |I|\leq \frac{c_\varepsilon}{p}$}
\Big\}
\end{align*}
has probability at least $1-1/n$.
\end{prop}
\begin{proof}
Fix any $\varepsilon>0$, $p,n$ and a subset $I$ of $[n]$ with $2\leq |I|\leq e^{-2/\varepsilon}/p$.
Consider random variables
$$\eta_i':=\max(|\{j\in I:\,i^L\to j^R\}|-1,0), \quad\quad i\leq n.$$
Informally, $\eta_i'$ counts {\it non-unique} occurences of the left vertex $i$ among in-neighbors of right vertices from $I$.
Observe that
$$\big|\inneigh(I)\big|=
\sum_{i\in I}\big|\inneigh(i)\big|-\sum_{i=1}^n \eta_i'.
$$
On the other hand, taking into account non-random ``horizontal'' edges of $\Gr$, if we define
$\eta_i:=\max(|\{j\in I:\,\delta_{ij}=1\}|-1,0)$ then every $\eta_i'$ can be estimated as
$\eta_i'\leq \eta_i+1$ for $i\in I$ and $\eta_i'=\eta_i$ for $i\notin I$.
We will use the standard Laplace transform method to estimate probabilities of deviations for $\eta_i$'s.
Clearly, $\Prob\{\eta_i=\ell\}\leq {|I|\choose \ell+1}p^{\ell+1}$, $\ell\in\N$.
Thus, for any number $\lambda>0$ such that $e^\lambda p|I|\leq 1$, we have
$$\Exp\big(e^{\lambda \eta_i}\big)
\leq 1+\sum_{\ell=1}^\infty \big(e^\lambda\big)^{\ell}\,p^{\ell+1}|I|^{\ell+1}((\ell+1)!)^{-1}\leq 1+p|I|,$$
and hence
$$\Prob\Big\{\sum_{i=1}^n \eta_i\geq t\Big\}\leq \frac{\big(1+p|I|\big)^n}{\exp(\lambda t)},\quad\quad t>0.$$
In particular, taking $t:=\frac{\varepsilon}{2} pn|I|$ and $\lambda:=\log\frac{1}{p|I|}$, we get
$$\Prob\Big\{\sum_{i=1}^n \eta_i\geq \frac{\varepsilon}{2} pn|I|\Big\}
\leq \exp\big(pn|I|-\varepsilon\lambda pn|I|/2\big)\leq \exp\big(-\varepsilon\lambda pn|I|/4\big).$$
Taking the union bound over all subsets of cardinality $k$ (for some $2\leq k\leq e^{-4/\varepsilon}/p$), we get
\begin{align*}
\Prob\Big\{&\big|\inneigh(I)\big|\leq
\sum_{i\in I}\big|\inneigh(i)\big|-\varepsilon pnk\;\mbox{for some set of right vertices $I$,\; $|I|=k$}\Big\}\\
&\leq\Prob\Big\{\big|\inneigh(I)\big|\leq
\sum_{i\in I}\big|\inneigh(i)\big|-k-\frac{\varepsilon}{2} pnk\;\mbox{for some $I$,\; $|I|=k$}\Big\}\\
&=
\Prob\Big\{\sum\limits_{i=1}^n\eta_i'\geq k+\frac{\varepsilon}{2} pnk\;\mbox{for some $I$,\; $|I|=k$}\Big\}
\leq\Prob\Big\{\sum\limits_{i=1}^n\eta_i\geq \frac{\varepsilon}{2} pnk\;\mbox{for some $I$,\; $|I|=k$}\Big\}\\
&\leq\bigg(\frac{1}{pk}\bigg)^{-\frac{1}{4}\varepsilon pnk}\bigg(\frac{en}{k}\bigg)^k
\leq \bigg(\frac{n}{k}\bigg)^{-k},
\end{align*}
provided that $\frac{1}{4}\varepsilon pn\geq C\log(pn)\geq C^2$ for a large enough $C=C(\varepsilon)$.
\end{proof}

\begin{prop}[Statistics of in- and out-neighbors]
\label{p: supports}
Let $n\in\N$, $p\in(0,1]$ and $\Gr$ satisfy \eqref{Asmp on G}. Denote
\begin{align*}
\Event_{\smallrefer{p: supports}}:=\Big\{&\big|\big\{i\leq n:\;|\outneigh(i^L)|\geq 2pn+u\big\}\big|
\leq \exp(-c_{\smallrefer{p: supports}}(pn+u))n\mbox{ and}\\
&\big|\big\{j\leq n:\;|\inneigh(j^R)|\geq 2pn+u\big\}\big|\leq \exp(-c_{\smallrefer{p: supports}}(pn+u))n\;\;\forall\;u=0,1,\dots\Big\}.
\end{align*}
Then $\Prob(\Event_{\smallrefer{p: supports}})\geq 1-e^{-c_{\smallrefer{p: supports}}pn},$
where $c_{\smallrefer{p: supports}}>0$ is a universal constant.
\end{prop}
\begin{proof}
Applying Bernstein's inequality (Lemma~\ref{l: bernsteins}), together with the definition of $\Gr$, we get for any $i\leq n$:
\begin{align*}
\Prob\big\{|\outneigh(i^L)|\geq 2pn+t\big\}\leq \exp(-c\,(pn+t)),\quad\quad t>0,
\end{align*}
for a universal constant $c>0$. Hence, by Markov's inequality,
$$\Prob\big\{|\{i\leq n:\;|\outneigh(i^L)|\geq 2pn+t\}|\geq \exp(-c\,(pn+t)/2)n\big\}\leq \exp(-c\,(pn+t)/2).$$
Similar argument is carried out for $\inneigh(j^R)$, $j\leq n$.
It remains to take the union of respective events over all $t=0,1,2,\dots$.
\end{proof}

\begin{lemma}\label{l: union of supp}
Let $n,p,\Gr$ and event $\Event_{\smallrefer{p: supports}}$ be as in Proposition~\ref{p: supports},
and fix any subset $M$ of $[n]$. Then, conditioned on $\Event_{\smallrefer{p: supports}}$, we have
\begin{align*}
\big|\outneigh(M^L)\big|
&\leq \sum\limits_{i\in M^L}\big|\outneigh(i)\big|
\leq C_{\smallrefer{l: union of supp}}\Big(pn+\log\frac{n}{|M|}\Big)\,|M|,
\mbox{ and }\\
\big|\inneigh(M^R)\big|
&\leq \sum\limits_{j\in M^R}\big|\inneigh(j)\big|
\leq C_{\smallrefer{l: union of supp}}\Big(pn+\log\frac{n}{|M|}\Big)\,|M|,
\end{align*}
where $C_{\smallrefer{l: union of supp}}>0$ is a universal constant.
\end{lemma}
\begin{proof}
Set
$$w:=\max\Big(0,\Big\lceil \frac{1}{c_{\smallrefer{p: supports}}}\log\frac{n}{|M|}-pn\Big\rceil\Big).$$
It is not difficult to see from the definition of $\Event_{\smallrefer{p: supports}}$
that
$$|\{i\leq n:\,|\outneigh(i^L)|\geq 2pn+w\}|\leq |M|,$$
and that
$$\sum\limits_{i\leq n:\,|\outneigh(i^L)|\geq 2pn+w}
\big|\outneigh(i^L)\big|
\leq C (2pn+w)e^{-c_{\smallrefer{p: supports}}(pn+w)}n.$$
Similar estimates hold for $M^R$.
The result follows.
\end{proof}

\begin{prop}\label{p: ell one norm}
Let $n,p,A$ satisfy assumptions \eqref{Asmp on p weak}--\eqref{Asmp on A}.
Define
\begin{align*}
\Event_{\smallrefer{p: ell one norm}}:=
\Big\{&\big|\big\{i\leq n:\;\|\row_i(A)\|_1\geq r\,pn\big\}\big|\leq n/r^{0.9}\mbox{ for all $r\geq pn$, and}\\
&\big|\big\{i\leq n:\;\|\col_i(A)\|_1\geq r\,pn\big\}\big|\leq n/r^{0.9}\mbox{ for all $r\geq pn$}\Big\}.
\end{align*}
Then $\Prob(\Event_{\smallrefer{p: ell one norm}})\geq 1-(pn)^{-c_{\smallrefer{p: ell one norm}}}$,
for a universal constant $c_{\smallrefer{p: ell one norm}}>0$.
\end{prop}
\begin{proof}
By the assumption on the distribution of the matrix entries,
we have $\Exp\|\row_i(A)\|_1=\Exp\|\col_i(A)\|_1\leq pn$, so it remains to apply Markov's inequality.
\end{proof}

\section{Combinatorial structure of the associated random graph $\Gr$}\label{s: graph section}

\subsection{Vertex types: definition and basic properties}\label{subs: types}

Let $k,m$ be large integers, and fix any parameter $K>0$. Let $G$ be a graph in $\grc_{k,m}$.
We will inductively introduce a classification of the right vertices of $G$ as follows.
For an index $j\leq m$, we will say that the vertex $j$ is {\it of type $(K,1)$}
if $|\outneigh(j)|\leq K$.
Denote by $T_{K,1}(G)\subset[m]$ the subset of vertices of type $(K,1)$.
Further, assume that $\ell\geq 2$ and that types $(K,1), (K,2), \dots, (K,\ell-1)$ have been identified.
Take $j\leq m$ which is not any of the types $(K,1), (K,2), \dots, (K,\ell-1)$.
Then we say that the $j$ is {\it of type $(K,\ell)$} if
$$\big|\outneigh(j)\setminus \inneigh(T_{K,1}(G)\cup\dots\cup T_{K,\ell-1}(G))\big|\leq K.$$
The set of all vertices of type $(K,\ell)$ is denoted by $T_{K,\ell}(G)$.
A vertex $j$ is {\it of type $(K,\infty)$} (of {\it infinite} type) if it is not of any types $(K,\ell)$, $\ell\in\N$.


A key point of our argument consists in establishing a correspondence between vertex types
of a graph in $\grc_{k,m}$ and its subgraphs.
Given a graph $G\in\grc_{k,m}$ and a subset $I\subset[m]$, denote by $\subgr{G}{I}$
the subgraph of $G$ obtained by removing the right vertices in $I$.
It will be convenient for us to assume that the right vertex set of $\subgr{G}{I}$
is indexed over $[m]\setminus I$ (so that we get a direct correspondence with vertices of $G$).
\begin{lemma}\label{l: hereditary prop}
Let $G\in\grc_{k,m}$, let $I\subset[m]$ be a subset, and let $\subgr{G}{I}$ be defined as before.
Then for any $K>0$ and $\ell\geq 1$ we have
$$T_{K,\ell}(\subgr{G}{I})\subset \bigcup_{h\leq \ell}T_{K,h}(G).$$
\end{lemma}
\begin{proof}
We will prove the statement by induction. The case $\ell=1$ is obvious; in fact $T_{K,1}(\subgr{G}{I})=T_{K,1}(G)\setminus I$.
Now, assume that $\ell\geq 2$ and that the statement has been verified for $\ell-1$.
Take any $j\in T_{K,\ell}(\subgr{G}{I})$. By definition, we have
$$
\big|\outneigh(j)\setminus \inneigh(T_{K,1}(\subgr{G}{I})\cup\dots\cup T_{K,\ell-1}(\subgr{G}{I}))\big|\leq K,$$
and therefore
$$\big|\outneigh(j)\setminus \inneigh(T_{K,1}(G)\cup\dots\cup T_{K,\ell-1}(G))\big|\leq K.$$
This last assertion immediately implies that $j\in \bigcup_{h\leq \ell}T_{K,h}(G)$.
\end{proof}

Note that the last lemma implies $T_{K,\infty}(\subgr{G}{I})\supset T_{K,\infty}(G)\setminus I$.
The opposite inclusion does not hold in general even in ``approximate'' sense. For example, it is not difficult to construct
a graph $G\in\grc_{k,m}$ and a subset $I\subset[m]$ of cardinality, say, $m/2$, such that $T_{K,\infty}(\subgr{G}{I})=[m]\setminus I$
while $T_{K,\infty}(G)\setminus I=\emptyset$.
Nevertheless, it turns out that under some assumptions (which hold with high probability in our random setting),
a kind of reverse inclusion can be observed. 

%
%
%
%
%

\begin{lemma}\label{l: local to global}
Let $G\in\grc_{k,m}$, and let $I\subset[m]$ be a set.
Assume that for some $K>0$ we have
$$\big|\outneigh(j)\,\cap\,\inneigh(I)\big|\leq K/2$$
for all $j\in [m]\setminus I$.
Then
$$T_{K,\infty}(\subgr{G}{I})\subset T_{K/2,\infty}(G).$$
\end{lemma}
\begin{proof}
Note that it is sufficient to show that for any $\ell\geq 1$ we have
$$T_{K/2,\ell}(G)\,\setminus I\subset \bigcup_{h\leq \ell}T_{K,h}(\subgr{G}{I}).$$
We will verify the statement by induction.
The case $\ell=1$ is obvious. Now, fix $\ell\geq 2$ and assume the assertion is true for $1,2,\dots,\ell-1$.
Pick any $j\in T_{K/2,\ell}(G)\setminus I$.
Then, by the definition,
$$\Big|\outneigh(j)\,\setminus \inneigh\Big(\bigcup_{h\leq \ell-1}T_{K/2,h}(G)\Big)\Big|\leq K/2.$$
By the assumptions of the lemma, we have
$$\big|\outneigh(j)\,\cap\,\inneigh(I)\big|\leq K/2,$$
and so
$$\Big|\outneigh(j)\,\setminus \inneigh\Big(\bigcup_{h\leq \ell-1}T_{K/2,h}(G)\setminus I\Big)\Big|\leq K.$$
By the induction hypothesis, we have
$$\inneigh\Big(\bigcup_{h\leq \ell-1}T_{K/2,h}(G)\setminus I\Big)\subset
\inneigh\Big(\bigcup_{h\leq \ell-1}T_{K,h}(\subgr{G}{I})\Big).$$
That, together with the above relation, implies
$$j\in \inneigh\Big(\bigcup_{h\leq \ell}T_{K,h}(\subgr{G}{I})\Big).$$
The result follows.
\end{proof}

\subsection{Cardinality of the infinite type in random setting} \label{subs: cardinality}

The random graphs we consider in this section will be subgraphs of $\Gr$ introduced at the beginning.
Fix a positive integer $m$, a positive real $p$, and let $\delta_{ij}$, $\mu_{ij}$ ($(i,j)\in[m]\times[m]$) be jointly independent Bernoulli variables
with $\Prob\{\delta_{ij}=1\}=p$ and $\Prob\{\mu_{ij}=1\}\geq 1/\alpha$.
Consider a random directed bipartite graph $\Gr'$ with the vertex set $[m]\sqcup[m]$,
such that
\begin{itemize}
\item $i\to j$ iff $\delta_{ij}=1$ or $i=j$;

\item $i\leftarrow j$ iff $\delta_{ij}\mu_{ij}=1$ or $i=j$.
\end{itemize}

Additionally, let $K_0$ be a parameter such that
$$\frac{pm}{2\alpha}\leq K_0\leq \frac{2pm}{3\alpha}.$$

\noindent The purpose of this subsection is to prove that {\it typically the cardinality of $T_{K_0,\infty}(G')$ is very large ---
almost $m$.}

It will be convenient for us to define a filtration of sigma-algebras $\mathcal F_k$, $k\in\N$,
where for each natural $k$, $\mathcal F_k$ is generated by (random) sets $T_{K_0,\ell}(G')$, $\ell\leq k$
and by sets of in-neighbors $\inneigh(j)$, $j\in T_{K_0,1}(G')\cup\dots\cup T_{K_0,k}(G')$.

Everywhere in this subsection, by $\inneigh(\cdot),\outneigh(\cdot)$ we understand corresponding sets of
in- and out-neighbors for the graph $\Gr'$.
Also, we use shorter notation $T_{K_0,g}$ for types of vertices $T_{K_0,g}(\Gr')$.

\begin{lemma}\label{l: aux q3409thg}
Let $\Gr'$, $K_0$, and filtration $\mathcal F_k$, $k\in\N$ be as above.
Then for any $d\geq 2$ we have
$$\Exp\bigg(\frac{\big|T_{K_0,d}\big|\,\indicator_d}
{\big|\inneigh(T_{K_0,d-1})
\setminus \inneigh\big(\bigcup_{g\leq d-2}T_{K_0,g}\big)\big|}\;\Big|\;\mathcal F_{d-1}\bigg)
\leq e^{-c_{\smallrefer{l: aux q3409thg}}pm},$$
where $\indicator_d$ is the indicator of the event
$\big\{\big|\inneigh\big(\bigcup_{\ell\leq d-1}T_{K_0,\ell}\big)\big|\leq m/4\big\}$,
and $c_{\smallrefer{l: aux q3409thg}}>0$ depends only on $\alpha$.
\end{lemma}

Let $j$ be a vertex of $\Gr'$ in the complement of $\bigcup_{\ell\leq d-1}T_{K_0,\ell}$.
Assume for a moment that the set of in- and -out neighbors of $j$ were independent of the set
$\bigcup_{\ell\leq d-1}T_{K_0,\ell}$ and their in-neighbors. Using the assumption that
$\big|\inneigh\big(\bigcup_{\ell\leq d-1}T_{K_0,\ell}\big)\big|\leq m/4$
(which defines the indicator $\indicator_d$ above), it would be easy to obtain a bound for the probability
$\Prob\big\{j\in T_{K_0,d}\big\}$ for each fixed $j$, and then sum up to bound the expectation of $|T_{K_0,d}|$.
In reality, the set of neighbors of $j$ depends on types $T_{K_0,\ell}$, $\ell\leq d-1$.
However, this dependence is almost negligible, and using a conditioning argument,
we will be able to show that a similar estimate for $|T_{K_0,d}|$ still holds.

\begin{proof}[Proof of Lemma~\ref{l: aux q3409thg}]
Fix a partition $L=(L_j)_{j=1}^d$ of $[m]$ and a collection of
subsets $M_{in}(h),$ $M_{out}(h)\subset[m]$, $h \in  L_1\cup\dots\cup L_{d-1}$,
such that
\begin{itemize}
\item $M_{out}(h)\subset M_{in}(h)$ for all $h$;

\item $|M_{out}(h)|\leq K_0$ whenever $h\in L_1$;

\item for any $2\leq \ell\leq d-1$
and $h\in L_\ell$, we have
$$\Big|M_{out}(h)\,
\setminus \bigcup_{r\in L_g,\,g\leq \ell-2}M_{in}(r)\Big|> K_0$$
and
$$\Big|M_{out}(h)\,\setminus \bigcup_{r\in L_g,\,g\leq \ell-1}M_{in}(r)\Big|\leq K_0;$$

\item $\big|\bigcup_{r\in L_g,\,g\leq d-1}M_{in}(r)\big|\leq m/4$.

\end{itemize}
The conditions on the sets $M_{in}(h),M_{out}(h)$ are designed so that the sets $L_g$, $g\leq d-1$, would play the role of
types $T_{K_0,g}$ in our random graph.
Further, for any $u\in L_d$ define event
\begin{align*}
\Event'_u:=\Big\{&\inneigh(h)=M_{in}(h)\mbox{ and }\outneigh(h)=M_{out}(h)\;\mbox{for all }h\in L_1\cup\dots\cup L_{d-1}, \mbox{and}\\
&\Big|\outneigh(j)\,
\setminus \bigcup_{r\in L_g,\,g\leq d-2}M_{in}(r)\Big|> K_0\mbox{ for all }j\in L_d\setminus\{u\}\Big\},
\end{align*}
and take
\begin{align*}
\widetilde\Event:=
\Big\{&\inneigh(h)=M_{in}(h)\mbox{ and }\outneigh(h)=M_{out}(h)\;\mbox{for all }h\in L_1\cup\dots\cup L_{d-1}, \mbox{and}\\
&\Big|\outneigh(j)\,
\setminus \bigcup_{r\in L_g,\,g\leq d-2}M_{in}(r)\Big|> K_0\mbox{ for all }j\in L_d\Big\},
\end{align*}
Note that if the event $\widetilde \Event$ occurs then $T_{K_0,g}=L_g$, $g\leq d-1$.
We are interested in bounding the probability that $u$ is of type $(K_0,d)$
on the $\mathcal F_{d-1}$--measurable event $\widetilde \Event$.
However, a direct computation is difficult due to dependencies, and
for that reason we have introduced the auxiliary event $\Event_u'$,
on which the sets of in- and out-neighbors of $u$ are defined by independent Bernoulli selectors.
We will bound the probability that $u$ is of type $(K_0,d)$ on the event $\Event_u'$, and
then compare the event $\Event_u'$ with $\widetilde \Event$.

\medskip

It is easy to see that $\widetilde \Event\subset \Event'_u$ for all $u\in L_d$.
Observe that, conditioned on $\Event_u'$, the event
$\{u\mbox{ is of type }(K_0,d)\}$ implies that
\begin{itemize}

\item[(a)] the set $\outneigh(u)$
has a non-empty intersection with
$M_{in}(h)\setminus \bigcup_{r\in L_g,\,g\leq d-2}M_{in}(r)$
for some $h\in L_{d-1}$,
and

\item[(b)] the intersection of $\outneigh(u)$
with the set $[m]\setminus \bigcup_{r\in L_g,\,g\leq d-1}M_{in}(r)$
has cardinality at most $K_0$.

\end{itemize}
In the case
$u\in \bigcup_{r\in L_{d-1}}M_{in}(r)\setminus \bigcup_{h\in L_g,\,g\leq d-2}M_{in}(h)$
the condition (a) is satisfied automatically
since $u\leftarrow u$ is in the edge set of the graph.
In this case, we simply estimate the probability
$\Prob\{u\mbox{ is of type }(K_0,d)\;|\;\Event_u'\}$, using condition (b), by the probability
$$\Prob\Big\{\Big|\outneigh(u)\setminus \bigcup_{r\in L_g,\,g\leq d-1}M_{in}(r)\Big|\leq K_0\Big\},$$
which in turn can be estimated by $\exp(-cpm)$, for some $c=c(\alpha)>0$,
using our assumption on the cardinality of $\bigcup_{r\in L_g,\,g\leq d-1}M_{in}(r)$.
In the case
$$u\notin \bigcup_{r\in L_{d-1}}M_{in}(r)\setminus \bigcup_{h\in L_g,\,g\leq d-2}M_{in}(h)$$
we use both (a) and (b) to get an upper bound
\begin{align*}
\Prob\big\{u\mbox{ is of type }(K_0,d)\;|\;\Event'_u\big\}
&\leq \Prob\big\{\mbox{(a) holds}\;|\;\Event'_u\big\}\cdot
\Prob\big\{\mbox{(b) holds}\;|\;\Event'_u\big\}
\\
&\leq p
\Big|\bigcup_{r\in L_{d-1}}M_{in}(r)\setminus \bigcup_{h\in L_g,\,g\leq d-2}M_{in}(h)\Big|
\cdot\exp(-cpm).
\end{align*}

Next, again by the assumption on the cardinality of the union of $M_{in}(r)$ (with $r\in L_1\cup\dots\cup L_{d-2}$),
we have
$$\Prob(\widetilde \Event\,|\,\Event'_u)
=\Prob\Big\{\Big|\outneigh(u)\,
\setminus \bigcup_{r\in L_g,\,g\leq d-2}M_{in}(r)\Big|> K_0\;\Big|\Big\}
> \frac{1}{2}\quad \mbox{for all }\quad u\in L_d.$$
This, together with the above, gives for every $u\in L_d$:
\begin{itemize}
\item $\Prob\{u\mbox{ is of type }(K_0,d)\;|\;\widetilde\Event\}\leq \exp(-c'pm)$,
if $u$ belongs to the set $\bigcup_{r\in L_{d-1}}M_{in}(r)\setminus \bigcup_{h\in L_g,\,g\leq d-2}M_{in}(h)$;

\item $\Prob\{u\mbox{ is of type }(K_0,d)\;|\;\widetilde\Event\}\leq p\,
\big|\bigcup_{r\in L_{d-1}}M_{in}(r)\setminus \bigcup_{h\in L_g,\,g\leq d-2}M_{in}(h)\big|\,\exp(-c'pm)$,
if $u\notin \bigcup_{r\in L_{d-1}}M_{in}(r)\setminus \bigcup_{h\in L_g,\,g\leq d-2}M_{in}(h)$.

\end{itemize}

Summing up over all $u\in L_d$, we get
\begin{align*}
\Exp&\big(\big|\big\{u\in L_d:\,u\mbox{ is of type }(K_0,d)\big\}\big|\;|\;\widetilde\Event\big)\\
&\leq e^{-c''pm}\,\Big|\bigcup_{r\in L_{d-1}}M_{in}(r)\setminus \bigcup_{h\in L_g,\,g\leq d-2}M_{in}(h)\Big|
\end{align*}
for some $c''>0$ depending only on $\alpha$.
Thus,
$$\Exp\bigg(\frac{\big|T_{K_0,d}\big|}
{|\inneigh(T_{K_0,d-1})
\setminus \inneigh(\bigcup_{g\leq d-2}T_{K_0,g})|}\;\Big|\;\widetilde\Event\bigg)
\leq e^{-c''pm}.$$
Moreover, the event $\widetilde\Event$ is an atom of the sigma-algebra $\mathcal F_{d-1}$.
The only restriction on the sets $T_{K_0,\ell}$ that we employ is that
$$\Big|\inneigh\Big(\bigcup_{\ell\leq d-1}T_{K_0,\ell}\Big)\Big|\leq m/4$$
(conditioned on $\widetilde\Event$).
Hence, we get from the above
$$\Exp\bigg(\frac{\big|T_{K_0,d}\big|\,\indicator_d}
{|\inneigh(T_{K_0,d-1})\setminus
\inneigh(\bigcup_{\ell\leq d-2}T_{K_0,\ell})|}\;\Big|\;\mathcal F_{d-1}\bigg)
\leq e^{-c''pm},$$
with $\indicator_d$ defined earlier.
\end{proof}

The next proposition asserts that the expected cardinality of the set of in-neighbors of the right vertices of $\Gr'$ of finite
types is much smaller than $m$. Thus, with large probability the majority of left vertices are connected only to right vertices of
the infinite type. When recast in terms of the random matrix $A-z\,\Id$ (see Subsection~\ref{subs: for matrices}),
the result says that with probability close to one
most of the rows of $A-z\,\Id$ are supported on the infinite column type.

\begin{prop}\label{p: support of finite classes}
Let $m,p$ and $\Gr'$ be as above. Then
$$
\Exp\Big(\Big|\inneigh\Big(\bigcup_{g\geq 1}T_{K_0,g}\Big)\Big|\Big)
\leq e^{-c_{\smallrefer{p: support of finite classes}} \,pm}\,m
$$
for $c_{\smallrefer{p: support of finite classes}}>0$ depending only on $\alpha$.
\end{prop}
\begin{proof}
Let $\indicator_\ell$ be as in Lemma~\ref{l: aux q3409thg}, and for each $\ell\in\N$ define
$\widetilde\indicator_\ell$ to be the indicator of the event
$$
\Event_\ell:=\Big\{\Big|\inneigh\Big(\bigcup_{h\leq \ell-1}T_{K_0,h}\Big)\Big|\leq \frac{m}{4}(1-2^{-\ell})\Big\}.
$$
Observe that $\widetilde\indicator_\ell\leq \indicator_\ell$ for all $\ell\geq 2$;
we will postulate that $\widetilde \indicator_1=1$ everywhere.
Importantly, $\widetilde\indicator_\ell$ is measurable with respect to the sigma-algebra $\mathcal F_{\ell-1}$ ($\ell\geq 1$).

We will prove the statement in two steps.
First, we show that, conditioned on the intersection $\bigcap\limits_{j=1}^\ell \Event_j$,
the cardinality of the set $\inneigh\big(T_{K_0,\ell}\big)\setminus
\inneigh\big(\bigcup_{g\leq \ell-1}T_{K_0,g}\big)$ is small on average.
Then, we show that the event $\bigcap\limits_{j=1}^\ell \Event_j$ holds with probability close to one.

For any $\ell\geq 2$, we have
\begin{align*}
&\Exp\big(|T_{K_0,\ell}|\widetilde\indicator_1\dots\widetilde \indicator_\ell\big)\\
&=
\Exp\bigg(\Exp\bigg(\frac{|T_{K_0,\ell}|\widetilde\indicator_1\dots\widetilde \indicator_\ell}
{\big|\inneigh(T_{K_0,\ell-1})
\setminus \inneigh\big(\bigcup_{g\leq \ell-2}T_{K_0,g}\big)\big|}\;\Big|\;\mathcal F_{\ell-1}\bigg)
\cdot\Big|\inneigh(T_{K_0,\ell-1})
\setminus \inneigh\Big(\bigcup_{g\leq \ell-2}T_{K_0,g}\Big)\Big|\widetilde\indicator_1\dots\widetilde \indicator_{\ell-1}\bigg)\\
&\leq e^{-c''pm}\,\Exp\Big(\Big|\inneigh(T_{K_0,\ell-1})
\setminus \inneigh\Big(\bigcup_{g\leq \ell-2}T_{K_0,g}\Big)\Big|
\widetilde\indicator_1\dots\widetilde \indicator_{\ell-1}\Big)
\end{align*}
for some $c''>0$ depending on $\alpha$, where at the last step we used Lemma~\ref{l: aux q3409thg}.

Let $W^0$ be any realization of the matrix $W:=(\mu_{ij})$,
$\mathcal A\in \mathcal F_{\ell-1}$
be any atom of the sigma-algebra $\mathcal F_{\ell-1}$ (i.e.\ some realization of sets
$T_{K_0,1},\dots, T_{K_0,\ell-1}$ and respective collections of in- and out-neighbors),
and set $\Event':=\mathcal A\cap\{W=W^0\}$
assuming that the event has a non-zero probability.
Observe that, conditioned on $\Event'$, the variables
$\delta_{ij}$ for $j\notin T_{K_0,1},\dots, T_{K_0,\ell-1}$ and $W^0_{ij}=0$, are mutually independent, and,
moreover, the set $T_{K_0,\ell}$ is completely determined by the values of $\delta_{ij}$ for $(i,j)$ with $W^0_{ij}=1$.
Hence,
\begin{align}
\Exp\big(\big|\big\{&(i,j)\in [m]\times[m]:\,j\in T_{K_0,\ell},\,W^0_{ij}=0,\,\delta_{ij}=1,\,
i\neq j\big\}\big|\;|\;\Event'\big)\nonumber\\
&\leq pm\,\Exp\big(|T_{K_0,\ell}|\;|\;\Event'\big).
\label{eq: aux kmff}
\end{align}
At the same time, by the definition of $T_{K_0,\ell}$, we have (deterministically)
\begin{equation}\label{eq: auxhr}
\Big|\outneigh(T_{K_0,\ell})
\setminus \inneigh\Big(\bigcup_{h\leq\ell-1}T_{K_0,h}\Big)\Big|\leq K_0|T_{K_0,\ell}|.
\end{equation}
Note that the set $\inneigh(T_{K_0,\ell})$ can be viewed as consisting of three parts:
the left vertices $i$ such that $i\to j$ and $i\not\leftarrow j$ for some $j\in T_{K_0,\ell}\setminus\{i\}$;
left vertices $i$ such that $i\to j$ and $i\leftarrow j$ for some $j\in T_{K_0,\ell}\setminus\{i\}$;
and the vertices $i$ with $i\in T_{K_0,\ell}$.
For the first category, we will apply formula \eqref{eq: aux kmff};
for the second --- formula \eqref{eq: auxhr}, and for the third --- the trivial upper bound.
Thus, removing conditioning with respect to $\{W=W^0\}$, we obtain
$$\Exp\Big(\Big|\inneigh(T_{K_0,\ell})\,
\setminus \inneigh\Big(\bigcup_{h\leq\ell-1}T_{K_0,h}\Big)\Big|\;|\;\mathcal F_{\ell-1}\Big)
\leq 3pm\,\Exp\big(|T_{K_0,\ell}|\;|\;\mathcal F_{\ell-1}\big).$$

Together with above estimate of $\Exp\big(|T_{K_0,\ell}|\widetilde\indicator_1\dots\widetilde \indicator_\ell\big)$
and the measurability of $\widetilde\indicator_1,\dots,\widetilde \indicator_\ell$ with respect to $\mathcal F_{\ell-1}$, this yields
\begin{align*}
\Exp\Big(&\Big|\inneigh(T_{K_0,\ell})
\setminus \inneigh\Big(\bigcup_{h\leq\ell-1}T_{K_0,h}\Big)\Big|
\widetilde\indicator_1\dots\widetilde \indicator_{\ell}\Big)\\
&\leq
e^{-\widetilde c \,pm}\,\Exp\Big(\Big|\inneigh(T_{K_0,\ell-1})
\setminus \inneigh\Big(\bigcup_{g\leq \ell-2}T_{K_0,g}\Big)\Big|
\widetilde\indicator_1\dots\widetilde \indicator_{\ell-1}\Big)
\end{align*}
for all $\ell\geq 2$, where $\widetilde c$ may only depend on $\alpha$.
It is an easy consequence of Bernstein--type inequalities that
$$\Exp\big(\big|\inneigh(T_{K_0,1})\big|
\big)\leq e^{-c \,pm}m$$
for a universal constant $c>0$.
Then, applying the previous relation iteratively, we obtain for all $\ell\geq 1$ and $c'=c'(\alpha)>0$:
\begin{equation}\label{eq: aux qping}
\Exp\Big(\Big|\inneigh(T_{K_0,\ell})
\setminus \inneigh\Big(\bigcup_{h\leq\ell-1}T_{K_0,h}\Big)\Big|
\widetilde\indicator_1\dots\widetilde \indicator_{\ell}\Big)\leq e^{-c' \,pm\, \ell}\,m.
\end{equation}
This completes the first part of the proof.
It remains to show that $\widetilde\indicator_1\dots\widetilde \indicator_{\ell}$
is equal to one with high probability.
Conditioned on the event $\{\widetilde\indicator_1\dots\widetilde \indicator_{\ell-1}=1\}$, we have
$$\Big|\inneigh\Big(\bigcup_{h\leq \ell-1}T_{K_0,h}\Big)\Big|
\leq \Big|\inneigh(T_{K_0,\ell-1})
\setminus \inneigh\Big(\bigcup_{h\leq \ell-2}T_{K_0,h}\Big)\Big|+ \frac{m}{4}(1-2^{-\ell+1}).$$
Hence, for any $\ell\geq 2$:
\begin{align*}
\Prob&\big\{\widetilde\indicator_1\dots\widetilde \indicator_{\ell-1}\cdot(1-\widetilde \indicator_\ell)=1\big\}\\
&=\Prob\big\{\widetilde \indicator_\ell=0\;|\;\widetilde\indicator_1\dots\widetilde \indicator_{\ell-1}=1\big\}\,
\Prob\big\{\widetilde\indicator_1\dots\widetilde \indicator_{\ell-1}=1\big\}\\
&\leq \Prob\Big\{\Big|\inneigh(T_{K_0,\ell-1})
\setminus \inneigh\Big(\bigcup_{g\leq \ell-2}T_{K_0,g}\Big)\Big|> \frac{m}{4}2^{-\ell}
\;\big|\;\widetilde\indicator_1\dots\widetilde \indicator_{\ell-1}=1\Big\}
\cdot\Prob\big\{\widetilde\indicator_1\dots\widetilde \indicator_{\ell-1}=1\big\}.
\end{align*}
Using \eqref{eq: aux qping} and  applying Markov's inequality, we obtain
\[
\Prob\Big\{\Big|\inneigh(T_{K_0,\ell-1})
\setminus \inneigh\Big(\bigcup_{g\leq \ell-2}T_{K_0,g}\Big)\Big|> \frac{m}{4}2^{-\ell}
\;\big|\;\widetilde\indicator_1\dots\widetilde \indicator_{\ell-1}=1\Big\}
\leq e^{-c'\,pm\,(\ell-1)}2^{\ell+2}.
\]
Thus,
$$\Prob\big\{\widetilde\indicator_1\dots\widetilde \indicator_{\ell-1}\cdot(1-\widetilde \indicator_\ell)=1\big\}
\leq e^{-c'\,pm\,(\ell-1)}2^{\ell+2}\,\Prob\big\{\widetilde\indicator_1\dots\widetilde \indicator_{\ell-1}=1\big\},\quad \ell\geq 2.$$
Rearranging, we get
$$\Prob\big\{\widetilde\indicator_1\dots\widetilde \indicator_{\ell-1}\widetilde\indicator_\ell=1\big\}\geq
\big(1-e^{-c'\,pm\,(\ell-1)}2^{\ell+2}\big)\Prob\big\{\widetilde\indicator_1\dots\widetilde \indicator_{\ell-1}=1\big\},
\quad\ell\geq 2,$$
and hence
$$\Prob\Big\{\prod_{\ell=1}^\infty\widetilde\indicator_\ell=1\Big\}\geq 1-e^{-c_1\,pm}.$$
It remains to apply the last relation to \eqref{eq: aux qping}:
we have
\begin{align*}
\Exp&\Big(\Big|\inneigh\Big(\bigcup_{g\geq 1}T_{K_0,g}\Big)\Big|\Big)\\
&\leq \Exp\Big(\Big|\inneigh\Big(\bigcup_{g\geq 1}T_{K_0,g}\Big)\Big|\prod_{\ell=1}^\infty\widetilde\indicator_\ell=1\Big)
+m\,\Prob\Big\{\prod_{\ell=1}^\infty\widetilde\indicator_\ell=0\Big\}\\
&\leq
\sum_{\ell=1}^\infty\Exp\Big(\Big|\inneigh(T_{K_0,\ell})
\setminus \inneigh\Big(\bigcup_{h\leq\ell-1}T_{K_0,h}\Big)\Big|
\widetilde\indicator_1\dots\widetilde \indicator_{\ell}\Big)+e^{-c_1\,pm}\,m\\
&\leq e^{-\bar c \,pm}\,m
\end{align*}
for $\bar c=\bar c(\alpha)>0$.
\end{proof}

\subsection{Chains}\label{subs: chains}

Define a subfamily $\grch_{n,n}\subset\grc_{n,n}$ as the collection of graphs having all ``horizontal'' edges,
namely, for any $G\in\grch_{n,n}$ and any $i\in[n]$ we have $i^L\to i^R$ and $i^L\leftarrow i^R$.
Note that the random graph $\Gr$ defined by \eqref{Asmp on G}, belongs to $\grch_{n,n}$
with probability one.
Such graphs are important for us since they correspond to matrices with a non-zero diagonal.

Let $G\in\grch_{n,n}$ and $k\geq 1$.
The left and right vertices of $G$ are indexed by the same set $[n]$.
For this moment, it will be convenient to write $j^{L}$ for the left and $j^{R}$ for the right vertices.
We will say that a sequence $(j_\ell^R)_{\ell=1}^k$ of right vertices of $G$
is {\it a chain of length $k$ for $G$} if
it lies on the path $j^{R}_1\to j^L_1\to j^R_2\to j^L_2\to\dots\to j^L_{k-1}\to j^R_k$,
with $j_{\ell}^R\neq j_{\ell+1}^R$ for $\ell<k$.
In other words, all edges leading to the left vertices are ``horizontal''.
If all $j_\ell^R$'s ($1\leq \ell\leq k$) are distinct, we will call such a chain {\it cycle-free}.
Further, if $j_\ell^R$'s for $1\leq \ell\leq k-1$ are all distinct but $j_k^R=j_{u}^R$ for some $u<k-1$, the chain will be called {\it cyclic}.

To verify that the square matrix $A-z\,\Id$ is non-singular with high probability, the above setting
is all that is needed. However, in the treatment of {\it intermediate} singular values
we will need a more general definition of chains for bipartite graphs with different
sets of left and right vertices.
We will extend the definition in the last part of the section.

The following is an elementary observation:
\begin{lemma}\label{l: chain types}
Let $J=(j_\ell)_{\ell=1}^k$ be a chain for a graph $G\in \grch_{n,n}$. Then one of the following two assertions is true:
either $J$ is cycle-free or there is a number $1\leq k_1\leq k$ such that $(j_\ell)_{\ell=1}^{k_1}$ is cyclic.
\end{lemma}

In what follows, it will be sometimes convenient for us to view a chain $J$ as a set rather
than a sequence. In particular, for any subset of integers $S$, notation $J\setminus S$
should be understood as a set consisting of those elements of $J$ which are not included in $S$.
Further, given a chain $J=(j_\ell)_{\ell=1}^k$ for $G$, let $\subgr{G}{J}$
be the subgraph of $G$ formed by removing right vertices $j_\ell$, $\ell\leq k$
(note that this notation for subgraphs is consistent with the one given in Subsection~\ref{subs: types}).
It will be convenient for us to assume that the right vertices of $\subgr{G}{J}$
are indexed by $[n]\setminus \{j_\ell\}_{\ell=1}^{k}$.

Given $K>0$, we will say that the chain $J$ for a graph $G\in\grch_{n,n}$ is
{\it $K$--self-balancing} if
$j_\ell\notin T_{K,\infty}(G)$ for all $\ell\leq k$ and, moreover,
for any $\ell\leq k$ we have
$$\outneigh(j_\ell)\subset 
\inneigh\Big(\bigcup_{g\geq 1}T_{K,g}(G)\setminus \{j_\ell\}\Big).$$
Note that in the above definition $j_\ell^L\in \outneigh(j_\ell^R)$ since the graph $G\in\grch_{n,n}$ is required to contain ``horizontal'' edges;
and it is possible that $\outneigh(j_\ell^R)$ consists of a single element $j_\ell^L$.

By negation, a chain $J=(j_\ell)_{\ell=1}^{k}$ for $G$ is not $K$--self-balancing if and only if
either (a) $j_\ell\in T_{K,\infty}(G)$ for some $\ell\leq k$ or
(b) $j_1,j_2,\dots,j_k\in [n]\setminus T_{K,\infty}(G)$ and there is
$j_\ell$ and a left vertex $i\in\outneigh(j_\ell)$ such that $\outneigh(i)\setminus\{j_\ell\}\subset T_{K,\infty}(G)$.
We also observe that if a chain $(j_\ell)_{\ell=1}^{k}$ is $K$--self-balancing then $(j_\ell)_{\ell=1}^{h}$
is $K$--self-balancing for any $h\leq \ell$.

The notion of chains plays the central role in our argument.
A connection with the matrix invertibility can be illustrated as follows:
assume that $x$ is a non-zero null vector of $A-z\,\Id$ (where we assume that
the matrix diagonal elements are non-zero), and let $\Gr$ be the corresponding bipartite graph.
Let $i\leq n$ be such that $x_i\neq 0$. Looking at $\row_i(A-z\,\Id)$ and noticing that
$a_{ii}-z\neq 0$, we find $j\neq i$
such that $a_{ij}x_j\neq 0$. This means that $i\in \inneigh(j)$.
Next, looking at the $j$-th row, by the same reason, we find $k\neq j$ such that $a_{jk}x_k\neq 0$, continuing with the construction of the chain.
In fact, for any $k\geq 1$ there exists a chain $J$ for $\Gr$ of length $k$, with all elements in the support of $x$.
A more detailed analysis shows that, conditioned on an event of probability close to one, all such chains
must be self-balancing. At the same time, as we show in Subsection~\ref{subs: self-balancing chains},
with high probability there are no self-balancing cyclic or cycle-free chains of logarithmic length.
This (combined with some additional observations) implies that $A-z\,\Id$ does not have very
sparse null vectors with high probability.
As we are interested in quantitative bounds on the smallest singular value,
this argument needs to be augmented: we have to take into account the magnitudes of the coordinates of $x$,
the distribution of cardinalities of supports of rows of $A-z\,\Id$, {\it statistics} of chains (i.e.\ number of self-balancing/non-self-balancing
cyclic and cycle-free chains of a given length). The lastly mentioned characteristic of the matrix is studied in this section.


Define
$$K_0:=\frac{pn}{2\alpha}.$$
The next lemma can be viewed as a decoupling procedure for the vertex chains.
Specifically, it will be used to replace vertex types of the graph $\Gr$ in the definition of a self-balancing
chain $J$ with vertex types of the subgraph $\subgr{\Gr}{J}$, taking advantage of independence of these
types from edges incident to $J$.
\begin{lemma}\label{l: local to global random}
Let $n$, $p$, $\Gr$ be as in \eqref{Asmp on G}.
Define
\begin{align*}
\Event_{\smallrefer{l: local to global random}}:=
\big\{&\mbox{for every cycle-free/cyclic chain $J$ of length $k\leq\log_{pn}n$}\\
&\mbox{for $\Gr$ we have $T_{K_0,\infty}(\subgr{\Gr}{J})\subset T_{K_0/2,\infty}(\Gr)$}
\big\}.
\end{align*}
Then $\Prob(\Event_{\smallrefer{l: local to global random}})\geq 1-n^{-10}$.
In fact, ``$-10$'' can be replaced with any negative constant.
\end{lemma}
\begin{proof}
It is not difficult to see that it is sufficient to prove the statement for cycle-free chains; corresponding bound for
cyclic chains will follow just by throwing away the last element of the chains.
Fix $k\leq \log_{pn}n$ and denote
\begin{align*}
\Event':=\big\{(1,2,\dots,k)
\mbox{ is a chain for $\Gr$ such }
\mbox{that $T_{K_0,\infty}(\subgr{\Gr}{J})\not\subset T_{K_0/2,\infty}(\Gr)$}
\big\}.
\end{align*}
We will compute the probability $\Prob(\Event')$.
By Lemma~\ref{l: local to global},
$\Event'$ does not occur, whenever for any $j\in [n]\setminus [k]$ we have
$\big|\outneigh(j)\,\cap\,
\inneigh([k])\big|\leq K_0/2$.
Thus,
\begin{align*}
\Prob(\Event')&\leq
\Prob\Big\{\ell\to\ell+1\mbox{ for all }\ell\leq k-1\mbox{ and there is }j\in [n]\setminus [k]\\
&\hspace{1cm}\mbox{with }\big|\outneigh(j)\,\cap\,
\inneigh([k])\big|> K_0/2
\Big\}\\
&\leq
\Prob\Big\{\ell\to\ell+1\mbox{ for all }\ell\leq k-1\mbox{ and there is }j\in [n]\setminus [k]\\
&\hspace{1cm}\mbox{with }\big|\inneigh(j)\,\cap\,[k]^L\big|> K_0/4\Big\}\\
&+
\Prob\Big\{\ell\to\ell+1\mbox{ for all }\ell\leq k-1\mbox{ and there is }j\in [n]\setminus [k]\\
&\hspace{1cm}\mbox{with }\big|\inneigh(j)\,\cap\,\inneigh([k]^R)\setminus[k]^L\big|> K_0/4
\Big\},
\end{align*}
where in the last inequality, we used that $\outneigh(j)\subset\inneigh(j)$ for a right vertex $j$.
Hence, we get
\begin{align*}
\Prob(\Event')
&\leq  p^{k-1}\cdot n\cdot p^{\lceil K_0/4\rceil }{k\choose \lceil K_0/4\rceil}
+ p^{k-1}\cdot n\cdot \Prob\Big\{\big|\inneigh(j)\,\cap\,\inneigh([k]^R)\setminus[k]^L\big|> K_0/4
\Big\}\\
&\leq p^{k-1}n \big(24k/n\big)^{K_0/4}
+ p^{k-1}n\cdot (p^2k)^{\lceil K_0/4-1\rceil}{n\choose \lceil K_0/4-1\rceil}.
\end{align*}
Bounding the first term is straightforward.
To bound the second term, we use the assumptions on $p$ and $k$, which imply that $p^2 k n\leq n^{-1/4}$.
This yields $\Prob(\Event')\leq p^k n^{-100}$.
It remains to observe that
\begin{align*}
\Prob\big\{&\mbox{there is a cycle-free 
chain $J$ of length $k$}\\
&\mbox{for $\Gr$ with }T_{K_0,\infty}(\subgr{\Gr}{J})\not\subset T_{K_0/2,\infty}(\Gr)\big\}\leq n^k \Prob(\Event'),
\end{align*}
and apply the union bound over all $k\leq \log_{pn}n$.
\end{proof}

\subsubsection{Self-balancing chains}\label{subs: self-balancing chains}

In this subsection, we study statistics of self-balancing chains. Further, in subsection~\ref{subs: matrix comp}
we will transfer the results to a generalized setting of $\phi$--chains.

\begin{lemma}[Number of self-balancing cycle-free chains]\label{l: cycle-free chains}
Let $n,p,\Gr$ be as in \eqref{Asmp on G}.
Let $1\leq k\leq \log_{pn}n$, and let $\IsoChains_k$ be the set of all $(K_0/2)$--self-balancing cycle-free chains of length $k$ for $\Gr$.
Denote by $\Event_0$ the event
$\{|\inneigh\big(\bigcup_{g\geq 1}T_{K_0,g}(\Gr)\big)|\leq n\,e^{-c_{\smallrefer{p: support of finite classes}}pn/2}\}$,
where $c_{\smallrefer{p: support of finite classes}}=c_{\smallrefer{p: support of finite classes}}(\alpha)>0$
is taken from Proposition~\ref{p: support of finite classes}.
Then
$$\Exp\big(|\IsoChains_k|\;|\;\Event_0\cap \Event_{\smallrefer{l: local to global random}}
\big)\leq 8(pk)^{k-1}(k-1)+n\,e^{-c_{\smallrefer{l: cycle-free chains}}pnk}$$
for some $c_{\smallrefer{l: cycle-free chains}}=c_{\smallrefer{l: cycle-free chains}}(\alpha)>0$.
\end{lemma}
\begin{proof}
Let us estimate conditional probability of the event
$$\Event:=\big\{(1,2,\dots,k)\mbox{ is a $(K_0/2)$--self-balancing chain for $\Gr$}\big\}$$
given $\Event_0\cap \Event_{\smallrefer{l: local to global random}}$.
Due to lack of independence, a direct estimate can be complicated.
To overcome this problem, we introduce auxiliary events $\Event(Q)$.
For any $Q\subset[n]$ set
$$\Event(Q):=\Big\{\Big|\inneigh\Big(\bigcup_{g\geq 1}T_{K_0,g}(\subgr{\Gr}{Q})\Big)\Big|\leq n\,
e^{-c_{\smallrefer{p: support of finite classes}}pn/2}\Big\}.$$
Observe that, by Lemma~\ref{l: hereditary prop},
we have $\Event_0\subset\Event(Q)$ for every $Q\subset[n]$.

We will bound $\Prob(\Event\,|\,\Event([k])\cap \Event_{\smallrefer{l: local to global random}})$ first.
We have
\[
\Prob\big(\Event\;|\;\Event([k])\cap \Event_{\smallrefer{l: local to global random}}\big)
=
\frac{\Prob(\Event([k]))}{\Prob(\Event([k])\cap \Event_{\smallrefer{l: local to global random}})}\cdot
\Prob\big(\Event\cap \Event_{\smallrefer{l: local to global random}}\;|\;\Event([k])\big)
\leq 2\,\Prob\big(\Event\cap \Event_{\smallrefer{l: local to global random}}\;|\;\Event([k])\big),
\]
where we used a simple estimate $\Prob(\Event([k])\cap \Event_{\smallrefer{l: local to global random}})\geq 1/2$,
which follows from the inclusion $\Event_0\cap \Event_{\smallrefer{l: local to global random}}\subset
\Event([k])\cap \Event_{\smallrefer{l: local to global random}}$
and the fact that both $\Event_0$ and $\Event_{\smallrefer{l: local to global random}}$
have probabilities close to one (see Lemma~\ref{l: local to global random} and Proposition~\ref{p: support of finite classes}).

By the definition of a self-balancing chain, on the event $\Event\cap \Event_{\smallrefer{l: local to global random}}$ we have
$i\in\inneigh\big(([k]\setminus\{j\})\cup(\bigcup\limits_{g\geq 1}T_{K_0/2,g}(\Gr)\setminus[k])\big)$
for all pairs $(i,j)\in\{k+1,\dots,n\}\times [k]$ with $i\leftarrow j$.
Moreover, on this event we have $\bigcup\limits_{g\geq 1}T_{K_0/2,g}(\Gr)\setminus[k]
\subset \bigcup\limits_{g\geq 1}T_{K_0,g}(\subgr{\Gr}{[k]})$.
Thus,
\begin{align*}
\Prob\big(\Event\cap \Event_{\smallrefer{l: local to global random}}\;|\;\Event([k])\big)\leq
\Prob\Big\{&i\to i+1\mbox{ for all }i\leq k-1\mbox{ and}\\
&\forall \;(i,j)\in \{k+1,\dots,n\}\times [k]
\mbox{ such that }i\leftarrow j\mbox{ we have}\\
&i\in \inneigh\big(\{h\leq k,\,h\neq j\}\big)\;\cup\;
\inneigh\Big(\bigcup_{g\geq 1}T_{K_0,g}(\subgr{\Gr}{[k]})\Big)\;\big|\;\Event([k])
\Big\}.
\end{align*}
Therefore,
\begin{align*}
\Prob\big(\Event\;|\;\Event([k])\cap \Event_{\smallrefer{l: local to global random}}\big)
\leq 2\Prob\big\{&i\to i+1\mbox{ for all }i\leq k-1\;\big|\;\Event([k])\big\}\cdot\\
\Prob\Big\{&\forall \;(i,j)\in \{k+1,\dots,n\}\times [k]
\mbox{ such that }i\leftarrow j\mbox{ we have}\\
&i\in \inneigh\big(\{h\leq k,\,h\neq j\}\big)\;\cup\;
\inneigh\Big(\bigcup_{g\geq 1}T_{K_0,g}(\subgr{\Gr}{[k]})\Big)\;\big|\;\Event([k])\Big\},
\end{align*}
where we used conditional independence of $\{\delta_{ij}:\,(i,j)\in[k]\times[k]\}$
and $\{\delta_{ij}, \mu_{ij}:\,(i,j)\in\{k+1,\dots,n\}\times [k]\}$ given event $\Event([k])$,
as this event refers only to the subgraph $\subgr{\Gr}{[k]}$.
Obviously,
$$\Prob\big\{i\to i+1\mbox{ for all }i\leq k-1\;\big|\;\Event([k])\big\}\leq p^{k-1}.$$
To estimate probability of the second event in the last formula, observe that
the condition
\begin{align*}
&\forall \;(i,j)\in \{k+1,\dots,n\}\times [k]
\mbox{ such that }i\leftarrow j\mbox{ we have}\\
&i\in \inneigh\big(\{h\leq k,\,h\neq j\}\big)\;\cup\;
\inneigh\Big(\bigcup_{g\geq 1}T_{K_0,g}(\subgr{\Gr}{[k]})\Big)
\end{align*}
means that for every left vertex $i$ of $\Gr$ with $i\notin W:=
[k]\cup\inneigh\big(\bigcup_{g\geq 1}T_{K_0,g}(\subgr{\Gr}{[k]})\big)$,
either
\begin{itemize}

\item[(a)] $\inneigh(i)\cap [k]=\emptyset$ or

\item[(b)] $|\outneigh(i)\cap [k]|\geq 2$ and $|\inneigh(i)\cap[k]|\geq 1$, that is,
$i$ has at least $2$ out-neighbors one of which is also its in-neighbor.

\end{itemize}
If there are $q$ pairs $(i,j)\in [n]\setminus W\,\times [k]$
such that $i\to j$ and the left vertex $i$ satisfies condition (b) then
necessarily
$$\sum_{h=1}^k\big|\inneigh(h^R)\big|-
\big|\inneigh([k]^R)\big|\geq q/2.$$
Otherwise, if there are less than $q$ such pairs then
$$\Big|\Big\{(i,j)\in [n]\setminus W\,\times[k]:\,i\leftarrow j\Big\}\Big|
\leq q.$$
In view of the above, we can write
\begin{align*}
\Prob\Big\{&\forall \;(i,j)\in \{k+1,\dots,n\}\times [k]
\mbox{ such that }i\leftarrow j\mbox{ we have}\\
&i\in \inneigh\big(\{h\leq k,\,h\neq j\}\big)\;\cup\;
\inneigh\Big(\bigcup_{g\geq 1}T_{K_0,g}(\subgr{\Gr}{[k]})\Big)\;\big|\;\Event([k])\Big\}\\
&\leq \Prob\Big\{\sum_{h=1}^k\big|\inneigh(h^R)\big|-
\big|\inneigh([k]^R)\big|\geq K_0 \,k/8\;\big|\;\Event([k])\Big\}\\
&+\Prob\Big\{\Big|\Big\{(i,j)\in [n]\setminus W\,\times[k]:\,i\leftarrow j\Big\}\Big|
\leq K_0 \,k/4\;\big|\;\Event([k])\Big\}.
\end{align*}
For the first of the two probabilities, we can apply Proposition~\ref{p: expansion}
to get an upper estimate $(\frac{n}{k})^{-k}$ (for $k=1$, the probability is clearly zero).
The second term can be represented as
$$\Prob\Big\{\sum\limits_{(i,j)\in [n]\setminus W\,\times[k]}
\delta_{ij}\mu_{ij}
\leq K_0 \,k/4\;\big|\;\Event([k])\Big\}.$$
Here, $\delta_{ij},\mu_{ij}$ are given by \eqref{Asmp on G} and are conditionally independent given $\Event([k])$.
Further, the cardinality of $[n]\setminus W$ given $\Event([k])$ is at least $n/2$.
Therefore,
\begin{align*}
\Prob\Big\{&\Big|\Big\{(i,j)\in [n]\setminus W\,\times[k]:\,i\leftarrow j\Big\}\Big|
\leq K_0 \,k/4\;\big|\;\Event([k])\Big\}\\
&\leq
\sum_{q=0}^{\lfloor K_0\,k/4\rfloor} {\lfloor nk/2\rfloor \choose q}(p/\alpha)^{q}(1-p/\alpha)^{\lfloor nk/2\rfloor-q}\\
&\leq e^{-pnk/(4\alpha)}+\sum_{q=1}^{\lfloor K_0\,k/4\rfloor} \big(e pnk/(\alpha q)\big)^q e^{-pnk/(4\alpha)}\\
&\leq e^{-c'pnk}
\end{align*}
for some $c'=c'(\alpha)>0$.
Combining all the estimates, we obtain
$$\Prob\big(\Event\;|\;\Event([k])\cap \Event_{\smallrefer{l: local to global random}}\big)
\leq 2p^{k-1}\,\big(e^{-c'pnk}+(n/k)^{-k}\indicator_{\{k\geq 2\}}\big).$$
Therefore
\begin{equation*}\label{eq: aux ghkmn}
\Prob\big(\Event\;|\;\Event_0\cap \Event_{\smallrefer{l: local to global random}}\big)
\leq \frac{\Prob(\Event([k])\cap \Event_{\smallrefer{l: local to global random}})}{\Prob(\Event_0\cap \Event_{\smallrefer{l: local to global random}})}
\cdot 2p^{k-1}\,\big(e^{-c'pnk}+(n/k)^{-k}\indicator_{\{k\geq 2\}}\big).
\end{equation*}

As we observed before, $\Prob(\Event_0\cap \Event_{\smallrefer{l: local to global random}})\geq 1/2$
and hence $\frac{\Prob(\Event([k])\cap \Event_{\smallrefer{l: local to global random}})}{\Prob(\Event_0\cap \Event_{\smallrefer{l: local to global random}})}\leq 2$.
Thus, 
$$\Prob\big(\Event\;|\;\Event_0\cap \Event_{\smallrefer{l: local to global random}}\big)\leq 4p^{k-1}\,\big(e^{-c'pnk}+(n/k)^{-k}\indicator_{\{k\geq 2\}}\big).$$
Finally, by the permutation invariance of our model, we get
$$\Exp\big(|\IsoChains_k|\;|\;\Event_0\cap \Event_{\smallrefer{l: local to global random}}\big)
\leq n^k\Prob\big(\Event\;|\;\Event_0\cap \Event_{\smallrefer{l: local to global random}}\big)
\leq 4(pn)^{k-1}n\,e^{-c'pnk}+8(pk)^{k-1}(k-1).$$
The result follows.
\end{proof}

\begin{lemma}[No self-balancing cyclic chains]\label{l: cyclic chains}
Let $n,p,\Gr$ and event $\Event_0$ be as in Lemma~\ref{l: cycle-free chains}.
Then for any $k\leq \log_{pn}n$ we have
\begin{align*}
\Prob\big\{&\mbox{$\Gr$ contains a $(K_0/2)$--self-balancing cyclic chain of length $k$}\;|\;\Event_0
\cap\Event_{\smallrefer{l: local to global random}}\big\}\\
&\leq e^{-c_{\smallrefer{l: cyclic chains}}pn k}+k\,\big(pk\big)^{k-1}
\end{align*}
for some $c_{\smallrefer{l: cyclic chains}}=c_{\smallrefer{l: cyclic chains}}(\alpha)>0$.
\end{lemma}
\begin{proof}
Fix any $k$ in $\{3,4,\dots,n\}$ (cyclic chains have length at least $3$).
As in the proof of Lemma~\ref{l: cycle-free chains}, denote by
$\Event(Q)$ ($Q\subset[n]$) the event
$\big\{\big|\inneigh\big(\bigcup_{g\geq 1}T_{K_0,g}(\subgr{\Gr}{Q})\big)\big|
\leq n\,e^{-c_{\smallrefer{p: support of finite classes}}pn/2}\big\}$.
By a similar argument, we get for every integer $w\leq k-2$:
\begin{align*}
\Prob\big\{(1,2,\dots,&k-1,w)\mbox{ is a $(K_0/2)$--self-balancing (cyclic) chain}\;|\;\Event([k-1])\cap
\Event_{\smallrefer{l: local to global random}}\big\}\\
\leq 2\Prob\big\{&i\to i+1\mbox{ for all }i\leq k-2;\;\;k-1\to w\;\big|\;\Event([k-1])\big\}\cdot\\
\Prob\Big\{&\forall \;(i,j)\in \{k,\dots,n\}\times [k-1]
\mbox{ such that }i\leftarrow j\mbox{ we have}\\
&i\in \inneigh\big(\{h\leq k-1,\,h\neq j\}\big)\;\cup\;
\inneigh\Big(\bigcup_{g\geq 1}T_{K_0,g}(\subgr{\Gr}{[k-1]})\Big)\;\big|\;\Event([k-1])\Big\}.
\end{align*}
The first probability is trivially at most $p^{k-1}$, whereas for the second we
apply the same argument as in the proof of Lemma~\ref{l: cycle-free chains}
to get an upper estimate
$$e^{-c'pn(k-1)}+(n/(k-1))^{-(k-1)}$$
for some $c'=c'(\alpha)>0$.
Hence,
\begin{align*}
\Prob\big\{&(1,2,\dots,k-1,w)\mbox{ is a $(K_0/2)$--self-balancing cyclic chain}\;|\;\Event_0\cap \Event_{\smallrefer{l: local to global random}}\big\}\\
&\leq
2\Prob\big\{(1,2,\dots,k-1,w)\mbox{ is a $(K_0/2)$--self-balancing cyclic chain}\;|\;\Event([k-1])\cap \Event_{\smallrefer{l: local to global random}}\big\}\\
&\leq 4p^{k-1}\big(e^{-c'pn(k-1)}+(n/(k-1))^{-(k-1)}\big),
\end{align*}
where we
reproduced arguments from the proof of Lemma~\ref{l: cycle-free chains}.

Notice that the number of ``potential'' cyclic chains of length $k$ is less than $k\cdot n^{k-1}$.
Hence,
\begin{align*}
\Prob\big\{&\mbox{$\Gr$ contains a $(K_0/2)$--self-balancing cyclic chain of length $k$}\;|\;\Event_0
\cap \Event_{\smallrefer{l: local to global random}}\big\}\\
&\leq 2k(pn)^{k-1}\big(e^{-c'pn(k-1)}+(n/(k-1))^{-(k-1)}\big)\\
&\leq e^{-c''pn k}+k\,\big(pk\big)^{k-1},
\end{align*}
where $c''>0$ may only depend on $\alpha$.
\end{proof}

Let us summarize the last two lemmas.

\begin{prop}[Statistics of self-balancing chains]\label{p: self-balancing stat}
There is $c_{\smallrefer{p: self-balancing stat}}>0$ depending only on $\alpha$ with the following property.
Let $n,p,\Gr$ satisfy \eqref{Asmp on G}.
For each $k\leq n$, denote by $\IsoChains_k$ the set of all $(K_0/2)$--self-balancing cycle-free chains of length $k$ for $\Gr$.
Finally, set
\begin{align*}
\Event_{\smallrefer{p: self-balancing stat}}:=
&\Big\{\Big|\inneigh\Big(\bigcup_{g\geq 1}T_{K_0,g}(\Gr)\Big)\Big|
\leq n\,e^{- c_{\smallrefer{p: self-balancing stat}}pn}\Big\}\,\cap\\
&\Big\{
|\IsoChains_k|\leq n\,e^{-c_{\smallrefer{p: self-balancing stat}}pnk}\;\;\mbox{ for all }\;1\leq k\leq n\Big\}\cap\\
&\Big\{\mbox{$\Gr$ does not contain $(K_0/2)$--self-balancing cyclic chains}
\Big\}.
\end{align*}
Then $\Prob(\Event_{\smallrefer{p: self-balancing stat}})\geq 1-\exp(-c_{\smallrefer{p: self-balancing stat}}pn)-n^{-c_{\smallrefer{p: self-balancing stat}}}$.
\end{prop}
\begin{proof}
The first part of the intersection in the definition of $\Event_{\smallrefer{p: self-balancing stat}}$
can be estimated using Proposition~\ref{p: support of finite classes}.
For the second part, if $1\leq k\leq \log_{pn}n$ is such that
$8(pk)^{k-1}(k-1)\leq n\,e^{-c_{\smallrefer{l: cycle-free chains}}pnk}$ then,
combining Lemma~\ref{l: cycle-free chains} with Markov's inequality, we get
$$
\Prob\big\{|\IsoChains_k|\geq \lceil n\,e^{-c_{\smallrefer{l: cycle-free chains}}pnk/2}\rceil
\;|\;\Event_0\cap \Event_{\smallrefer{l: local to global random}}\big\}
\leq 2e^{-c_{\smallrefer{l: cycle-free chains}}pnk/2}.
$$
On the other hand, if $8(pk)^{k-1}(k-1)\geq n\,e^{-c_{\smallrefer{l: cycle-free chains}}pnk}$
then, by Lemma~\ref{l: cycle-free chains},
$$
\Exp\big(|\IsoChains_k|\;|\;\Event_0\cap \Event_{\smallrefer{l: local to global random}}
\big)\leq 16(pk)^{k-1}(k-1)\leq n^{-c}
$$
for a constant $c>0$, where we have used that $p$ satisfies \eqref{Asmp on p weak}.
Hence, applying Markov's inequality, we get
$$
\Prob\big\{|\IsoChains_k|\geq \lceil n\,e^{-c_{\smallrefer{l: cycle-free chains}}pnk/2}\rceil
\;|\;\Event_0\cap \Event_{\smallrefer{l: local to global random}}\big\}
\leq \Prob\big\{|\IsoChains_k|\geq 1\;|\;\Event_0\cap \Event_{\smallrefer{l: local to global random}}\big\}\leq n^{-c}.
$$
This together with the union bound over all $1\leq k\leq \log_{pn}n$
and $\Prob(\Event_0\cap \Event_{\smallrefer{l: local to global random}})\geq 1-e^{-c'pn}-n^{-10}$,
gives
$$
\Prob\big\{
|\IsoChains_k|\leq n\,e^{-c_{\smallrefer{l: cycle-free chains}}pnk/2}\;\;\mbox{ for all }\;1\leq k\leq \log_{pn}n\big\}
\geq 1-e^{-c''pn}-n^{-c''}.
$$
It remains to note that, as any $(K_0/2)$--self-balancing chain of length $k>\log_{pn}n$ contains
$(K_0/2)$--self-balancing subchains of length $\lfloor \log_{pn}n\rfloor$, we get from the above
$$
\Prob\big\{
|\IsoChains_k|\leq n\,e^{-pnk}\;\;\mbox{ for all }\;\log_{pn}n< k\leq n\big\}\\
\geq \Prob\big\{
\IsoChains_{\lfloor \log_{pn}n\rfloor}= \emptyset\big\}\geq 1-e^{-c''pn}-n^{-c''},
$$
where we have used that $n\,e^{-c_{\smallrefer{l: cycle-free chains}}pn\,\lfloor\log_{pn}n\rfloor/2}<1$.

For the third part, by removing conditioning in Lemma~\ref{l: cyclic chains},
we get
\begin{align*}
\Prob\big\{\mbox{$\Gr$ contains a $(K_0/2)$--self-balancing cyclic chain of length $k\leq \log_{pn}n$}\big\}
\leq e^{-\widetilde cpn}+n^{-\widetilde c}.
\end{align*}
for some constant $\widetilde c>0$.
At the same time, applying the above estimate for cycle-free chains, we obtain
\begin{align*}
\Prob\big\{&\mbox{$\Gr$ contains a $(K_0/2)$--self-balancing cyclic chain of length $k>\log_{pn}n$}\big\}\\
&\leq\Prob\big\{\IsoChains_{\lceil\log_{pn}n\rceil-1}\neq\emptyset\big\}\leq e^{-c''pn}+n^{-c''}.
\end{align*}
The result follows.
\end{proof}

\subsection{Graph compression}\label{subs: matrix comp}

The notion of chains, the way it is considered in the previous subsection, would be sufficient if our only goal
was to bound the smallest
singular value of the shifted matrix $A-z\,\Id$. However, bounding {\it intermediate} singular values requires
more elaborate arguments.
In particular, our notion of chains should be extended to cover what can be called ``graph compression''.

A compression of a graph $G\in\grc_{n,n}$ is glueing together some pairs of left vertices of $G$
satisfying certain additional assumptions.
Namely, left vertices $i_1,i_2$ can be glued together only if their sets of out-neighbors are disjoint 
and all out-neighbors of $\{i_1,i_2\}$ belong to the infinite type $T_{K,\infty}(G)$.

Formally, let $G\in\grc_{n,n}$, let $K\geq 1$ be a parameter, and let $\phi:[n]\to[m]$ be
a surjective mapping. Let us assume that the mapping $\phi$ satisfies the following assumptions:
\begin{itemize}

\item For any $i\in[m]$, the preimage $\phi^{-1}(i)$ consists of either one or two elements;

\item For any $i_1\neq i_2\in[m]$ such that $\phi(i_1)=\phi(i_2)$, we have
$\outneigh(i_1)\cap\outneigh(i_2)=\emptyset$, and,
moreover, $\outneigh(i_1)\cup\outneigh(i_2)\subset T_{K,\infty}(G)$.

\end{itemize}
We will call such a mapping $\phi$ {\it $(G,K)$--admissible}.
The crucial observation, which will be made rigorous further in the paper, is that
a ``compression'' of the matrix $A-z\,\Id$ (which is defined as a matrix equivalent of the compression for $\Gr$)
typically contains only well-spread vectors in its kernel.

We say that a $(G,K)$--admissible mapping is {\it $u$--light} for some $u>0$ if
the set of in-neighbors of any right vertex contains not more than $u$ left vertices glued by $\phi$; formally,
$$\big|\big\{i\leq n:\;|\phi^{-1}(\phi(i))|=2\big\}\cap\inneigh(j)\big|\leq u\;\;\mbox{ for all }j\leq n.$$
The notion of $u$--light mappings allows us to identify those compressions which preserve expansion
properties of the graph (see Lemma~\ref{l: phi expansion} below).

\medskip

Denote by $\phi(G)$ the directed bipartite graph in $\grc_{m,n}$
obtained from $G$ by glueing together left vertices by $\phi$. That is,
if $\phi(i_1)=\phi(i_2)$ then
$$\outneigh^{\phi}(\phi(i_1))=\outneigh(i_1)\cup\outneigh(i_2)$$
and
$$\inneigh^{\phi}(\phi(i_1))=\inneigh(i_1)\cup\inneigh(i_2).$$
Here and further, $\inneigh(\cdot)$ and $\outneigh(\cdot)$ mean the sets
of in- and out-neighbors in the graph $G$ as before, and
$\inneigh^{\phi}(\cdot)$ and $\outneigh^{\phi}(\cdot)$
stand for the sets of in- and out-neighbors in the compressed graph $\phi(G)$.

\begin{lemma}\label{l: phi expansion}
Let $K,u>0$, let $G\in\grc_{n,n}$, $m\leq n$, and let $\phi:[n]\to[m]$ be a $u$--light $(G,K)$--admissible
mapping.
Further, assume that for some $\delta>0$ and $\varepsilon>0$ we have
$$\big|\inneigh(I)\big|\geq \sum_{i\in I}\big|\inneigh(i)\big|-\varepsilon K|I|\;\;
\mbox{for any subset $I$ of right vertices with $|I|\leq \delta n$}.$$
Then
$$\big|\inneigh^\phi(I)\big|\geq \sum_{i\in I}\big|\inneigh^\phi(i)\big|-(\varepsilon K+u)|I|\;\;
\mbox{for any subset of right vertices $I$ with $|I|\leq \delta n$}.$$
\end{lemma}
\begin{proof}
Fix any subset $I\subset[n]$ of right vertices with $|I|\leq \delta n$.
By the definition of a $(G,K)$--admissible mapping, the number of in-neighbors of any right vertex
does not change under compression, and hence
$$\sum_{i\in I}\big|\inneigh(i)\big|=\sum_{i\in I}\big|\inneigh^\phi(i)\big|.$$
On the other hand, the assumption that $\phi$ is light implies that
\begin{align*}
\big|\inneigh^\phi (I)\big|&\geq \big|\inneigh(I)\big|
-\sum\limits_{i\in I}\,|\{j\leq m:\;|\phi^{-1}(j)|=2,\,\phi^{-1}(j)\cap \inneigh(i)\neq\emptyset\}|\\
&\geq \big|\inneigh(I)\big|-u|I|.
\end{align*}
The result follows.
\end{proof}

\begin{lemma}\label{l: comp types}
Let $K>0$, let $G\in\grc_{n,n}$, $m\leq n$, and let $\phi:[n]\to[m]$ be a $(G,K)$--admissible
mapping. Then for any $g\geq 1$ we have $T_{K,g}(G)=T_{K,g}(\phi(G))$; as a consequence,
$T_{K,\infty}(G)=T_{K,\infty}(\phi(G))$.
\end{lemma}
\begin{proof}
The left vertices which are glued together by $\phi$, do not have out-neighbors
in $\bigcup\limits_{g\geq 1} T_{K,g}(G)$.
As we will show, this implies that each $T_{K,g}$ is preserved under compression.

First, since $|\inneigh^\phi(j)|=|\inneigh(j)|$ for all right vertices $j$,
we have that $T_{K,1}(G)=T_{K,1}(\phi(G))$. Now, fix $g>1$ and assume that all vertex types
$(K,1)$, $(K,2)$, $\dots$, $(K,g-1)$ coincide for $G$ and $\phi(G)$.

Take any $j\in T_{K,g}(G)$, and observe that
$\outneigh(j)$ and $\inneigh\big(\bigcup\limits_{h<g}T_{K,h}(G)\big)$
are not affected by the mapping $\phi$.
This, together with the definition of the type $(K,g)$, implies that $j\in T_{K,g}(\phi(G))$.

On the other hand, if $j\notin T_{K,g}(G)$ then
$\big|\outneigh(j)\setminus \inneigh\big(\bigcup\limits_{h<g}T_{K,h}(G)\big)\big|>K$,
and hence by the induction hypothesis
$\big|\outneigh^\phi(j)\setminus \inneigh^\phi\big(\bigcup\limits_{h<g}T_{K,h}(\phi(G))\big)\big|>K$.
Thus, $j\notin T_{K,g}(\phi(G))$.
\end{proof}

In the next lemma we show that, conditioned on certain realization of our graph $\Gr$, given
a uniform random $\ell$--element subset $J$ of the left vertices of $\Gr$, with high probability (with respect to $J$) we can find
a light $(\Gr,K)$--admissible mapping $\phi$ which glues together only some vertices in $J$.
The lemma will be applied to estimate the intermediate singular values of the random matrix $A-z\,\Id$
(see remark before Proposition~\ref{prop: othonormal system}).

\begin{lemma}\label{l: phi construction}
Fix $K>0$, $r>0$ and a realization of $\Gr$ from event $\EE_{\smallrefer{p: supports}}$ such that
$$\Big|\inneigh\Big(\bigcup_{g\geq 1}T_{K,g}\Big)\Big|\Big)
\leq e^{-r pn}\,n.$$
Let $\ell\geq n^{1/2}$
be a natural number, and let $\varepsilon\in(0,1/32)$. Let $J$ be a uniform random subset of $[n]$ of cardinality $\ell$
(defined on another probability space).
Denote by $\Event$ the event that there exists a $(\Gr,K)$--admissible
mapping $\phi$ satisfying the conditions
\begin{itemize}

\item $|\phi([n])|= n-\lfloor\varepsilon\ell\rfloor$;

\item $\phi^{-1}(\phi(i))=i$ for all $i\notin J$;

\item $\phi$ is $(64\varepsilon pn)$--light.

\end{itemize}
Then
$$\Prob_J(\Event)\geq 1-2e^{-c_{\smallrefer{l: phi construction}}\ell pn},$$
where $c_{\smallrefer{l: phi construction}}>0$ may only depend on $r$.
\end{lemma}
\begin{proof}
We will assume that $pn$ is large.
We will construct the mapping $\phi$ in three steps.
First, we show that with large probability we can find a subset $J'$ of $J$ of cardinality at least $\ell/2$
such that for any $i\in J'$ both the left vertex $i^L$ and the right vertex $i^R$ have ``good'' sets of neighbors.
Second, we extract from $J'$ a subset of pairs of left vertices $H_1$ with disjoint sets of out-neighbors in each pair.
Third, we construct subset $H_2\subset H_1$ which will define our compression $\phi$.

Let
$$M:=\big\{j\leq n:\;|\inneigh(j^R)|\geq 2pn\big\}.$$
Define a subset $I$ of all left vertices of $\Gr$ satisfying the following conditions:
\begin{itemize}

\item $|\outneigh(j^L)|\leq 2pn$;

\item $\outneigh(j^L)\subset T_{K,\infty}(\Gr)$;

\item $j^L\notin\inneigh(M)$.

\end{itemize}
Since we condition on event $\EE_{\smallrefer{p: supports}}$ and  because of Proposition~\ref{p: supports},
we have
$|\inneigh(M)|\leq \exp(-c_1pn)n$.
Further, by our assumption on the realization of $\Gr$,
the total number of left vertices of $\Gr$
whose sets of out-neighbors have non-empty intersection with $[n]\setminus T_{K,\infty}(\Gr)$,
is bounded from above by $\exp(-c_2pn)n$, for a constant $c_2>0$.
Finally, again in view of conditioning on $\EE_{\smallrefer{p: supports}}$, cardinality of the set of left vertices
with at least $2pn$ out-neighbors is bounded above by $\exp(-c_3pn)n$.
Combining the bounds, we get $|I^c|\leq \exp(-c'pn)n$, for some $c'>0$.

Now, since $J$ is chosen uniformly at random, standard concentration inequalities imply that the set
$J':=J\cap I$ has cardinality at least $\ell/2$ with probability at least $1-e^{-\tilde c\ell pn}$.

From now on, we fix $J'$ with $|J'|\geq \ell/2$ and work with it as a deterministic set.
We construct the set of pairs $H_1$ step by step as follows:
at $k$--th step, choose any index $j_k$ in $J'$ which was not previously selected.
By our construction of $J'$, $|\outneigh(j_k^L)|\leq 2pn$.
On the other hand, since we conditioned on the event $\EE_{\smallrefer{p: supports}}$,
for every $u\in \outneigh(j_k^L)$, the cardinality of $\inneigh(u)$ is at most $2pn+C\log n$.
Hence, the number of unselected indices in $J'$ whose sets of out-neighbors have a non-empty intersection with
$\outneigh(j_k^L)$, is bounded above by $2pn(2pn+C\log n)$. Choose one of the indices of $J'$ which
does not belong to the set, and add the resulting pair to $H_1$.

Continuing the process, we get a collection of pairs $H_1$ with $|H_1|\geq (|J'|-2pn(2pn+C\log n))/2\geq \ell/8$,
where we used our assumptions on $\ell$ and $pn$.

It remains to construct a subset $H_2\subset H_1$ of pairs of vertices to be glued together by the mapping $\phi$.

Let $Q$ be a uniform random subset of $H_1$ of cardinality $\lfloor 2\varepsilon \ell\rfloor \leq 16\varepsilon |H_1|$
(defined on another probability space).
By the construction of $H_1$, for any right vertex $u$ of $\Gr$
we have
$$\big|\big\{(j,\widetilde j)\in H_1:\;\{j,\widetilde j\}\cap \inneigh(u)\neq\emptyset\big\}\big|\leq 2pn.$$
Hence, a standard concentration inequality implies
$$\Prob_Q\big\{
\big|\big\{(j,\widetilde j)\in Q:\;\{j,\widetilde j\}\cap \inneigh(u)\neq\emptyset\big\}\big|\geq 64\varepsilon pn
\big\}\leq \exp(-c_4\varepsilon pn),\;\;u\leq n.$$
Further, using the condition $|\outneigh(j^L)|\leq 2pn$, $j\in J'$, we get
$$\big|\big\{u\leq n:\;\big|\big\{(j,\widetilde j)\in H_1:\;\{j,\widetilde j\}\cap \inneigh(u)\neq\emptyset\big\}\big|\geq
\varepsilon pn \big\}\big|\leq C_2|H_1|/\varepsilon.$$
For every $u\leq n$, denote by $\indicator_u$ the indicator of the event (with respect to the randomness of $Q$)
$$\big\{\big|\big\{(j,\widetilde j)\in Q:\;\{j,\widetilde j\}\cap \inneigh(u)\neq\emptyset\big\}\big|\geq 64\varepsilon pn\big\}.$$
Then, by the above, $\Exp_Q\indicator_u\leq \exp(-c_4\varepsilon pn)$ for all $u\leq n$ and, moreover,
$\indicator_u=0$ for all but at most $C_2|H_1|/\varepsilon$ indices. Hence,
$$\sum\limits_{u=1}^n\indicator_u\leq \frac{C_2|H_1|}{\varepsilon}\exp(-c_4\varepsilon pn)\leq \exp(-c_5\varepsilon pn)|H_1|$$
for some realization $Q_0$ of $Q$.

Denote by $R$ the collection of all indices $u\leq n$ with $\big|\big\{(j,\widetilde j)\in Q_0:\;\{j,\widetilde j\}\cap\inneigh(u)
\neq\emptyset\big\}\big|\leq 64\varepsilon pn$. The above argument shows that
$|R^c|\leq \exp(-c_5\varepsilon pn)|H_1|$.

Define $H_2$ as the subset of $Q_0$ of cardinality $\lfloor\varepsilon \ell\rfloor$,
where we remove all pairs of indices $\{j,\widetilde j\}$
intersecting with $\inneigh(u)$ for some right vertex $u\in R^c$ (such a subset exists
since $j^L\notin\inneigh(M)$ for all $j\in J'$, and $|R^c|$ is small).
Finally, observe that if we define any surjective mapping $\phi:[n]\to [n-| H_2|]$ with
$\phi(j)=\phi(\widetilde j)$ for all $(j,\widetilde j)\in H_2$,
is $(\Gr,K)$--admissible and is $(64\varepsilon pn)$--light.
\end{proof}


As the next step, we extend the notion of chains to compressed graphs.
Specifically, let $m\leq n$, let $G\in\grch_{n,n}$,
and let $\phi:[n]\to[m]$ be a $(G,K)$--admissible map.
We will say that a sequence $(j_\ell)_{\ell=1}^k$ of {\it right} vertices of $\phi(G)$
is {\it a $\phi$--chain of length $k$ for $\phi(G)$} if for any $\ell<k$ we have
$j_{\ell}\neq j_{\ell+1}$ and $\phi(j_\ell)^L\to j_{\ell+1}^R$ (i.e.\ the edge from $\phi(j_\ell)^L$
to $j_{\ell+1}^R$ belongs to the edge set of $\phi(G)$).
If all $j_\ell$'s ($1\leq \ell\leq k$) are distinct, we will call such a $\phi$--chain {\it cycle-free}.
Further, if $j_\ell$'s for $1\leq \ell\leq k-1$ are all distinct but $j_k=j_{u}$ for some $u<k-1$, the $\phi$--chain will be called {\it cyclic}.
Note that for $m=n$ and $\phi$ being the identity map, the above notion of chains coincides with the one given
in the previous subsection.
Similarly to the ``uncompressed'' setting,
$\phi$--chains can be associated with ``zig-zag'' paths on the graph $\phi(G)$.
Namely, $(j_\ell)_{\ell=1}^k$ lies on the path $j_1^R\to \phi(j_1)^L\to j_2^R\to \phi(j_2)^L\to\dots \to \phi(j_{k-1})^L\to j_k^R$,
with $j_\ell^R\neq j_{\ell+1}^R$ for all $\ell<k$.

\medskip

The definition of self-balancing $\phi$--chains and contact elements carries to the generalized setting in a straightforward way.
We restate the definitions for completeness.
Let $G\in\grch_{n,n}$ and let $\phi:[n]\to[m]$ be $(G,K)$--admissible.
Given $K>0$, a $\phi$--chain $J=(j_\ell)_{\ell=1}^k$ for $\phi(G)$ is
{\it $(K,\phi)$--self-balancing} if
$j_\ell\notin T_{K,\infty}(\phi(G))$ for all $\ell\leq k$ and, moreover,
for any $\ell\leq k$, we have
$$\outneigh^\phi(j_\ell)\subset \inneigh^\phi\Big(\bigcup_{g\geq 1}T_{K,g}\big(\phi(G))\setminus\{j_\ell\}\Big).$$

By a ``self-balancing chain'' for a graph $G\in\grch_{n,n}$ we mean $\psi$--self-balancing chain
with $\psi$ being the identity mapping. This makes the new generalized notions compatible
with the previous definitions.

The following lemma allows to easily transfer the results of the previous subsection to the new generalized setting.
\begin{lemma}[Compression via an admissible mapping]\label{l: compviaadm}
Let $G\in\grch_{n,n}$ and let $\phi:[n]\to[m]$ be $(G,K)$--admissible (for some $K>0$).
Then
\begin{itemize}

\item Any chain $J$ for $G$ is also a $\phi$--chain for $\phi(G)$. Converse is not true in general, however

\item Any $\phi$--chain $J$ for $\phi(G)$ such that $J\cap T_{K,\infty}(G)=\emptyset$, is also a chain for $G$;

\item Any cyclic $\phi$--self-balancing chain for $\phi(G)$ is also an self-balancing cyclic chain for $G$, and vice versa;

\item Any cycle-free $\phi$--self-balancing chain for $\phi(G)$ is also a cycle-free self-balancing chain for $G$, and vice versa.

\end{itemize}
\end{lemma}
\begin{proof}
The first assertion of the lemma follows immediately from definitions.

For the second assertion, assume that $J=(j_\ell)_{\ell=1}^k$ is a $\phi$--chain for $\phi(G)$.
Then $j_\ell^R\to\phi(j_\ell)^L\to j_{\ell+1}^R$ for all $1\leq \ell<k$. Now, the condition $J\cap T_{K,\infty}(G)=\emptyset$
implies that $\phi^{-1}(\phi(j_\ell))=j_\ell$.
On the other hand, $j_\ell^R\to j_{\ell}^L\to j_{\ell+1}^R$ in $G$ if and only if
$j_\ell^R\to\phi(j_\ell)^L\to j_{\ell+1}^R$ in $\phi(G)$, $\ell<k$.
Thus, $J$ is a chain for $G$.

For the third and fourth assertions, let $J=(j_\ell)_{\ell=1}^k$ be a
$(K,\phi)$--self-balancing cyclic (resp., cycle-free) chain for $\phi(G)$.
Then, in particular, $j_\ell\notin T_{K,\infty}(\phi(G))$ for all $\ell\leq k$,
and hence, by the definition of an admissible mapping,
for any left vertex $i$ of $G$, such that $\outneigh(i)\cap J\neq \emptyset$, we necessarily have
$\phi^{-1}(\phi(i))=i$. That is, restricted to $\inneigh(j_\ell)$ ($\ell\leq k$), $\phi$
acts as a bijection. This, together with the stability of the vertex types under the ``compression''
operation (Lemma~\ref{l: comp types})
implies that $J$ must be an self-balancing cyclic (resp., cycle-free) chain for $G$
as well (in the sense of subsection~\ref{subs: chains}).
The converse statement is checked in a similar way.
\end{proof}

Combining Proposition~\ref{p: self-balancing stat}, Lemma~\ref{l: local to global random} and
Lemma~\ref{l: compviaadm}, we get
\begin{prop}[Statistics of self-balancing $\phi$--chains]\label{p: chains combined}
Let $n,p,\Gr$ satisfy assumptions \eqref{Asmp on G}.
Define
\begin{align*}
\Event_{\smallrefer{p: chains combined}}^1&:=\Big\{\Big|\inneigh\Big(\bigcup_{g\geq 1}T_{K_0,g}(\Gr)\Big)\Big|
\leq n\,e^{- c_{\smallrefer{p: self-balancing stat}}pn}\Big\};\\
\Event_{\smallrefer{p: chains combined}}^2&:=\Big\{\mbox{For any $(\Gr,K_0/2)$--admissible map $\phi$}\\
&\hspace{0.9cm}\mbox{there are no $(K_0/2,\phi)$--self-balancing cyclic chains for $\phi(\Gr)$}\Big\};\\
\Event_{\smallrefer{p: chains combined}}^3&:=\Big\{\mbox{For any $(\Gr,K_0/2)$--admissible map $\phi$ and any}\;\;1\leq k\leq n^{1/4},\\
&\hspace{0.9cm}\mbox{all but at most }n\,e^{-c_{\smallrefer{p: self-balancing stat}}pnk}\mbox{ cycle-free $\phi$--chains of length $k$}\\
&\hspace{0.9cm}\mbox{for $\phi(\Gr)$ are not $(K_0/2,\phi)$--self-balancing}
\Big\},
\end{align*}
and set
\begin{align*}
\Event_{\smallrefer{p: chains combined}}:=
\Event_{\smallrefer{p: chains combined}}^1\cap \Event_{\smallrefer{p: chains combined}}^2\cap
\Event_{\smallrefer{p: chains combined}}^3.
\end{align*}
Then $\Prob(\Event_{\smallrefer{p: chains combined}})\geq 1-\exp(-c_{\smallrefer{p: chains combined}}pn)-n^{-c_{\smallrefer{p: chains combined}}}$
for a universal constant $c_{\smallrefer{p: chains combined}}>0$.
\end{prop}

\subsubsection{Number of $\phi$--chains}

The next lemma bounds the number of starting vertices of $\phi$--chains terminating in a set of right vertices $S$.

\begin{lemma}\label{l: number of chains}
Let $n,p,\Gr$ satisfy assumption \eqref{Asmp on G},
and $\Event_{\smallrefer{p: supports}}$ be as in Proposition~\ref{p: supports}.
Fix a realization of $\Gr$ in $\Event_{\smallrefer{p: supports}}$.
Let $K>0$ and let $\phi$ be a $(\Gr,K)$--admissible map.
Then for any subset $S\subset[n]$ and any $k\geq 1$, the set
$$\ChainsSet_{k,S}:=\big\{j\leq n:\;\mbox{there is a $\phi$--chain $J=(j_\ell)_{\ell=1}^{u}$ for $\phi(\Gr)$
with}\;u\leq k,\;\;
\mbox{$j_1=j$, $j_u\in S$}\big\}$$
has cardinality at most
$$\Big(C_{\smallrefer{l: number of chains}}pn+C_{\smallrefer{l: number of chains}}\log\frac{n}{|S|}\Big)^{k-1}\,|S|,$$
where $C_{\smallrefer{l: number of chains}}>0$ is a universal constant.
\end{lemma}
\begin{proof}
We will construct the $\phi$--chains with the last element in $S$ ``backwards'',
using the definition of $\Event_{\smallrefer{p: supports}}$.
Clearly, $\ChainsSet_{1,S}=S$.
Take any $k> 1$.
Observe that cardinality of the set $\ChainsSet_{k,S}$ can be bounded
from above by {\it twice} the cardinality of the union
of $\inneigh^\phi(i)$, $i\in \ChainsSet_{k-1,S}$, plus the cardinality of $S$,
with the latter coming from chains of length one.
Indeed, any $\phi$--chain $J=(j_\ell)_{\ell=1}^u$ of length $u\geq 2$
must satisfy $j_1^R\to\phi(j_1)^L\to j_2^R$,
and therefore $j_1$ necessarily belongs to the set $\phi^{-1}(\inneigh^\phi(j_2))$.
Thus, all possible choices of $j_1$ are contained in the set
$\phi^{-1}\big(\inneigh^\phi(\ChainsSet_{k-1,S})\big)$
of cardinality
at most $2\,|\inneigh^\phi(\ChainsSet_{k-1,S})|$.
In view of Lemma~\ref{l: union of supp}, we obtain
\begin{align*}
\frac{1}{2}\big|\ChainsSet_{k,S})\big|&\leq
\big|\inneigh^\phi(\ChainsSet_{k-1,S})\big|
\leq
\big|\inneigh(\ChainsSet_{k-1,S})\big|
\leq C
\Big(pn+\log\frac{n}{|\ChainsSet_{k-1,S}|}\Big)\,|\ChainsSet_{k-1,S}|
+|S|\\
&\leq C'\Big(pn+\log\frac{n}{|S|}\Big)\,|\ChainsSet_{k-1,S}|,\quad k>1.
\end{align*}
Applying the estimate iteratively, we get the result.
\end{proof}

As a corollary of the last estimate, we obtain
\begin{lemma}\label{l: n of chains 2}
Let $n,p,\Gr$ satisfy assumption \eqref{Asmp on G},
and $\Event_{\smallrefer{p: supports}}$ be as in Proposition~\ref{p: supports}.
Fix a realization of $\Gr$ in $\Event_{\smallrefer{p: supports}}$.
Further, take $K>0$ and a $(\Gr,K)$--admissible map $\phi$.
Let $V$ be a subset of $\phi([n])$, let $k\geq 1$,
and let $\mathcal J$ be a collection of $\phi$--chains for $\phi(\Gr)$
with distinct first elements, each $J=(j_\ell)_{\ell=1}^p\in\mathcal J$ of length at most $k$
and such that
$$\inneigh^\phi(J)\cap V\neq\emptyset.$$
Then necessarily
$$|\mathcal J|\leq \Big(C_{\smallrefer{l: n of chains 2}}pn+C_{\smallrefer{l: n of chains 2}}\log\frac{n}{|V|}\Big)^{k}\,|V|$$
for a universal constant $C_{\smallrefer{l: n of chains 2}}>0$.
\end{lemma}
\begin{proof}
First, we estimate cardinality of the subset $V'\subset[n]$ of all right vertices $j$ such that
$\inneigh^\phi(j)\cap V\neq\emptyset$.
Clearly, the last condition is equivalent to $\inneigh(j)\cap \phi^{-1}(V)\neq\emptyset$,
where $|\phi^{-1}(V)|\leq 2|V|$.
Lemma~\ref{l: union of supp} implies
$$|V'|\leq C\Big(pn+\log\frac{n}{|V|}\Big)\,|V|$$
for a universal constant $C>0$.
On the other hand, by the definition of sets $\ChainsSet_{k,S}$
from Lemma~\ref{l: number of chains}, we get
$|\mathcal J|\leq|\ChainsSet_{k,V'}|
$,
and hence, by the cardinality estimate from Lemma~\ref{l: number of chains},
$$|\mathcal J|\leq \Big(C'pn+C'\log\frac{n}{|V'|}\Big)^{k-1}\,|V'|.$$
The result follows.
\end{proof}

\subsection{Shells} \label{subs: shells}

In our approach, we separate observations related to the structure of
the underlying graph $\Gr$, from linear algebraic aspects of the problem.
The notion which connects these two parts of the argument is {\it shell}.

Let $G\in\grc_{k,m}$, let $d\geq 1$ be a natural number,
and $M\subset[k]$ be any subset of left vertices of $G$.
We say that a finite sequence $\Am=(\Ael_\ell)_{\ell=0}^{d}$ of sets of right vertices of $G$ is an
{\it $M$--shell of depth $d$ for $G$}
if for any $0\leq \ell\leq d-1$ and any
$j\in \Ael_\ell$ we have the following: whenever a left vertex $i\in [k]\setminus M$ is such that $i\leftarrow j$,
there is a right vertex $j'=j'(i,j)\neq j$ in $\Ael_{\ell+1}$ such that
$i\to j'$.
The sets $\Ael_\ell$ are called {\it layers} of the shell.
The subset $\Ael_0$ will be called {\it the center} of $\Am$.

As we prove below, assuming certain expansion properties for the graph $G$, we can show that
{\it any} shell centered in $T_{K,\infty}(G)$ (with the center of sufficiently large cardinality),
must be fast expanding in the sense that cardinalities of the layers grow at an exponential rate.
\begin{lemma}[Expansion property of shells]\label{l: expansion for structures}
Let $k,m$ be large integers, $M\subset[k]$,
$K>0$; let $G\in\grc_{k,m}$, and assume that for some $\delta\in(0,1]$ and
$\varepsilon\in (0,1/32)$
we have
\begin{equation}\label{eq: col exp prop}
\big|\inneigh(I)\big|\geq \sum_{i\in I}\big|\inneigh(i)\big|-\varepsilon K|I|,\;\;
\mbox{for any subset of right vertices $I$ with $|I|\leq \delta m$}.
\end{equation}
Further, fix any non-empty $J\subset T_{K,\infty}(G)$ with $|J|\leq \delta m/2$
such that
$$\frac{2}{K}\sum_{i\in M}|\outneigh(i)|\leq \frac{|J|}{2}.$$
Then any $M$--shell $\Am=(\Ael_\ell)_{\ell=0}^{d}$ for
$G$ of depth $d\geq 1$
centered in $J$ (if such a shell exists), satisfies
$$|\Ael_\ell|\geq \min\big(\lfloor\delta m/4\rfloor,(32\varepsilon)^{-\ell}|J|\big),\quad\quad 0\leq\ell\leq d.$$
\end{lemma}
\begin{proof}
Let us fix any $d\geq 1$ and any $M$--shell $\Am=(\Ael_\ell)_{\ell=0}^{d}$ for $G$ centered in $J$
(if such a shell does not exist then there is nothing to prove).
Observe that the total number of right vertices whose sets of in-neighbors intersect
with $M$ at least on $K/2$ vertices, is at most $\frac{2}{K}\sum_{i\in M}|\outneigh(i)|$,
which is less than $|J|/2$, by the assumptions on $J$ and $G$.

We will prove assertion of the lemma via an inductive argument.
At zero step, we set $\widetilde V_0$ to be the subset of all vertices $j\in J$ such that
\begin{equation}\label{eq: aux apbg;skdf}
\Big|\outneigh(j)\setminus\Big(M\,\cup\,\inneigh\Big(\bigcup_{g\geq 1}T_{K,g}(G)\Big)\Big)\Big|\geq K/2.
\end{equation}
By the above remark on the cardinality of $M$ and the definition of $T_{K,\infty}(G)$,
we get $|\widetilde V_0|\geq |J|/2$. We then let $V_0$ to be a subset of $\widetilde V_0$
of cardinality $\min(\lfloor \delta m/2\rfloor, |\widetilde V_0|)$.

Now, fix $1\leq \ell\leq d$, and assume that a subset $V_{\ell-1}$ of $\Ael_{\ell-1}\cap T_{K,\infty}(G)$,
of cardinality $|J|/2\leq|V_{\ell-1}|\leq\delta m/2$,
such that all $j\in V_{\ell-1}$ satisfy \eqref{eq: aux apbg;skdf},
has been defined.

Denote by $Q_\ell$ the collection of all edges
$i\leftarrow h$, with
$h\in V_{\ell-1}$ and $i\in [k]\setminus M$.
Note that by the definition of an $M$--shell, for any edge $i\leftarrow h$ in $Q_\ell$
there is right vertex $r\neq h$ with $r\in \Ael_\ell$ and $i\to r$.
Thus, we can define a function $f:Q_\ell\to[m]$, with $r=f(i\leftarrow h)$
($f$ need not be uniquely defined).
Further, by the condition on $V_{\ell-1}$,
for any $h\in V_{\ell-1}$ there are at least $K/2$ left vertices $i$ such that
the edge $i\leftarrow h$ belongs to $Q_\ell$ and $f(i,h)\in T_{K,\infty}(G)$.
Thus, the set $Q'_\ell:=\{i\leftarrow h\mbox{ in }Q_\ell:\;f(i,h)\in T_{K,\infty}(G)\cap \Ael_\ell\}$ has cardinality at least $K|V_{\ell-1}|/2$.

Set $S:=\{f(i,h):\,i\leftarrow h\mbox{ in }Q_\ell'\}$, then
for any $h\in V_{\ell-1}$ we have $|\inneigh(h)\,\cap\,\inneigh(S\setminus\{h\})|\geq K/2$. This immediately implies that
$$\sum_{r\in S\cup V_{\ell-1}}\big|\inneigh(r)\big|-\big|\inneigh(S\cup V_{\ell-1})\big|\geq K|V_{\ell-1}|/4.$$
Combining this with the expansion property taken as the assumption of the lemma, we get
$$K|V_{\ell-1}|/4\leq \varepsilon K|S\cup V_{\ell-1}|,$$
unless $|S\cup V_{\ell-1}|>\delta m$.
Hence, we have
$$|S|\geq \min\big(\big((4\varepsilon)^{-1}-1\big)|V_{\ell-1}|,\delta m-|V_{\ell-1}|\big)
\geq \min\big((8\varepsilon)^{-1}|V_{\ell-1}|,\delta m/2\big).$$
Further, let $S'\subset S$ be the set of all right vertices in $S$ whose
sets of in-neighbors intersect with $M$
at most on $K/2$ elements. Obviously, the total number of vertices in $S\setminus S'$
cannot be bigger than $\frac{2}{K}\sum_{i\in M}|\outneigh(i)|$.
Hence, by the assumptions on $J$ and $\delta$, we have
$$|S'|\geq |S|-\frac{|J|}{2}\geq\min\Big((8\varepsilon)^{-1}|V_{\ell-1}|-\frac{|J|}{2},\delta m/4\Big)
\geq \min\big(\lfloor (16\varepsilon)^{-1}|V_{\ell-1}\rfloor,\lfloor\delta m/4\rfloor\big),$$
where in the last inequality we used the induction hypothesis.
Now, we set $V_\ell$ as a subset of $S'$ of cardinality
$\min\big(\lfloor (16\varepsilon)^{-1}|V_{\ell-1}|\rfloor,\lfloor\delta m/4\rfloor\big)$.
Then $|J|/2\leq |V_\ell|\leq \delta m/4$, completing the induction.
The result follows.
\end{proof}

The next lemma shows that if the center of a shell is sufficiently large then
the union of first few layers has a large intersection with $T_{K_0/2,\infty}(\Gr)$.

\begin{lemma}\label{l: hit infty type}
There are constants $C_{\smallrefer{l: hit infty type}},c_{\smallrefer{l: hit infty type}},c_{\smallrefer{l: hit infty type}}'>0$
with the following property.
Let $n,p,\Gr$ satisfy assumptions \eqref{Asmp on G} and fix a realization of $\Gr$
in $\Event_{\smallrefer{p: supports}}\cap\,\Event_{\smallrefer{p: chains combined}}$.
Let $m\leq n$, and let $\phi:[n]\to[m]$ be a
$(\Gr,K_0/2)$--admissible surjective mapping.
Further, let $M\subset[m]$ be a subset of left vertices of $\phi(\Gr)$ satisfying
$$|M|\leq n/\sqrt{L},$$
for some $L>1$.
Let $1\leq k$, and let $\Am=(\Ael_\ell)_{\ell=0}^k$ be any $M$--shell for $\phi(\Gr)$
such that
$$|\Ael_0|\geq \big(C_{\smallrefer{l: hit infty type}}pn+
C_{\smallrefer{l: hit infty type}}\log L\big)^{k+2} n/\sqrt{L}+C_{\smallrefer{l: hit infty type}}\, e^{-c_{\smallrefer{p: self-balancing stat}}pnk}n.
$$
Then necessarily
\begin{align*}
\Big|\bigcup_{\ell=0}^{k}\Ael_\ell\,\cap T_{K_0/2,\infty}
(\phi(\Gr))\Big|
&\geq(k+1)\max\Big(n/\sqrt{L},\sum\limits_{i\in M}|\outneigh(i)|\Big).
\end{align*}
\end{lemma}
The assumption on $|\Ael_0|$ requires $L$ to be sufficiently large; otherwise, the statement is vacuous.
\begin{proof}
Fix an $M$--shell $\Am=(\Ael_\ell)_{\ell=0}^k$ for $\phi(\Gr)$ satisfying the above condition
for $\Ael_0$, where we assume that $C_{\smallrefer{l: hit infty type}}>0$ is a large universal constant to be chosen later.
Note that the definition of $\phi(\Gr)$,
implies that $\phi(j)^L\leftarrow j^R$ for all $j\leq n$.
Now, starting with any $j\in \Ael_0$, let us construct a sequence of vertices
$J=J(j)=(j_\ell)_{\ell=0}^q$ (with $q\leq k-1$) as follows:

At Step $0$, we set $j_0:=j\in\Ael_0$.

At Step $\ell$, $k\geq\ell\geq 1$, we have indices $j_0,j_1,\dots,j_{\ell-1}$
constructed, with $j_{r}\in\Ael_r$ for all $r<\ell$.
We do the following.
If $\inneigh^\phi(j_{\ell-1})\,\cap \,M\neq \emptyset$ then set $q:=\ell-1$ and teminate.
Otherwise, if $j_{\ell-1}=j_r$ for some $r<\ell-1$ then, again, set $q:=\ell-1$ and terminate.
Otherwise, as $\inneigh^\phi(j_{\ell-1})\,\cap \,M= \emptyset$ and
$\phi(j_{\ell-1})^L\leftarrow j_{\ell-1}^R$, by the definition of an $M$--shell
there is a right vertex $j_\ell\neq j_{\ell-1}$ such that $j_\ell\in \Ael_\ell$ and $\phi(j_{\ell-1})\to j_\ell$;
this vertex is added to the sequence.

At Step $k$ (if this step is reached), we set $q:=k-1$ and terminate.

\medskip

As a result of the above procedure, for any $j\in\Ael_0$ we obtain
a $\phi$--chain $J=J(j)=(j_\ell)_{\ell=0}^q$
for $\phi(\Gr)$ of length $q+1\leq k$ such that $j=j_0$,
$\inneigh^\phi(j_\ell)\,\cap \,M= \emptyset$ for all $\ell\leq q-1$, and, additionally,
one of the following three conditions holds:
\begin{itemize}

\item[(a)] $\inneigh^\phi (j_q)\,\cap \,M\neq \emptyset$;

\item[(b)] $q=k-1$, $\inneigh^\phi (j_q)\,\cap \,M= \emptyset$, and $J$ is cycle-free;

\item[(c)] $\inneigh^\phi (j_q)\,\cap \,M= \emptyset$ and $J$ is cyclic.

\end{itemize}
Fix one such chain for each $j\in \Ael_0$ and denote the set of these chains by $\mathcal J$.
If $n/\sqrt{L}\geq 1$ then Lemma~\ref{l: n of chains 2} yields that the number of $\phi$--chains from $\mathcal J$
satisfying condition (a) is bounded from above by
$$\Big(\widetilde C pn+\widetilde C\log\frac{n}{\lfloor n/\sqrt{L} \rfloor}\Big)^{k}\,\lfloor n/\sqrt{L} \rfloor
\leq \big(C'pn+C'\log L\big)^{k} n/\sqrt{L}$$
(for $n/\sqrt{L}< 1$ we have $M=\emptyset$, and the upper bound trivially holds as well).
Further, on the event $\Event_{\smallrefer{p: chains combined}}$,
no $\phi$--chains satisfying condition (c) are $(K_0/2,\phi)$--self-balancing, and at most
$n e^{-c_{\smallrefer{p: self-balancing stat}}pnk}$ chains satisfying (b)
are $(K_0/2,\phi)$--self-balancing.
Therefore, because of our assumption on $|\mathcal J|=|\Ael_0|$ (choosing a sufficiently large $C_{\smallrefer{l: hit infty type}}$), we get
that there is a subset $\mathcal J'\subset \mathcal J$ of cardinality at least $\frac{1}{2}|\mathcal J|$
such that any $J\in\mathcal J'$ satisfies conditions (b) or (c) and is not $(K_0/2,\phi)$--self-balancing.

Pick a chain $J=(j_\ell)_{\ell=0}^q$ in $\mathcal J'$.
Then, by definition of non--self-balancing chains, we have the following alternative.

\begin{enumerate}

\item\label{item: (1)} There is $v\leq q$ such that $j_v\in T_{K_0/2,\infty}(\phi(\Gr))$.
We denote the vertex $j_v$ by ${\bf j}_J$.

\item\label{item: (2)}
The chain $J$ does not satisfy \eqref{item: (1)}, and
there is $v\leq q$ and a left vertex $i\in \outneigh^\phi(j_v)$
such that $\outneigh^\phi(i)\setminus\{j_v\}\subset T_{K_0/2,\infty}(\phi(\Gr))\setminus J$.
As $J$ is of type (b) or (c), we have $i\notin M$, and hence by the definition of an $M$-shell, the
set $\outneigh^\phi(i)\setminus\{j_v\}$ must be non-empty, implying that
there is a right vertex
$$j\in \big(\outneigh^\phi(i)\setminus\{j_v\}\big)\cap \Ael_{v+1}
\subset (\Ael_0\cup\dots\cup \Ael_k)\,\cap\,T_{K_0/2,\infty}(\phi(\Gr)).$$
In this case, we set ${\bf w}_J:=j$.

\end{enumerate}

Denote
$$S_1:=\big\{{\bf j}_J:\;J\mbox{ satisfies }\eqref{item: (1)}\big\};\quad
S_2:=\big\{{\bf w}_J:\;J\mbox{ satisfies }\eqref{item: (2)}\big\},
$$
and observe that $S_1,S_2\subset  (\Ael_0\cup\dots\cup \Ael_k)\,\cap\,T_{K_0/2,\infty}(\phi(\Gr))$ and at least one of these sets is non-empty.

Take any set $W$ of right vertices containing $S:=S_1\cup S_2$.
By Lemma~\ref{l: number of chains} applied to chains $J$ truncated at ${\bf j}_J$, we have
$$\big|\big\{J\in\mathcal J':\;J\mbox{ satisfies }\eqref{item: (1)}\big\}\big|
\leq \Big(C_{\smallrefer{l: number of chains}}pn+C_{\smallrefer{l: number of chains}}\log\frac{n}{|W|}\Big)^{k-1}\,|W|.$$
Similarly, by Lemma~\ref{l: n of chains 2} we have
$$\big|\big\{J\in\mathcal J':\;J\mbox{ satisfies }\eqref{item: (2)}\big\}\big|
\leq \big(\widetilde Cpn+\widetilde C\log\frac{n}{|\inneigh^\phi(W)|}\Big)^{k}\,|\inneigh^\phi(W)|,
$$
where
$$\big|\inneigh^\phi(W)\big|\leq  C\Big(pn+\log\frac{n}{|W|}\Big)\,|W|,$$
in view of Lemma~\ref{l: union of supp}.

Combining the inequalities and taking the minimum over $R=|W|\geq |S|$, we get
\begin{equation}\label{eq: aux pognf,smf}
\frac{1}{2}|\Ael_0|\leq |\mathcal J'|\leq \min\limits_{n\geq R\geq |S|}\Big(C'pn+C'\log\frac{n}{R}\Big)^{k+1}\,R
\end{equation}
for a large enough universal constant $C'>0$.
Let us show that the last relation implies that
$$|S|\geq (k+1)\max\Big(n/\sqrt{L},\sum\limits_{i\in M}|\outneigh^\phi(i)|\Big).$$
Since we condition on event $\Event_{\smallrefer{p: supports}}$, Lemma~\ref{l: union of supp}
and the upper bound $|M|\leq n/\sqrt{L}$
imply that $\sum\limits_{i\in M}|\outneigh^\phi(i)|\leq \bar C \big(pn+\log L\big)n/\sqrt{L}$
for a universal constant $\bar C>0$.
Thus, it is sufficient to show that 
\[
 |S|\geq \bar C(k+1)\big(pn+\log L\big)n/\sqrt{L}.
\]
As $S \neq \emptyset$, 
the last inequality can be false false only when $\bar C(k+1)\big(pn+\log L\big)n/\sqrt{L}\geq 1$.
In this case, choose
 $R_0:=\lfloor \bar C(k+1)\big(pn+\log L\big)n/\sqrt{L}\rfloor$ and observe that
$$\Big(C'pn+C'\log\frac{n}{R_0}\Big)^{k+1}\,R_0\leq (C'''pn+C'''\log L)^{k+2}n/\sqrt{L}.$$
Thus, if the constant $C_{\smallrefer{l: hit infty type}}$ is sufficiently large, we get contradiction to \eqref{eq: aux pognf,smf}.

The result follows.
\end{proof}

\section{Almost null vectors cannot be very sparse}\label{s: very sparse}

In this section, we show that a shifted (very sparse) random matrix $A-z\,\Id$ satisfying the above assumptions
on the distribution of the entries and on the non-random shift,
typically does not have almost null very sparse vectors.
The main statement of the section is Proposition~\ref{prop: no very sparse null vectors}.
The main difficulty in proving the result, compared to the standard setting dealing with dense matrices
as well as sparse matrices with at least logarithmic average number of non-zero elements in rows/columns,
lies in the fact that in the very sparse regime some rows and columns of $A$ have only zero components.
The absence of very sparse null vectors (and, as we show later, non-singularity of $A-z\,\Id$)
is guaranteed by the presence of the non-zero shift $z\,\Id$. Accordingly, the way to study the kernel
of $A-z\,\Id$ is significantly different from the geometric approach to invertibility of dense random matrices.
The random graphs, considered in the previous section,
provide a helpful tool in analyzing the structure of non-zero entries of the matrix $A-z\,\Id$,
taking into account their magnitudes.
In the next subsection, we will consider matrix equivalents of the notions
of a $\phi$-compression, an $M$-shell and vertex types.

\subsection{Compressions, shells and types for matrices}\label{subs: for matrices}

Let $B=(b_{ij})$ be an $m\times k$ matrix with complex entries, and let $K>0$ be a parameter.
We associate with $B$ a graph $G_B\in\grc_{m,k}$ with the edge set defined as follows:
$i\to j$ if and only if $b_{ij}\neq 0$, and $i\leftarrow j$ if and only if $|b_{ij}|\geq 1/\alpha$.
This way, rows of $B$ correspond to left vertices of $G_B$, and columns --- to right ones.
When the matrix $B$ is random, this association generates coupling with a random graph from $\grc_{m,k}$.
In particular, when $B=A-z\,\Id$, with $A$ and $z$ satisfying conditions \eqref{Asmp on p weak}--\eqref{Asmp on A}--\eqref{Asmp on z},
the associated graph $\Gr:=G_B$ satisfies \eqref{Asmp on G} where we take $\mu_{ij}$ as indicators of events
$\{|\xi_{ij}|\geq 1/\alpha\}$.
For every $g\in \N\cup\{\infty\}$, we let $T_{K,g}(B):=T_{K,g}(G_B)$.
We will refer to sets $T_{K,g}(B)$ as {\it column types} of $B$.
The infinite type $T_{K,\infty}(B)$ is of particular importance to us as it corresponds
to a nicely expanding part of the graph.

Further, consider a square $n\times n$ matrix $B$.
Let $\phi$ be any $(G_B,K)$--admissible mapping with $\phi([n])=[m]$ for some $m\leq n$.
We define the $m\times n$ matrix $\phi(B)$ ({\it{}the $\phi$--compression of $B$})
by
$$\row_i(\phi(B)):=\sum\limits_{v\in\phi^{-1}(i)}\row_v(B),\quad i\leq m.$$
The above means that we add rows whose indices are glued together by $\phi$, and have disjoint supports,
in view of the definition of a $(G_B,K)$--admissible mapping.
Note that $T_{K,g}(\phi(B))=T_{K,g}(\phi(G_B))$.
In what follows, such a mapping $\phi$ will be called $(B,K)$--admissible.

We say that a $(B,K)$--admissible mapping $\phi$ (for some $n\times n$ matrix $B$)
is {\it $u$--light} for some $u>0$ if
$$\big|\big\{i\leq n:\;|\phi^{-1}(\phi(i))|=2\big\}\cap \supp\col_j(B)\big|\leq u\;\;\mbox{ for all }j\leq n.$$
Clearly, the notion is consistent with that of a $u$--light mapping for graphs, given in the previous section.

Shells for matrices are defined as shells for the associated graphs.
Specifically, let $B=(b_{ij})$ be a $k\times m$ matrix with complex entries, let $d\geq 1$ be a natural number,
and $M\subset[k]$ be any subset.
We say that a finite sequence $\Am=(\Ael_\ell)_{\ell=0}^{d}$ of subsets of $[m]$ is an
{\it $M$--shell of depth $d$ for $B$}
if for any $0\leq \ell\leq d-1$ and any
$j\in \Ael_\ell$ we have the following: whenever $i\in [k]\setminus M$ is such that $|b_{ij}|\geq 1/\alpha$,
there is an index $j'=j'(i,j)\neq j$ from $\Ael_{\ell+1}$ such that
$b_{ij'}$ is non-zero.
The subset $\Ael_0$ will be called {\it the center} of $\Am$.

\subsection{Matrix shells in non-random setting}\label{s: column structures}

In the next lemma, we relate structural properties of almost null vectors of a matrix to properties of its $M$-shells.
More specifically, we will show that if the coordinates of an almost null vector $x$
are large on some subset of indices $J$ then there exists an $M$-shell $\Am$ centered at $J$
such that $x_i$'s are also large for all $i$ in the first few layers of $\Am$.

\begin{lemma}[Order statistics and $M$-shells]\label{l: order statistics}
Assume that $B=(b_{ij})$ is a $k\times m$ matrix with complex entries.
Further, let $M\subset[k]$, let $x\in \C^m$ be a complex vector, and fix any non-empty $J\subset[m]$ and $d\geq 1$.
Denote
$$L:=\max\limits_{i\in [k]\setminus M}\sum\limits_{j=1}^m |b_{ij}|.$$
Assume that
$$\Big|\sum\limits_{j=1}^m b_{ij}x_j\Big|\leq \frac{1}{2\alpha}\, \big(2\alpha\, L\big)^{-d}\,\min\limits_{j\in J}|x_j|,
\quad i\in[k]\setminus M.$$
Then there exists an $M$--shell $\Am=(\Ael_\ell)_{\ell=0}^d$ of depth $d$ centered in $J$
such that for any $1\leq q\leq d$ with $\Ael_q\neq\emptyset$, we have
$$x^*_{|\Ael_q|}\geq  \big(2\alpha\, L\big)^{-q}\,\min\limits_{j\in J}|x_j|.$$
\end{lemma}
\begin{proof}
We start with the following observation. Let $(i,\ell)\in([k]\setminus M)\times [m]$ and assume that
$$ |b_{i\ell}|\geq 1/\alpha\quad\mbox{and}\quad|x_\ell|\geq \big(2\alpha\, L\big)^{-d}\,\min\limits_{j\in J}|x_j|.$$

Then the upper bound on $|\sum_{j=1}^m b_{ij}x_j|$
implies that there is $h=h(i,\ell)\neq \ell$ such that $b_{ih}\neq 0$
and $|x_h|\geq \frac{1}{2\alpha\, L}|x_\ell|$. Indeed, if it was not the case then
the inner product $\sum_{j=1}^m b_{ij}x_j$ could be estimated as
$$\Big|\sum_{j=1}^m b_{ij}x_j\Big|\geq |x_\ell\,b_{i\ell}|-\sum_{r\neq \ell}|x_r\,b_{i r}|
> |x_\ell\,b_{i\ell}|-L \cdot \frac{1}{2\alpha\, L}|x_\ell|\geq \frac{1}{2\alpha}\big(2\alpha\, L\big)^{-d}
\,\min\limits_{j\in J}|x_j|,$$
leading to contradiction.

Denote by $Q$ the collection of all pairs $(i,\ell)\in([k]\setminus M)\times[m]$ satisfying the assumption above.
If $Q\neq \emptyset$ then the observation above tells us that we can define a mapping $f:Q\to[m]$
taking $f(i,\ell):=h$, with $h=(i,\ell)$ satisfying the aforementioned conditions.

Now, we can construct an $M$-shell $(\Ael_q)_{q=0}^d$ centered in $J$ as follows.
Let $\Ael_1:=\{f(i,\ell):\,(i,\ell)\in Q\cap([k]\times J)\}$; and for any $1\leq q\leq d-1$, let
$$\Ael_{q+1}:=\{f(i,\ell):\,(i,\ell)\in Q\cap([k]\times \Ael_q)\}.$$
Observe that by the construction for any $1\leq q\leq d$ and any $\ell\in \Ael_q$ we have, by induction,
$$|x_\ell|\geq \big(2\alpha\, L\big)^{-q}\,\min\limits_{j\in J}|x_j|.$$
Thus, whenever $(i,\ell)\in ([k]\setminus M)\times \Ael_q$ is such that $|b_{i\ell}|\geq 1/\alpha$,
there is at least one index $\ell'\neq \ell$ with $\ell'\in \Ael_{q+1}$ and $b_{i\ell'}\neq 0$.
Thus, the $M$--shell is well defined.
Finally, assuming that $\Ael_q$ is non-empty, we have
$$x^*_{|\Ael_q|}\geq  \big(2\alpha\, L\big)^{-q}\,\min\limits_{j\in J}|x_j|,$$
and the result follows.
\end{proof}

Clearly, the above lemma gives a non-trivial estimate only when all shells for $B$
are ``expanding'' in the sense that cardinalities of the $q$-th subset of each shell is much greater than $|J|$.
This expansion property is guaranteed by Lemma~\ref{l: expansion for structures}.
Combining it with Lemma~\ref{l: order statistics}, we obtain
\begin{cor}\label{cor: order stat decay}
Let $k,m$, $M\subset[k]$, $K$, $\varepsilon$, $\delta$ and the associated graph $G:=G_B$
be as in Lemma~\ref{l: expansion for structures}
(in particular, $G_B$ satisfies \eqref{eq: col exp prop}).
Set
$$L:=\max\limits_{i\in [k]\setminus M}\sum\limits_{j=1}^m |b_{ij}|.$$
Fix a non-empty subset $J\subset T_{K,\infty}(B)$ with $|J|\leq \delta m/2$ and
$$\frac{2}{K}\sum_{i\in M}|\supp\,\row_i(B)|\leq \frac{|J|}{2}.$$
Let $x$ be a complex vector such that
$$\Big|\sum\limits_{j=1}^m b_{ij}x_j\Big|\leq \frac{1}{2\alpha}\, \big(2\alpha\, L\big)^{-d}\,\min\limits_{j\in J}|x_j|,
\quad i\in[k]\setminus M.$$
Then for any $1\leq q\leq d$ and $k_q:=\min\big(\lfloor\delta m/4\rfloor,(32\varepsilon)^{-q}|J|\big)$ we have
$$x^*_{k_q}\geq  \big(2\alpha\, L\big)^{-q}\min\limits_{j\in J}|x_j|.$$
\end{cor}

Roughly speaking, the last statement tells us that whenever $B$ satisfies certain expansion properties,
any almost null vector of $B$, ``well supported'' on $T_{K,\infty}(B)$, must necessarily be well spread.

\subsection{Order statistics of almost null vectors}\label{subs: very sparse}

In this subsection, all the results on the random graph $\Gr$ and its compressions
obtained in Section~\ref{s: graph section}, come into play.
As in the text before, we define parameter $K_0$ as $K_0:=pn/(2\alpha)$.
By some abuse of terminology, we will say that some event holds for a random $n\times n$ matrix $B$
if that event holds for the associated graph $G_B$.

\begin{lemma}\label{l: treatment of sparse}
There are universal constants $C_{\smallrefer{l: treatment of sparse}},c_{\smallrefer{l: treatment of sparse}}>0$
with the following property.
Let $n$, $p$, $z$ and $A$ satisfy \eqref{Asmp on p weak}--\eqref{Asmp on A}--\eqref{Asmp on z}, and
set $\widetilde A:=A-z\,\Id$.
Fix a realization of $A$ such that $\Event_{\smallrefer{p: ell one norm}}$ occurs for $A$
and event $\Event_{\smallrefer{p: supports}}\cap\,\Event_{\smallrefer{p: chains combined}}\,\cap\,\Event_{\smallrefer{p: expansion}}
(1/(512\alpha))$ occurs for $\widetilde A$.
Let $q$ be in the interval $\{1,2,\dots,\lfloor e^{-c_{\smallrefer{p: self-balancing stat}}pn/2}n\rfloor\}$.
Let $m\leq n$, and let $\phi:[n]\to[m]$ be a
$(\widetilde A,K_0/2)$--admissible $(K_0/256)$--light mapping.
Then for any vector $x\in\C^n$ with
$$\|\phi(\widetilde A)\, x\|_2
\leq \frac{\sqrt{n}}{2\alpha}(2\alpha)^{-C_{\smallrefer{l: treatment of sparse}}\log^2\frac{4n}{q}\,\log^2\log\frac{4n}{q}}\,x^*_q$$
we have
$$x^*_{q}
\leq (2\alpha)^{C_{\smallrefer{l: treatment of sparse}}\log^3\frac{4n}{q}\,\log^2\log\frac{4n}{q}}
x^*_{\lfloor c_{\smallrefer{l: treatment of sparse}}/p\rfloor}.$$
\end{lemma}
\begin{proof}
Observe that the condition on the Euclidean norm of $\phi(\widetilde A)\, x$ implies that
the set $\widetilde M\subset[m]$ of all indices $i$ such that
$$\big|\big\langle \row_i(\phi(\widetilde A)),\bar x\big\rangle\big|
\geq \frac{1}{2\alpha}(2\alpha)^{-\frac{C_{\smallrefer{l: treatment of sparse}}}{2}\log^2\frac{4n}{q}\,\log^2\log\frac{4n}{q}}\,x^*_q,$$
has cardinality
\begin{equation*}\label{aux: 134 alsuhfalk}
|\widetilde M|\leq n (2\alpha)^{-C_{\smallrefer{l: treatment of sparse}}\log^2\frac{4n}{q}\,\log^2\log\frac{4n}{q}}.
\end{equation*}
Denote $B:=\phi(\widetilde A)$ and let $k:=\lceil \frac{1}{cpn}\log\frac{4n}{q}\rceil$,
$L:=\big(k\log\frac{4n}{q}\big)^{Ck}\big(\frac{n}{q}\big)^2$,
and $d:=\lceil C\log L\rceil$, where $c>0$ is small and $C>0$
is large enough constant, whose values can be recovered from the proof below.
Let $M'\subset[m]$ be the set of all indices $i$ such that
$\|\row_i(B)\|_1\geq L$, and set $M:=\widetilde M\cup M'$.
On event $\Event_{\smallrefer{p: ell one norm}}$, we have
$|M'|\leq n/(2\sqrt{L})$,
and so $|M|\leq n/\sqrt{L}$.

Let $\Am=(\Ael_\ell)_{\ell=0}^{d}$ be the $M$--shell for $B$ centered in the set
$\Ael_0:=\{i\leq n:\;|x_i|\geq x^*_q\}$, constructed in Lemma~\ref{l: order statistics}.
By the assumption on $q$ and the definition of $L$ and $k$ we have
\begin{align*}
&\big(C_{\smallrefer{l: hit infty type}}pn+
C_{\smallrefer{l: hit infty type}}\log L\big)^{k+2} n/\sqrt{L}\\
&\leq
\Big(C_{\smallrefer{l: hit infty type}}pn+
2C_{\smallrefer{l: hit infty type}}Ck\log\log\frac{4n}{q}+2C_{\smallrefer{l: hit infty type}}\log\frac{n}{q}
\Big)^{k+2} \Big(k\log\frac{4n}{q}\Big)^{-Ck/2} q
\leq \frac{q}{2},
\end{align*}
if $C$ is sufficiently large.
Also,
$$C_{\smallrefer{l: hit infty type}}\, e^{-c_{\smallrefer{p: self-balancing stat}}pnk}n\leq \frac{q}{2}$$
if $c$ is small enough. Thus, the cardinality of $\Ael_0$ satisfies assumptions in Lemma~\ref{l: hit infty type}.
Applying Lemma~\ref{l: hit infty type}, we get
$$\Big|\bigcup_{\ell=0}^{k}\Ael_\ell\,\cap T_{K_0/2,\infty}
(B)\Big|\geq (k+1)\max\Big(n/\sqrt{L},\sum\limits_{i\in M}|\supp\row_i(B)|\Big),
$$
and hence there is $\ell_0\in\{0,1,\dots,k\}$, such that
$$|\Ael_{\ell_0}\,\cap T_{K_0/2,\infty}
(B)|\geq \max\Big(n/\sqrt{L},\sum\limits_{i\in M}|\supp\row_i(B)|\Big),
$$
implying
$$\frac{4}{K_0}
\sum\limits_{i\in M}|\supp\row_i(B)|
\leq \frac{|\Ael_{\ell_0}\,\cap T_{K_0/2,\infty}(B)|}{2}.$$
Since we are on the event $\Event_{\smallrefer{p: expansion}}(1/(512\alpha))$,
the graph $G_{\widetilde A}$ satisfies
$$
\big|\inneigh(I)\big|
\geq \sum_{i\in I}\big|\inneigh(i)\big|- \kappa K\,|I|
\;\mbox{for every set of right vertices $I$,\quad$2\leq |I|\leq \delta n$},
$$
with $\kappa:=1/128$, $K:=K_0/2$, and $\delta:=c_1/(pn)$ (for some constant $c_1>0$ depending on $\alpha$).
Next, we use the assumption that the mapping $\phi$ is $(K_0/256)$--light.
Applying Lemma~\ref{l: phi expansion}, we get for the graph $G_B=\phi(G_{\widetilde A})$:
$$\big|\inneigh^\phi(I)\big|\geq \sum_{i\in I}\big|\inneigh^\phi(i)\big|-\varepsilon K |I|\;\;
\mbox{for any subset of right vertices $I$ with $|I|\leq \delta n$},$$
with $\varepsilon:=1/64$.
Applying Lemma~\ref{l: expansion for structures} with $J:=\Ael_{\ell_0}\,\cap T_{K_0/2,\infty}(B)$
and using that $d\geq Ck$ and $\ell_0\leq k$, we obtain
$$|\Ael_{d}|\geq \min\big(\lfloor\delta m/4\rfloor,2^{d-\ell_0}|J|\big)
\geq\min\big(\lfloor\delta m/4\rfloor,2^{C\log L/2}|J|\big)
\geq  c_2/p.$$
Thus, using Lemma~\ref{l: order statistics}, we get
$$x^*_{\lfloor c_2/p\rfloor}\geq (2\alpha\, L)^{-d}x^*_q,$$
and the result follows.
\end{proof}

In the next proposition, we extend
Lemma~\ref{l: treatment of sparse}
to all $q\leq c/p$.
\begin{prop}\label{prop: no very sparse null vectors}
There are universal constants $C_{\smallrefer{prop: no very sparse null vectors}},c_{\smallrefer{prop: no very sparse null vectors}}>0$
with the following property.
Let $n$, $p$, $z$ and $A$ satisfy \eqref{Asmp on p weak}--\eqref{Asmp on A}--\eqref{Asmp on z};
set $\widetilde A:=A-z\,\Id$.
Assume that event $\Event_{\smallrefer{p: supports}}\cap\,\Event_{\smallrefer{p: chains combined}}\,\cap\,\Event_{\smallrefer{p: expansion}}
(1/(512\alpha))$ occurs for $\widetilde A$ and event $\Event_{\smallrefer{p: ell one norm}}$ occurs for $A$.
Let $q$ be in the interval $\{1,2,\dots,\lfloor c_{\smallrefer{prop: no very sparse null vectors}}/p\rfloor\}$.
Let $m\leq n$, and let $\phi:[n]\to[m]$ be a
$(\widetilde A,K_0/2)$--admissible $(K_0/256)$--light mapping.
Then for any vector $x\in\C^n$ with
$$\|\phi(\widetilde A)\, x\|_2
\leq \frac{\sqrt{n}}{2\alpha}(2\alpha)^{-C_{\smallrefer{prop: no very sparse null vectors}}\log^2\frac{4n}{q}\,\log^2(pn+\log\frac{4n}{q})}\,x^*_q$$
we have
$$x^*_{q}
\leq (2\alpha)^{C_{\smallrefer{prop: no very sparse null vectors}}\log^2\frac{4n}{q}\,\log^2(pn+\log\frac{4n}{q})}
x^*_{\lfloor c_{\smallrefer{prop: no very sparse null vectors}}/p\rfloor}.$$
\end{prop}
\begin{proof}
We will choose constant $C_{\smallrefer{prop: no very sparse null vectors}}>0$ large
and $c_{\smallrefer{prop: no very sparse null vectors}}>0$ small enough (the precise relation can be recovered from the argument below).
In the range $q\in\{1,2,\dots,\lfloor e^{-c_{\smallrefer{p: self-balancing stat}}pn/2}n\rfloor\}$ the
statement is proved above (Lemma~\ref{l: treatment of sparse}).
When $q\geq e^{-c_{\smallrefer{p: self-balancing stat}}pn/2}n$, we have on $\Event_{\smallrefer{p: chains combined}}$
that the subset
$$J:=\big\{i\leq n:\;|x_i|\geq x^*_q\big\}\,\cap\,T_{K_0,\infty}(\phi(\widetilde A))$$
has cardinality at least $q/2$.
Set $L:=C(pn)^2(n/q)^2$ (for a large enough constant $C>0$).
As in the proof of Lemma~\ref{l: treatment of sparse}, we define two subsets $\widetilde M,M'\subset[m]$:
$\widetilde M$ is the set of all indices $i$ such that
$$\big|\big\langle \row_i(\phi(\widetilde A)),\bar x\big\rangle\big|
\geq \frac{1}{2\alpha}(2\alpha)^{-\frac{C_{\smallrefer{prop: no very sparse null vectors}}}{2}\log^2\frac{4n}{q}\,\log^2(pn+\log\frac{4n}{q})}\,x^*_q,$$
and $M'$ the set of all indices $i$ such that $\|\row_i(\phi(\widetilde A))\|_1\geq L$. Define $M:=\widetilde M\cup M'$.
On the event $\Event_{\smallrefer{p: ell one norm}}\cap \Event_{\smallrefer{p: supports}}$ we have $|M'|\leq n/\sqrt{L}$.
Therefore, by Lemma~\ref{l: union of supp} and since $q\geq e^{-c_{\smallrefer{p: self-balancing stat}}pn/2}n$,
we obtain $\sum_{i\in M'}|\supp\,\row_i(\phi(\widetilde A))|\leq C' pn\cdot n/\sqrt{L}$.
Thus,
$$\frac{2}{K_0}\sum_{i\in M'}|\supp\,\row_i(\phi(\widetilde A))|\leq \frac{|J|}{4}.$$
Similarly to the proof of Lemma~\ref{l: treatment of sparse}, we can estimate the cardinality of $\widetilde M$ as
$$|\widetilde M|\leq n (2\alpha)^{-C_{\smallrefer{prop: no very sparse null vectors}}
\log^2\frac{4n}{q}\,\log^2(pn+\log\frac{4n}{q})}.$$
Proceeding as above, we get
$$\frac{2}{K_0}\sum_{i\in M}|\supp\,\row_i(\phi(\widetilde A))|\leq \frac{|J|}{2}.$$
As in the proof of the above lemma, observe that
on event $\Event_{\smallrefer{p: expansion}}(1/(512\alpha))$, in view of Lemma~\ref{l: phi expansion},
the matrix $\phi(\widetilde A)$ satisfies
\eqref{eq: col exp prop} with $\varepsilon:=1/64$, $K:=K_0$ and $\delta:=c_1/p$ (for some universal constant $c_1>0$).
Then, applying Corollary~\ref{cor: order stat decay} with $d:=\lfloor C'\log \frac{2n}{q}\rfloor$ (for an appropriate constant $C'>0$), we get
the required estimate.
\end{proof}
As an immediate corollary, we get the following statement:
\begin{cor}\label{cor: no very sparse for smin}
Let $n$, $p$, $A$, $z$ satisfy assumptions \eqref{Asmp on p weak}--\eqref{Asmp on A}--\eqref{Asmp on z}.
Fix a realization of $\widetilde A:=A-z\,\Id$
such that events $\Event_{\smallrefer{p: supports}}\cap\,
\Event_{\smallrefer{p: chains combined}}\,\cap\,\Event_{\smallrefer{p: expansion}}
(1/(512\alpha))$ and $\Event_{\smallrefer{p: ell one norm}}$ occur.
Then
any vector $x\in\C^n$ such that
$\|\widetilde A x\|_\infty\leq (2\alpha)^{-C_{\smallrefer{cor: no very sparse for smin}}\log^3 n}\|x\|_{\infty}$,
satisfies
$$\|x\|_\infty\leq (2\alpha)^{C_{\smallrefer{cor: no very sparse for smin}}\log^3 n}\,
x_{\lfloor c_{\smallrefer{cor: no very sparse for smin}}/p\rfloor}^*,$$
and, moreover, for all $q$ in $\{1,2,\dots,\lfloor c_{\smallrefer{cor: no very sparse for smin}}/p\rfloor\}$
we have
$$x^*_{q}
\leq (2\alpha)^{C_{\smallrefer{cor: no very sparse for smin}}\log^2\frac{4n}{q}\log^2(pn+\log\frac{4n}{q})}
x^*_{\lfloor c_{\smallrefer{cor: no very sparse for smin}}/p\rfloor}.$$
\end{cor}

\section{Almost null vectors cannot be moderately sparse}\label{s: moderately sparse}

In this section, we extend the results of Section~\ref{subs: very sparse} showing that the the matrix
$\phi(\widetilde{A})$ typically does not have almost null vectors which are close to
$n/\log (pn)$--sparse.
By the results of the previous section, it is enough to consider vectors $x \in S^{n-1}(\C)$
having sufficiently large $x^*_m$ for $m\approx p^{-1}$.
Unlike the treatment of very sparse vectors, which relied on the properties of the graph associated to the matrix,
the analysis of the moderately sparse vectors uses $\e$-nets.
However, the standard $\e$-net argument cannot be applied here as the operator norm of the matrix $\phi(\widetilde{A})$ is too large.
Instead, we will analyze each inner product of a row of $\phi(\widetilde{A})$
and $x$ separately. In this analysis, in contrast with the dense matrices,
the approximation in $\ell_{\infty}$ norm works better than that in $\ell_2$ norm.
We note here that several versions of the $\e$--net argument have been developed recently to deal with sparse random matrices;
see, in particular, \cite{LLTTY JMAA, Cook circ, BCZ, BR inv, BR circ, LLTTY smin, LLTTY huge}.
The argument presented here differs considerably from those works.

To approximate any moderately sparse vector in $\ell_{\infty}$ norm, one can consider a covering of $S^{n-1}(\C)$ by cubes. This covering, however, is too large to be combined with a small ball probability estimate, which is rather weak due to the sparsity of the matrix. Yet, the sparsity, being an obstacle, can be turned into an advantage. For a fixed vector $x$ any given row typically has zero entries in the spots corresponding a few largest coordinates of $x$. If this occurs, the largest coordinates of $x$ do not have to be approximated, which allows to reduce the cardinality of the net. We implement this program below.

For a vector $x\in\C^n$ and a number $r\geq 1$, denote by $\Max_r(x)$ an $\lfloor r\rfloor$-element subset of $[n]$ containing coordinates
of $x$ with largest absolute values (the ties are broken arbitrarily).
Let $n/2 \le m \le n$. We will consider a sparse $m \times n$ matrix $\phi(\widetilde A)$,
where $\phi:[n]\to[m]$ is a $(K_0/2,\widetilde A)$--admissible surjective mapping and, as before, $\widetilde A=A-z\Id_n$.

We start with showing that for a fixed vector $x$, the probability that a given row of the matrix $\widetilde A$
has a large product with $x$ and the entries corresponding to the largest coordinates of $x$ are zeroes, is non-negligible.

\begin{lemma} \label{l: Levy concentration}
Let $n,p,z,A$ satisfy \eqref{Asmp on p weak}---\eqref{Asmp on A}---\eqref{Asmp on z}, and let $\widetilde A=(\widetilde a_{ij}):=A-z\,\Id_n$.
Let $\tau \in (0,1)$ be a parameter and $q$ be an integer with $\tau p^{-1} \le q \le p^{-1}$.
Fix $x \in S^{n-1}(\C)$.
For any $i\notin \Max_{q}(x)$, consider the event
\begin{align*}
\Omega_{x}^i
=\Big\{\sum_{j=1}^n |\widetilde a_{ij}| \le C_{\smallrefer{l: Levy concentration}}pn \;\; \& \;\;
\widetilde a_{ij}=0\;\mbox{for all }j \in \Max_{q/2}(x) \;\; \& \;\;
|\langle\row_i(\widetilde A),\bar x\rangle|\geq \frac{1}{2 \a} x_q^*  \Big\}.
\end{align*}
Then
\[
\Prob(\Omega_{x}^i) \ge c_{\smallrefer{l: Levy concentration}}
\]
for some $c_{\smallrefer{l: Levy concentration}}=c_{\smallrefer{l: Levy concentration}}(\tau)>0$ depending only on $\tau$.
\end{lemma}
\begin{proof}
Let $U$ be the event that $\widetilde a_{ij}=0$ for all $j \in \Max_{q/2}(x)$.
Then, clearly, $\Prob(U) \ge c_1$, where $c_1$ is an absolute constant.

Next, condition on any realization of $\widetilde a_{ij}$, $j \in ([n] \setminus \Max_q(x))\cup\Max_{q/2}(x)$, and set
\[
y:=\sum_{j \in\Max_q(x)\setminus\Max_{q/2}(x)} \widetilde a_{ij} x_j.
\]
Let $\Omega_k^i(x)$ be the event that exactly $k$ of the entries $a_{ij}$,
$j \in \Max_q(x) \setminus \Max_{q/2}(x)$, have absolute value greater or equal to $1/\a$, and the other entries are zero.
Clearly, $y=0$ everywhere on $\Omega_0^i(x)$, and $|y|\geq \frac{1}{\alpha} x^*_q$ everywhere
on $\Omega_1^i(x)$. Further, by the conditions on $q$, we have
$\P( \Omega_0^i(x)), \P(\Omega_1^i(x)) \ge c(\tau)$ for some $c(\tau)>0$.
Together with the above observation on the probability of the event $U$, this gives
$$\Prob\Big\{\widetilde a_{ij}=0\;\mbox{for all }j \in \Max_{q/2}(x) \;\; \& \;\;
|\langle\row_i(\widetilde A),\bar x\rangle|\geq \frac{1}{2\a} x_q^*  \Big\}\geq c'(\tau)$$
for some $c'(\tau)>0$.
Finally, observe that
as the $\row_i(\widetilde A)$ contains one entry which is shifted by $z$ and $|z| \le pn$, by Markov's inequality there is $C_1(\tau)>0$ such that
 \[
   \Prob\Big\{\sum_{j=1}^n |\widetilde a_{ij}|>C_1 pn\Big\} \le \frac{pn\, \E |\xi_{ij}|+pn}{C_1 pn} \le  \frac{c'(\tau)}{2}.
 \]
The combination of the last two probability estimates yields the lemma.
\end{proof}

To pass from a single coordinate of $\widetilde Ax$ to bounds for its norm,
we need the following elementary lemma showing that with overwhelming probability, the number of good rows is large.
\begin{lemma}\label{l: many rows}
For any $\tau\in(0,1)$ there are $c_{\smallrefer{l: many rows}},c_{\smallrefer{l: many rows}}'>0$ depending only on $\tau$ with the following property.
Let $x$ and $n,p,z,A,q$ be as in Lemma~\ref{l: Levy concentration}.
Denote
\[
  S(x):= \sum\limits_{i\in[n]\setminus \Max_{q}(x)}\indicator_{\Omega_x^i},
\]
where $\indicator_{\Omega_x^i}$ is the indicator of the event $\Omega_x^i$.
Set
 \[
  \Omega_x:=\big\{|S(x)| \ge c_{\smallrefer{l: many rows}} n \big\}.
 \]
 Then
 \[
  \P(\Omega_x) \ge 1- \exp(-c'_{\smallrefer{l: many rows}}  n).
 \]
\end{lemma}
\begin{proof}
Since the events $\Omega_{x}^i$ for different $i$ are independent,
Bernstein's inequality and Lemma~\ref{l: Levy concentration} imply this bound.
\end{proof}

The next proposition is the main step toward proving the result of this section.
It asserts that if $X$ is an almost null vector for $\phi(\widetilde A)$ and
$x^*_{m/2}$ and $x^*_m$ are commensurate,
then typically $x^*_M$ is also commensurate with $x^*_m$ for $M$ almost proportional to $n$.

\begin{prop} \label{prop: almost proportional}
Let $\tau\in (0,1)$, let $n,p,z,A,q$ be as in Lemma~\ref{l: Levy concentration}, and let
\[
  \tau p^{-1} \le m \le p^{-1}.
\]
There exist positive constants $C_{\smallrefer{prop: almost proportional}}, \tilde{C}_{\smallrefer{prop: almost proportional}},
c_{\smallrefer{prop: almost proportional}}, \tilde{c}_{\smallrefer{prop: almost proportional}}, \hat{c}_{\smallrefer{prop: almost proportional}}$
depending on $\tau$ and $\alpha$ with the following property.
Let
  \[
   M:=\lfloor\tilde{c}_{\smallrefer{prop: almost proportional}} n/\log (np)\rfloor,
  \]
and define
\begin{align*}
\Event_{\smallrefer{prop: almost proportional}}:=\Big\{&\forall \,\phi:[n]\to\N\mbox{ with }
|\{i\leq n:\;|\phi^{-1}(\phi(i))|\geq 2\}|\leq
c_{\smallrefer{prop: almost proportional}} n\mbox{ and }\\
&\forall\,x \in S^{n-1}(\C)
\mbox{ such that }x_m^* > e^{-\tilde{c}_{\smallrefer{prop: almost proportional}} pn} x^*_{\lfloor m/2\rfloor}
\;\; \& \;\; x_M^* \le \frac{c_{\smallrefer{prop: almost proportional}}}{pn}  x_m^*\\
&\mbox{we have } \;\; \|\phi(\widetilde A) x\|_2 > \hat{c}_{\smallrefer{prop: almost proportional}} \sqrt{n} x_m^* \Big\}.
\end{align*}
Then $\Prob(\Event_{\smallrefer{prop: almost proportional}})\geq 1-\exp(- C_{\smallrefer{prop: almost proportional}} n)$.
\end{prop}

\begin{proof}
Let
\[
\e:= \frac{c_{\smallrefer{prop: almost proportional}} }{pn},
\]
 where the constant $c_{\smallrefer{prop: almost proportional}}$ will be chosen later. Denote
\[
V:=\left \{x \in S^{n-1}(\C):\; x_{m}^* > e^{- \tilde{c}_{\smallrefer{prop: almost proportional}} pn}x_{\lfloor m/2\rfloor}^* \;\; \& \;\;
x_M^* \le \e \cdot x_m^*   \right \}.
\]
and define a function $f:V \to \C^n$ by
\[
f(x):=\frac{\Proj_{[n]\setminus \Max_{m/2}(x)}(x)}{x_m^*},
\]
where $\Proj_{[n]\setminus \Max_{m/2}(x)}$ denotes the coordinate projection
onto $[n]\setminus \Max_{m/2}(x)$.
Define also
\[
W:=\left \{f(x): \ x \in V \right \}.
\]
The proof uses an $\e$-net in the set $W$ in the $\ell_\infty$ metric.
Note that for every $x \in V$
 \[
   \norm{f(x)}_{\infty} \le H:=\exp(\tilde{c}_{\smallrefer{prop: almost proportional}} pn), \quad f(x)_{\lceil m/2\rceil+1}^* \le 1
   \quad \text{and} \quad f(x)_{M-\lfloor m/2\rfloor+1}^* \le \e.
 \]
To construct the net, we first choose an $M$-element subset of $[n]$ corresponding to $\Max_M(x)$
and a further $\lfloor m/2\rfloor$-subset corresponding to $\Max_{m/2}(x)$.
After these sets are chosen, we cover the cube of size $H$ of (complex) dimension $\lceil m/2\rceil$
and the unit cube of dimension $M-m$ by cubes of size $\e$.
This allows to construct an $\e$-net $\NN' \subset W$ with cardinality
\begin{align*}
|\NN'| &\le \binom{n}{M-\lfloor m/2\rfloor} \binom{M-\lfloor m/2\rfloor}{\lceil m/2\rceil}
\left( \frac{3}{\e} \right)^{2(M-m)}  \left( \frac{3 H}{\e} \right)^{m+1} \\
&\le \exp \left( M \left[ \log \left( \frac{en}{M} \right) + 2\log \left( \frac{3}{\e} \right ) \right]
+ (m+1) \left[ \log \left( \frac{eM}{m/2} \right) + \log \left( H \right ) \right]
\right).
\end{align*}
Here,
\[
M \log \left( \frac{en}{M} \right)
\leq \tilde{c}_{\smallrefer{prop: almost proportional}} \frac{n}{\log (pn)}
\log \left( e \log \frac{pn}{\tilde{c}_{\smallrefer{prop: almost proportional}}} \right) \le\tilde{c}_{\smallrefer{prop: almost proportional}} n
\]
and
\[
M \log \left( \frac{3}{\e} \right )
= \tilde{c}_{\smallrefer{prop: almost proportional}} \frac{n}{\log (pn)}  \log (3 c_{\smallrefer{prop: almost proportional}}^{-1} pn)
\le C_1 \tilde{c}_{\smallrefer{prop: almost proportional}} n
\]
for some constant $C_1$ which does not depend on $\tilde{c}_{\smallrefer{prop: almost proportional}}$
as long as ${c}_{\smallrefer{prop: almost proportional}}^{-1} \le pn$.
Also, since $\tau p^{-1}\le m \le p^{-1}$, we have
\[
(m+1) \left[ \log \left( \frac{eM}{m/2} \right) + \log \left( H \right ) \right]
\le 2p^{-1} \left[ \log( C_2 \tilde{c}_{\smallrefer{prop: almost proportional}} pn)+ \tilde{c}_{\smallrefer{prop: almost proportional}} pn \right]
   \le C_3  \tilde{c}_{\smallrefer{prop: almost proportional}} n
  \]
  where $C_2,C_3$ do not depend on $\tilde{c}_{\smallrefer{prop: almost proportional}}$.
  This allows us to conclude that
  \[
  |\NN'|
  \le \exp \left(  C_4 \tilde{c}_{\smallrefer{prop: almost proportional}} n \right).
  \]
 We will use a modification of this net to approximate the $n-m/2$ smallest coordinates of a vector.
 Let us construct this modification.

For every $u \in \NN'$ and every $I \subset [n], \ |I|=\lfloor m/2\rfloor$, pick a vector $x \in V$
such that $\Max_{m/2}(x)=I$ and
$\norm{f(x)-u}_{\infty} \le \e$. If such $x$ does not exist for a given $I$, we skip this $I$.
If such $x$ does not exist for any $I$, we skip $u$. This process creates a set $\NN \subset V$ such that
 \[
 |\NN| \le \binom{n}{\lfloor m/2\rfloor} \cdot |\NN'|
 \le \left(\frac{en}{\lfloor m/2\rfloor} \right)^{\lfloor m/2\rfloor} \cdot  \exp \left(  C_4 \tilde{c}_{\smallrefer{prop: almost proportional}} n \right)
 \le  \exp \left(  C_5 \tilde{c}_{\smallrefer{prop: almost proportional}} n \right)
 \]
  where  $C_5$ does not depend on $\tilde{c}_{\smallrefer{prop: almost proportional}}$.
  By construction, for any $y \in V$, there exists an $x \in \NN$ with
 \[
  \Max_{m/2}(x)=\Max_{m/2}(y) \quad \text{and} \quad \norm{f(y)-f(x)}_{\infty} \leq 2\e.
 \]
Assume that the event $\Omega=\bigcap_{x \in \NN} \Omega_x$ occurs.
Take any $y \in V$ and choose $x \in \NN$ satisfying the condition above.
Since $\Max_{m/2}(x)=\Max_{m/2}(y)$,  for any $i$ such that $\Omega_x^i$ holds, we have
 \begin{align*}
|\langle\row_i(\widetilde A), y\rangle|
&= |\langle\row_i(\widetilde A), \Proj_{[n]\setminus \Max_{m/2}(y)}(y)\rangle|\\
&= |\langle\row_i(\widetilde A), f(y)\rangle| \cdot y_m^*
\ge \left( |\langle\row_i(\widetilde A), f(x)\rangle|  - |\langle\row_i(\widetilde A),f(x)-f(y)\rangle| \right) \cdot y_m^* \\
&\ge \Big( \frac{1}{2 \a}  - \norm{f(x)-f(y)}_{\infty} \cdot \sum_{j=1}^n |\widetilde a_{ij}| \Big) \cdot y_m^*
\ge \Big(  \frac{1}{2 \a} - 2\e \cdot C_{\smallrefer{l: Levy concentration}} pn  \Big) \cdot y_m^*\\
&\ge  \frac{1}{4 \a}   \cdot y_m^*,
\end{align*}
if the constant $c_{\smallrefer{prop: almost proportional}}$ appearing in the definition of $\e$ is chosen sufficiently small.
Since $S(x) \ge c_{\smallrefer{l: many rows}} n$ on $\Omega$, and since $|\{i\leq n:\;|\phi^{-1}(\phi(i))|\geq 2\}|\leq
c_{\smallrefer{prop: almost proportional}}n$, this implies that
 \[
  \norm{\phi(\widetilde A) y}_2  \ge  \hat{c}_{\smallrefer{prop: almost proportional}} \sqrt{n} y_m^*
 \]
for an appropriately chosen $ \hat{c}_{\smallrefer{prop: almost proportional}}$.
We proved that if $\Omega$ occurs, then the event $\Event_{\smallrefer{prop: almost proportional}}$ does not occur.
It remains to estimate the probability of $\Omega^c$.
By Lemma~\ref{l: many rows}, we have
\begin{align*}
\P(\Omega^c) \le \sum_{x \in \NN} \P(\Omega_x^c) \le |\NN| \cdot \exp(-c'_{\smallrefer{l: many rows}} n)
\le  \exp \left(   C_5 \tilde{c}_{\smallrefer{prop: almost proportional}} n  - c'_{\smallrefer{l: many rows}} n\right)
\le \exp(-(c'_{\smallrefer{l: many rows}}/2)  n)
 \end{align*}
if the constant $\tilde{c}_{\smallrefer{prop: almost proportional}}$ is chosen appropriately small.
This finishes the proof of the proposition.
\end{proof}

Now, we combine Proposition~\ref{prop: almost proportional} and Proposition \ref{prop: no very sparse null vectors}
to derive the main result of this section.
It will be convenient for us to define a single event which encapsulates all good properties
of the matrix $\widetilde A$ and the associated graph $G_{\widetilde A}$.
Set
$\Event:=\EE_{\smallrefer{p: supports}}\cap\,
\EE_{\smallrefer{p: chains combined}}\,\cap\,\EE_{\smallrefer{p: expansion}}(1/512\alpha) \cap \EE_{\smallrefer{prop: almost proportional}}$,
and let $\Event_{good}$ be the event
$$\Event_{good}:=\big\{\mbox{$\Event$ occurs for both $\widetilde A$ and $\widetilde A^\top$}\big\}.$$
Note that the results we have proved up to now show that
$\Prob(\Event_{good})\geq 1-(pn)^{-c}$ for an absolute constant $c>0$.

\begin{prop} \label{prop: almost proportional balance}
Let $n$, $p$, $z$ and the matrix $A$
satisfy assumptions \eqref{Asmp on p weak}--\eqref{Asmp on A}--\eqref{Asmp on z}.
Fix a realization of $A$ in $\Event_{good}$.
Let $q$ be in the interval $\{1,2,\dots,\lfloor c_{\smallrefer{prop: almost proportional balance}}/p\rfloor\}$.
Let $m\leq n$, and let $\phi:[n]\to[m]$ be a
$(\widetilde A,K_0/2)$--admissible $(K_0/256)$--light mapping.
Set
\[
M=\lfloor c_{\smallrefer{prop: almost proportional balance}}n/\log (np)\rfloor.
\]
Then for any vector $x\in\C^n$ with
$$\|\phi(\widetilde A)\, x\|_2
\leq \frac{\sqrt{n}}{2\alpha}(2\alpha)^{-C_{\smallrefer{prop: almost proportional balance}}\log^2\frac{4n}{q}\,\log^2(pn+\log\frac{4n}{q})}\,x^*_q$$
we have
$$x^*_{q}
\leq (2\alpha)^{C_{\smallrefer{prop: almost proportional balance}}\log^2\frac{4n}{q}\,\log^2(pn+\log\frac{4n}{q})}
x^*_{M}.$$
\end{prop}

Combining Corollary~\ref{cor: no very sparse for smin} and Proposition~\ref{prop: almost proportional},
we obtain
\begin{cor}\label{cor: no moderately sparse for smin}
Let $n$, $p$, $A$, $z$ satisfy assumptions \eqref{Asmp on p weak}--\eqref{Asmp on A}--\eqref{Asmp on z}.
Fix a realization of $\widetilde A$ in $\Event_{good}$.
Set
\[
M=\lfloor c_{\smallrefer{cor: no moderately sparse for smin}}n/\log (pn)\rfloor.
\]
Then
any vector $x\in\C^n$ such that
$\|\widetilde A x\|_2\leq (2\alpha)^{-C_{\smallrefer{cor: no moderately sparse for smin}}\log^3 n}\|x\|_{\infty}$,
satisfies
$$\|x\|_\infty\leq (2\alpha)^{C_{\smallrefer{cor: no moderately sparse for smin}}\log^3 n}\,
x_{M}^*.$$
\end{cor}

\section{The smallest singular value} \label{sec: smallest}

In this short section, we establish one of the main results of the paper, namely,
the lower bound on the smallest singular value.
Sections~\ref{s: very sparse} and~\ref{s: moderately sparse}
provide a probabilistic lower bound on $\|\widetilde{A}x\|_2$ for sparse vectors $x$. The methods used there cannot however be extended to spread vectors. A method for treating these vectors suggested in \cite{RV invertibility} was used to derive a lower bound on the smallest singular value.
If we know that a certain coordinate of $x$, say $x_1$, has a large absolute value, then we can use the orthogonal projection $P_1$ onto the space $H_1=\spn\big\{\col_2(\widetilde{A}),\dots,\col_n(\widetilde{A})\big\}^{\perp}$ to bound $\|\widetilde{A}x\|_2$ from below:
\[
 \|\widetilde{A}x\|_2 \ge \|P_1 \widetilde{A}x\|_2= |x_1| \cdot \|P_1 \col_1(\widetilde{A})\|_2.
\]
The quantity $\|P_1 \col_1(\widetilde{A})\|_2$ can in turn be estimated below by $|\langle \nu_1,\col_1(\widetilde A)\rangle|$,
where $\nu_1$ is a unit vector orthogonal to $\col_2(\widetilde{A}),\dots,\col_n(\widetilde{A})$.
This estimate provides the desired lower bound for all vectors $x$ with a sufficiently large first coordinate. If we don't know which coordinate of $x$ is large, but know that many of them are, we can construct a probabilistic version of this estimate by choosing a coordinate uniformly at random.
We implement this idea below.
\begin{theor}[Bound for $s_{\min}$]\label{th: bound on smin}
Let $n$, $p$, $A$, $z$ satisfy assumptions \eqref{Asmp on p weak}--\eqref{Asmp on A}--\eqref{Asmp on z},
and, as before, let $\widetilde A:=A-z\,\Id$.
Then
$$\Prob\big\{s_{\min}(\widetilde A)\leq (2\alpha)^{-C_{\smallrefer{th: bound on smin}}\log^3 n}\big\}
\leq (pn)^{-c_{\smallrefer{th: bound on smin}}},$$
where $c_{\smallrefer{th: bound on smin}}>0$ is a universal constant.
\end{theor}
\begin{proof}
For any $j\leq n$, let $\nu_j$ be a random unit normal vector to the linear span of columns $\col_u(\widetilde A)$, $u\neq j$
(of course, $\nu_j$ is not uniquely defined). We will assume that $\nu_j$ is measurable with respect to the $\sigma$--algebra
generated by the columns $\col_u(\widetilde A)$, $u\neq j$, that is, $\nu_j$ and $\col_j(\widetilde A)$ are independent for each $j$.
Further, let $X_{\min}$ be a normalized right singular vector corresponding to the smallest singular
value of $\widetilde A$.

Denote
$$\Event_{smin}:=\Event_{good}
\cap\big\{s_{\min}(\widetilde A)\leq (2\alpha)^{-2C_{\smallrefer{cor: no moderately sparse for smin}}\log^3 n}/(\alpha n)\big\}.$$
and
$$\Event_j:=\Event_{good}\cap\big\{|\langle \nu_j,\col_j(\widetilde A)\rangle|\leq (2\alpha)^{-C_{\smallrefer{cor: no moderately sparse for smin}}\log^3 n}
/(\alpha\sqrt{n})\big\},\quad j=1,2,\dots,n.$$
Observe that Corollary~\ref{cor: no moderately sparse for smin} yields
$$\Event_{smin}\subset \big\{\|X_{\min}\|_\infty\leq (2\alpha)^{C_{\smallrefer{cor: no moderately sparse for smin}}\log^3 n}\,
(X_{\min})_{M}^*\big\},$$
where
$$M:=\lfloor c_{\smallrefer{cor: no moderately sparse for smin}}n/\log (pn)\rfloor,$$
as in Corollary~\ref{cor: no moderately sparse for smin}.
The last relation, in combination with $\|X_{\min}\|_2=1$, implies that within the event $\Event_{smin}$, at least $M$ coordinates of $X_{\min}$
are greater than $(2\alpha)^{-C_{\smallrefer{cor: no moderately sparse for smin}}\log^3 n}/\sqrt{n}$ by absolute value.
On each of those coordinates, we have
$$s_{\min}(\widetilde A)=\|\widetilde A X_{\min}\|_2\geq |\nu_j^\top \widetilde A X_{\min}|
\geq (2\alpha)^{-C_{\smallrefer{cor: no moderately sparse for smin}}\log^3 n}
|\langle \nu_j,\col_j(\widetilde A)\rangle|/\sqrt{n},$$
hence,
$$|\langle \nu_j,\col_j(\widetilde A)\rangle|\leq (2\alpha)^{-C_{\smallrefer{cor: no moderately sparse for smin}}\log^3 n}/(\alpha\sqrt{n}).$$
Thus, for any $\omega\in \Event_{smin}$ there are at least $M$ indices $j$ such that $\omega\in\Event_j$.
Equivalently, we can write
\begin{equation}\label{eq: aux apsofianpfka}
\sum\limits_{j=1}^n \indicator_{\Event_j}\geq M\quad \mbox{everywhere on }\Event_{smin}.
\end{equation}
As the final step of the proof, observe that for each $j\leq n$:
\begin{align*}
\Prob(\Event_j)&\leq
\Prob\big\{|\langle \nu_j,\col_j(\widetilde A)\rangle|\leq (2\alpha)^{-C_{\smallrefer{cor: no moderately sparse for smin}}\log^3 n}/(\alpha\sqrt{n})
\mbox{ and }(\nu_j)^*_M\geq (2\alpha)^{-C_{\smallrefer{cor: no moderately sparse for smin}}\log^3 n}/\sqrt{n}\big\}.
\end{align*}
Indeed, the inequality
$|\langle \nu_j,\col_j(\widetilde A)\rangle|\leq (2\alpha)^{-C_{\smallrefer{cor: no moderately sparse for smin}}\log^3 n}/(\alpha\sqrt{n})$,
together with the condition $\langle \nu_j,\col_u(\widetilde A)\rangle=0$, $u\neq j$,
implies $\|\widetilde A^\top \nu_j\|_2 \leq (2\alpha)^{-C_{\smallrefer{cor: no moderately sparse for smin}}\log^3 n}\|\nu_j\|_{\infty}$.
Since $\Event_j\subset\Event_{good}$, by Corollary~\ref{cor: no moderately sparse for smin}
(applied to the transposed matrix $\widetilde A^\top$) this
implies that $\nu_j$ is spread in the sense that
$(\nu_j)^*_M\geq (2\alpha)^{-C_{\smallrefer{cor: no moderately sparse for smin}}\log^3 n}\|\nu_j\|_\infty$.

By our choice of $\nu_j$'s, the event
$\{(\nu_j)^*_M\geq (2\alpha)^{-C_{\smallrefer{cor: no moderately sparse for smin}}\log^3 n}/\sqrt{n}\}$
is measurable with respect to $\sigma$-algebra generated by columns $\col_u(\widetilde A)$, $u\in [n]\setminus\{j\}$,
and hence
\begin{align*}
\Prob\big\{&|\langle \nu_j,\col_j(\widetilde A)\rangle|\leq (2\alpha)^{-C_{\smallrefer{cor: no moderately sparse for smin}}\log^3 n}/(\alpha\sqrt{n})
\mbox{ and }(\nu_j)^*_M\geq (2\alpha)^{-C_{\smallrefer{cor: no moderately sparse for smin}}\log^3 n}/\sqrt{n}\big\}\\
&= \Exp\,\Big(\Prob\big\{|\langle \nu_j,\col_j(\widetilde A)\rangle|\leq (2\alpha)^{-C_{\smallrefer{cor: no moderately sparse for smin}}\log^3 n}/(\alpha\sqrt{n})
\;|\;\nu_j\big\}\indicator_{\{(\nu_j)^*_M\geq (2\alpha)^{-C_{\smallrefer{cor: no moderately sparse for smin}}\log^3 n}/\sqrt{n}\}}\Big)\\
&\leq \sup\limits_Y\Prob\big\{|\langle Y,\col_j(\widetilde A)\rangle|
\leq (2\alpha)^{-C_{\smallrefer{cor: no moderately sparse for smin}}\log^3 n}/(\alpha\sqrt{n})\big\},
\end{align*}
where in the last relation the supremum is taken over all unit (non-random) complex vectors $Y$ with
$Y^*_M\geq (2\alpha)^{-C_{\smallrefer{cor: no moderately sparse for smin}}\log^3 n}/\sqrt{n}$.
Fix $Y$ for which the supremum is attained.

Note that, by the assumptions on the matrix, each entry $\widetilde a_{ij}$ of $\widetilde A$ satisfies
$\cf\big(\widetilde a_{ij},\frac{1}{\alpha}\big)\leq 1-\frac{p}{\alpha}$, $i=1,2,\dots,n$,
and hence, denoting $\xi_i:=\widetilde a_{ij}Y_i$, we get
$\cf(\xi_i,\frac{1}{\alpha}(2\alpha)^{-C_{\smallrefer{cor: no moderately sparse for smin}}\log^3 n}/\sqrt{n})\leq 1-\frac{p}{\alpha}$,
$i\in\Max_M(Y)$.
Then, by Lemma~\ref{l: rogozin}, the probability $\Prob(\Event_j)$
can be bounded from above as
$$\Prob(\Event_j)\leq\cf\Big(\sum\limits_{i\in \Max_M(Y)}\xi_i,
(2\alpha)^{-C_{\smallrefer{cor: no moderately sparse for smin}}\log^3 n}/(\alpha\sqrt{n})\Big)\leq\frac{C\sqrt{\alpha}}{\sqrt{p M}}\leq
C'\sqrt{\alpha\log(pn)/pn},$$
for universal constants $C,C'>0$.
Thus, we have
$$\Exp\sum\limits_{j=1}^n \indicator_{\Event_j}\leq C'n\sqrt{\alpha\log(pn)/pn}.$$
This, together with \eqref{eq: aux apsofianpfka} and Markov's inequality, implies
$$\Prob(\Event_{smin})\leq C''\sqrt{\alpha}\log^{3/2}(pn)/\sqrt{pn}.$$
It remains to note that $\Prob(\Event_{good})\geq 1-(pn)^{-c}$.
\end{proof}

\section{Randomized restricted invertibility} \label{sec: restricted inv}

We seek to extend the method of bounding the smallest singular value from the previous section
to bounding the $k$-th smallest one.
It would be natural to suggest replacing rank one random projections by the higher rank ones. Such idea was implemented in \cite{RV rectangular} leading to the optimal estimate for the intermediate singular value of a dense matrix (see \cite{Wei} for a matching upper estimate).
However, the method used in \cite{RV rectangular} to construct random test projections provides a probability
estimate which becomes too weak if we consider very sparse matrices.
To improve the probability estimate, we will take advantage of the special structure of a test projection. This will be achieved by applying the restricted invertibility principle originating in the classical work of Bourgain and Tzafriri \cite{BT}.
The argument based on the restricted invertibility was used in the recent papers of Cook \cite{Cook smallest singular} and Nguyen \cite{Nguyen},
but our method will be different.

We need a probabilistic version of the Bourgain--Tzafriri restricted invertibility theorem \cite[Theorem 1.2]{BT}.
\begin{lemma}  \label{lem: Bourgain-Tzafriri}
Let $\eta \in (0,1)$, $\rho>0$;
let $k<n$, and let $V$ be a $k \times n$ matrix with complex entries whose rows $\row_1(V) \etc \row_k(V)$ are orthonormal.
Assume that $(\row_j(V))_{\lfloor\eta n\rfloor}^* \ge \rho/\sqrt{n}$ for all $j \in [k]$.
Let
  \[
    \ell:=\lfloor\tilde{c}_{\smallrefer{lem: Bourgain-Tzafriri}} \eta^3 \rho^2 k\rfloor,
  \]
let $\beta_1 \etc \beta_n$ be independent $\text{\rm Bernoulli}(\ell/n)$ random variables, and set $J:=\{j \in [n]: \ \beta_j=1 \}$.
Denote the columns of $V$ by $V_1 \etc V_n$. Then with probability at least $(\hat{c}_{\smallrefer{lem: Bourgain-Tzafriri}} \eta)^\ell$, the set $J$ satisfies
\begin{enumerate}
   \item $|J| = \ell$;
   \item $\norm{V_j}_2 \le  \sqrt{\frac{C_{\smallrefer{lem: Bourgain-Tzafriri}}k}{\eta n}}$ for all $j \in J$;
   \item $\norm{\sum_{j \in J} z_j V_j}_2 \ge c_{\smallrefer{lem: Bourgain-Tzafriri}} \rho \sqrt{\eta \frac{ k}{n}} \norm{z}_2$ for any $z \in \C^J$.
\end{enumerate}
\end{lemma}
\begin{proof}
Let $R>1$ be a parameter to be chosen later, and let $\mu_1 \etc \mu_n$ be independent $\text{\rm Bernoulli}(R\ell/n)$ random variables. Denote $J_1:=\{j \in [n]: \ \sigma_j=1 \}$. We will prove that some weaker properties hold for the random set $J_1$ with probability at least $1/2$.
Then we will extract a subset $J$ of cardinality
$\ell$ from each good realization of the set $J_1$ satisfying (1), (2), and (3).
This extraction can be viewed as a random selection using axillary $2R\ell$ independent Bernoulli$(1/R)$ random variables. In this case, the probability that the correct subset $J$ is selected is at least $\exp(-c\ell)$.

We pass to a detailed construction.
First, we select a subset of columns $\hat I$ with upper and lower bounds on the Euclidean norm.
By assumption of the lemma, the matrix $V=(v_{ji})$ satisfies
\[
|I_j|:=|\{i \in [n]: \ |v_{ji}| \ge \rho/\sqrt{n} \}| \ge \lfloor\eta n\rfloor \quad \text{for all } j \in [k].
\]
Denote $Y_i:=|\{j \in [k]: \ i \in I_j\}|$. Then $0 \le Y_i \le k$, and the previous inequality
implies that $\sum_{i=1}^n Y_i \ge \lfloor\eta n\rfloor
k$. Hence,
 \begin{align*}
  |\{i \in [n]: Y_i \ge \eta k/2\}|
  &\ge \frac{1}{k} \sum_{i \in [n]: \ Y_i \ge \eta k/2} Y_i
  \ge \frac{1}{k} \left( \lfloor\eta n\rfloor k - \frac{\eta k}{2} \cdot |\{i \in [n]: Y_i < \eta k/2\}|   \right)\\
  &\ge \frac{\eta n}{2}-1.
 \end{align*}
By the definition of $Y_i$'s, for any $i\leq k$ we have $\norm{V_i}_2\geq \rho \sqrt{Y_i}/\sqrt{n}$, and so
 \[
   |\widetilde{I}|:= \Big| \Big\{i \in [n]: \ \norm{V_i}_2 \ge  \rho \sqrt{\frac{\eta k}{2 n}}\, \Big\} \Big| \ge \frac{\eta n}{2}-1.
 \]
Set $\hat{I}:=\Big\{i \in \widetilde{I}: \ \norm{V_i}_2 \le  \sqrt{\frac{4k}{\eta n}}\, \Big\}$.
By assumption on the matrix $V$, $\sum_{i=1}^n \norm{V_i}_2^2 =k$, so
 \[
   |\hat{I}|
   \ge  |\widetilde{I}|- \bigg| \bigg\{i \in [n]: \ \norm{V_i}_2 \ge  \sqrt{\frac{4k}{\eta n}}\, \bigg\} \bigg| \ge \frac{\eta n}{4}-1.
 \]
We will select a good subset of indices inside $\hat I$.

Let $\Omega_1$ be the event that $\sum_{i \in \hat{I}} \sigma_i \ge \frac{\eta R\ell}{8}$.
Then, by standard concentration inequalities, $\P(\Omega_1^c) \le 1/8$;
moreover, on the event $\Omega_1$, the set
$$J_2:=J_1\cap\hat I=\left\{ i \in J_1: \  \rho \sqrt{\frac{\eta k}{2 n}}
\le \norm{V_i}_2 \le \sqrt{\frac{4k}{\eta n}}\right\}$$
satisfies $|J_2|\ge \frac{\eta R\ell}{8}$.

For any set $I \subset [n]$, let $Q_I: \C^k \to \C^k$ be the orthogonal projection on $\spn\{V_i, \ i \in I\}$.
Notice that for every given $i \in [n]$,
the random variables $\sigma_i$ and $\norm{Q_{J_1 \setminus \{i\}} V_i}_2$ are independent. Therefore,
 \begin{align*}
  \E\sum_{i \in J_1} \norm{Q_{J_1 \setminus \{i\}} V_i}_2^2
  &=\E  \sum_{i=1}^n \sigma_i \norm{Q_{J_1 \setminus \{i\}} V_i}_2^2
  = \frac{R\ell}{n} \E \sum_{i=1}^n \norm{Q_{J_1  \setminus \{i\}} V_i}_2^2 \\
  &\le \frac{R\ell}{n} \E \sum_{i=1}^n \norm{Q_{J_1} V_i}_2^2
  = \frac{R\ell}{n}  \E \norm{Q_{J_1} V}_{HS}^2 \\
  &\le \frac{R\ell}{n}  \E (\norm{Q_{J_1} }_{HS}^2\cdot\|V\|^2)
  \le \frac{R\ell}{n} \cdot R\ell.
 \end{align*}
 Let $\Omega_2$ be the event that $\sum_{i \in J_1} \norm{Q_{J_1 \setminus \{i\}} V_i}_2^2 \le 8 \frac{(R\ell)^2}{n}$.
 By the Markov inequality and the above estimates, $\P(\Omega_2^c) \le 1/8$. On the event $\Omega_2$, we have
 \[
  \big|\big\{i \in J_1: \ \norm{Q_{J_1  \setminus \{i\}} V_i}_2^2 \ge 128R\ell/\eta n \big\}\big| \le \frac{\eta R\ell}{16}.
 \]
 Let us summarize our conclusions. On the event $\Omega_1 \cap \Omega_2$, whose probability is greater than $3/4$, we have
\begin{align*}
|J_2|:=&\bigg| \bigg\{ i \in J_1: \  \rho \sqrt{\frac{\eta k}{2 n}}
\le \norm{V_i}_2 \le \sqrt{\frac{4k}{\eta n}} \bigg\} \bigg| \ge \frac{\eta R\ell}{8},\;\;\mbox{ and }\\
&|\{i \in J_1: \ \norm{Q_{J_1  \setminus \{i\}} V_i}_2^2 \ge 128R\ell/\eta n \}| \le \frac{\eta R\ell}{16}.
\end{align*}
If $J_3:=\{i \in J_2: \ \norm{Q_{J_1  \setminus \{i\}} V_i}_2^2
\le 128R\ell/\eta n \}$, then on this event, $|J_3| \ge  \frac{\eta R\ell}{16}$.  Thus, for any $i \in J_3$,
 \begin{align*}
  \norm{(\Id-Q_{J_3  \setminus \{i\}}) V_i}_2^2
  &= \norm{V_i}_2^2-\norm{Q_{J_3  \setminus \{i\}} V_i}_2^2
  \ge \norm{V_i}_2^2-\norm{Q_{J_1  \setminus \{i\}} V_i}_2^2 \\
  &\ge \norm{V_i}_2^2 - 128R\ell/\eta n
  \ge \norm{V_i}_2^2/2
 \end{align*}
 if we assume that $R$ and $\tilde{c}$ are chosen so that
 \begin{equation}\label{eq: choice of c tilde}
   128 R \tilde{c}_{\smallrefer{lem: Bourgain-Tzafriri}} \eta <\frac{1}{4}.
 \end{equation}
Now, we assume that event $\Omega_1 \cap \Omega_2$ occurs and fix a realization of the set $J_3$.
The rest of the proof follows \cite[Theorem 1.2]{BT} and is deterministic. Arguing exactly as in \cite[Theorem 1.5]{BT},
we conclude that there is a subset $J_4 \subset J_3$ with $|J_4| \ge |J_3|/3$ such that for any $z_1 \etc z_n \in \C$,
 \[
   \norm{ \sum_{i \in J_4} z_i V_i}_2
   \ge \bar{c} \rho \sqrt{\eta \frac{k}{n}} \cdot |J_4|^{-1/2} \sum_{i \in J_4} |z_i|.
 \]
 Then, following the second proof of \cite[Theorem 1.2]{BT} and combining Grothendieck's theorem and the Pietsch factorization,
we find a subset $J_5 \subset J_4$ with $|J_5| \ge |J_4|/2$ such that
 \[
   \norm{ \sum_{i \in J_5} z_i V_i}_2
   \ge c^* \rho \sqrt{\eta \frac{k}{n}} \cdot \left( \sum_{i \in J_5} |z_i|^2 \right)^{1/2}
 \]
 for any $z_1 \etc z_n \in \C$. Here, $|J_5| \ge  \frac{\eta R\ell}{96}$.
 Choosing $R=100/\eta$, we can select an $\ell$-element  subset $J_6 \subset J_5$
such that the previous inequality holds with $J_6$ in place of $J_5$.
Now, choose $\tilde{c}_{\smallrefer{lem: Bourgain-Tzafriri}}$ such that \eqref{eq: choice of c tilde} is satisfied.
We constructed the subset $J_6$ of cardinality $\ell$ for which the assertion (3) of the Lemma holds.
Assertion (2) holds as well since $J_6 \subset J_2$.

It remains to recast the selection of $J_6$ as a random choice. To this end, we introduce independent
$\text{Bernoulli}(1/R)$ random variables $\eta_1 \etc \eta_n$ and set $\beta_j=\sigma_j \eta_j, \ j \in [n]$.
Then $\beta_j$ are independent $\text{Bernoulli}(\ell/n)$ random variables as required.
Recall that $J_1=\{i \in [n]: \ \sigma_i=1\}$.
Let $\Omega_3$ be the event that $(1/2) R\ell \le |J_1| \le 2 R\ell$, so $\P(\Omega_3^c) \le 1/8$,
and thus $\P (\Omega_1 \cap \Omega_2 \cap \Omega_3) \ge 1/2$.
Condition on $\sigma_1 \etc \sigma_n$ for which $\Omega_1 \cap \Omega_2 \cap \Omega_3$ occurs.
Set $J=\{i \in J_1: \ \eta_i=1\}$. Then
\begin{align*}
\P(J=J_6 \mid \sigma_1 \etc \sigma_n)\ge (1/R)^\ell \cdot (1-1/R)^{|J_1|-\ell}
\ge \left( \frac{\eta}{100} \right)^\ell e^{-2\ell} \ge (c \eta)^\ell,
\end{align*}
so the proof is complete.
\end{proof}

\begin{cor} \label{cor: Bourgain-Tzafriri}
Let $\eta, \ell, \rho$, and $V$ be as in Lemma \ref{lem: Bourgain-Tzafriri}. Let $I \subset [n]$ be a random set uniformly chosen among the subsets of $[n]$ of cardinality $\ell$. Then with probability at least $(\hat{c}_{\smallrefer{lem: Bourgain-Tzafriri}} \eta)^\ell$, the set $I$ satisfies
 \begin{enumerate}
   \item $\norm{V_j}_2 \le  \sqrt{\frac{C_{\smallrefer{lem: Bourgain-Tzafriri}}k}{\eta n}}$ for all $j \in I$;
   \item $\norm{\sum_{j \in I} z_j V_j}_2 \ge c_{\smallrefer{lem: Bourgain-Tzafriri}} \rho \sqrt{\eta \frac{ k}{n}} \norm{z}_2$ for any $z \in \C^I$.
 \end{enumerate}
\end{cor}
\begin{proof}
Let $J$ be the set appearing in Lemma \ref{lem: Bourgain-Tzafriri}. Conditionally on the event $|J|=\ell$,
the set $J$ is uniformly distributed among the $\ell$--element subsets of $[n]$.
Since the conditional probability is at least as large as the unconditional one, the corollary follows.
\end{proof}

If $B$ is an $n \times n$ matrix, then
$s_{n-k+1}(B) \le s$ if and only if there exists a linear subspace $E \subset \C^n$ of complex dimension $k$,
such that for any $x \in E$, $\norm{Bx}_2 \le s \norm{x}_2$.
The subspace $E$ can be represented as $V^{\top} \C^k$, where $V$ is
a $k \times n$ matrix with orthonormal rows $\row_1(V) \etc \row_k(V)$.
The bound on the singular value is thus equivalent to $\norm{B V^{\top}} \le s$.
Assume that we managed to construct the matrix $V$ so that its rows are well spread.
Then Corollary~\ref{cor: Bourgain-Tzafriri} allows us to relate the bound on $s_{n-k+1}(B)$
to magnitudes of projections of columns of $B$ onto orthogonal complements to spans of some other columns,
thus eliminating the unknown matrix $V$.
To take advantage of this corollary, we will combine it with the following deterministic lemma.

\begin{lemma}\label{l: BT to dist}
Let $\eta, \ell, \rho$ be as in Lemma \ref{lem: Bourgain-Tzafriri},
and let $\mathcal{V}$ be the set of all $k \times n$ matrices $V$ with orthonormal rows such that
$(\row_j(V))_{\lfloor\eta n\rfloor}^* \ge \rho/\sqrt{n}$ for all $j \in [k]$.
Assume that $B$ is an $n \times n$ matrix such that
$\norm{BV^\top} \le s$ for some $V\in\mathcal V$ and $s>0$.
Let $\mathcal{I}_V$
be the set of all subsets $I \subset [n]$ of cardinality $\ell$ satisfying conditions (1) and (2) of Corollary~\ref{cor: Bourgain-Tzafriri}.
Then for any $I\in \mathcal I_V$ we have
$$\norm{P_I \col_j(B)}_2 \le \frac{\sqrt{2}}{c_{\smallrefer{lem: Bourgain-Tzafriri}} \rho} \sqrt{\frac{n}{\eta k}}s\quad\mbox{ for at least
$\ell/2$ indices }j\in I,$$
where $P_I$ denotes the orthogonal projection onto $\big(\spn\{\col_u(B):\;u\in[n]\setminus I\}\big)^\perp$.
\end{lemma}
\begin{proof}
Fix a set $I \in \mathcal{I}_V$.
We use the following identity valid for all vectors $X_{(i)}=(x_{(i)1},\dots,x_{(i)n})\in\C^n$, $i\in I$:
\begin{align*}
\Big\|\sum_{i\in I}V_i X_{(i)}^\top\Big\|_{HS}^2&=\sum_{j=1}^n \Big\|\sum_{i\in I}x_{(i)j}V_i \Big\|_2^2
\geq \inf\limits_{w\in\C^I,\,\|w\|_2=1}\Big\|\sum_{i\in I}w_i V_i\Big\|_2^2\cdot\sum\limits_{j=1}^n\sum_{i\in I}|x_{(i)j}|^2\\
&=\inf\limits_{w\in\C^I,\,\|w\|_2=1}\Big\|\sum_{i\in I}w_i V_i\Big\|_2^2\cdot\sum_{i\in I}\|X_{(i)}\|_2^2.
\end{align*}
Applying the identity to vectors $P_I \col_i(B)$, $i\in I$,
and using the fact that $I \in \mathcal{I}_V$, we obtain that
 \begin{align*}
    c_{\smallrefer{lem: Bourgain-Tzafriri}}^2 \rho^2 \eta \frac{k}{n} \cdot \sum_{i \in I} \norm{P_I \col_i(B)}_2^2
    &\le \bigg\|  \sum_{i \in I} V_i (P_I \col_i(B))^{\top}   \bigg\|_{HS}^2.
 \end{align*}
The Hilbert--Schmidt norm can be estimated as
 \begin{align*}
  \bigg\|  \sum_{i \in I} V_i (P_I \col_i(B))^{\top}   \bigg\|_{HS}&=
  \bigg\|P_I \Big( \sum_{i \in I} \col_i(B) V_i^\top \Big) \bigg\|_{HS}
  =\bigg\|P_I \Big( \sum_{i=1}^n \col_i(B) V_i^\top \Big) \bigg\|_{HS}\\
  &=\norm{P_I B V^\top}_{HS} \le \norm{P_I}_{HS} \cdot \norm{BV^\top} \le \sqrt{\ell}s.
 \end{align*}
 Hence,
 \begin{equation*}  \label{eq: I'}
  \exists I' \subset I \ |I'|=\lceil\ell/2\rceil \text{ and } \forall i \in I' \
   \norm{P_I \col_i(B)}_2 \le \frac{\sqrt{2} s}{c_{\smallrefer{lem: Bourgain-Tzafriri}} \rho} \sqrt{\frac{n}{\eta k}},
 \end{equation*}
as required.
\end{proof}

To make use of Lemma~\ref{l: BT to dist} in our random model, we will need sufficiently strong anti-concentration
estimates for $\|P_J \col_j(\widetilde A)\|_2$, which are not always available. Indeed, if $\ell\leq e^{-Cpn}n$
then with a large probability the matrix $\widetilde A$ contains at least $\ell$ rows whose only non-zero elements
are the diagonal ones. Then, whenever $J$ is the set of indices of those rows, the kernel $\ker ((\csubm{\widetilde A}{J})^\top)$
is the coordinate subspace, and $\|P_J \col_j(\widetilde A)\|_2=|z|$ with probability close to one for all $j\in J$.
However, with $J$ chosen uniformly at random, we will be able to show that with very large probability
corresponding kernel contains a large orthonormal set of spread vectors, and the random variables
$\ker ((\csubm{\widetilde A}{J})^\top)$ are well spread.
Thus we are forced to introduce the exceptional set of realizations of $J$ for which we do not have a good anti-concentration.
The key property is that the probability of $J$ falling into this exceptional set is much smaller
than the probability of the event described in Corollary~\ref{cor: Bourgain-Tzafriri}.

The next proposition is the main result of this section.
We consider the case when $B$ is a random matrix with independent columns, and the matrix $V$
can be constructed to have sufficiently spread rows.
\begin{prop} \label{prop: via projection}
Let $\eta, \rho>0$, and let $\mathcal{V}$ be the set of all $k \times n$ matrices $V$ with orthonormal rows such that
$(\row_j(V))_{\lfloor\eta n\rfloor}^* \ge \rho/\sqrt{n}$ for all $j \in [k]$.
Let $B$ be an $n \times n$ random matrix with independent columns.
  Let
  \[
    \ell:=\lfloor\tilde{c}_{\smallrefer{lem: Bourgain-Tzafriri}} \eta^3 \rho^2 k\rfloor,
  \]
For $I \subset [n]$, denote by $P_I$ the $n \times n$ orthogonal projection matrix
whose kernel is the linear span of $\col_j(B), \ j \in [n]\setminus I$.

For any $\ell$-element subset $I\subset[n]$, let $\mathcal{F}_\ell(I)$ be a Borel-measurable set of $(n-\ell ) \times n$
matrices with columns indexed by the complement of the set $I$.
Assume that for any $I$, any $j \in I$, and any realization of $\csubm{B}{I}$ from $\mathcal{F}_\ell(I)$, we have
  \begin{equation} \label{eq: projection for intermediate}
   \P \Big\{\norm{P_I \col_j(B)}_2 \le \frac{\sqrt{2}}{c_{\smallrefer{lem: Bourgain-Tzafriri}} \rho} \sqrt{\frac{n}{\eta k}}s\;
   \big|\; \csubm{B}{I} \Big\}
   \le t.
  \end{equation}
  Let $J$ be a random subset of $[n]$ uniformly chosen among the subsets of cardinality $\ell$.
  Let
  \[
   \FF_\ell:= \left\{ M\in\C^{n\times n}: \ \P_{J}\{\csubm{M}{J'} \notin \FF_\ell(J) \} \le (\hat{c}_{\smallrefer{lem: Bourgain-Tzafriri}} \eta/2)^\ell
   \right\}.
  \]
  Then
  \[
   \P \left\{ \exists V \in \mathcal{V}: \ \|BV^\top\| \le s \  \& \  B \in \FF_\ell \right\}
   \le \left(\frac{C_{\smallrefer{prop: via projection}} \sqrt{t} }{\eta}\right)^\ell.
  \]
\end{prop}
\begin{rem}
In our proof, $\mathcal{F}_\ell(I)$ will be the set of all matrices $\csubm{\widetilde A}{I}$
such that  the kernel of   $(\csubm{\widetilde A}{I})^{\top}$ contains $c\ell$
orthonormal vectors with a good $(cn/\log pn)$-th order statistic, ensuring estimates \eqref{eq: projection for intermediate}.
Thus $\mathcal{F}_\ell$ is the event that the kernel of $(\csubm{\widetilde A}{J})^{\top}$
has $c\ell$ orthonormal vectors with a good order statistic for a random set $J$.
\end{rem}

\begin{proof}
As in Lemma~\ref{l: BT to dist},
for a given $V\in\mathcal V$ let $\mathcal{I}_V$
be the set of all subsets $I \subset [n]$ of cardinality $\ell$ satisfying conditions (1) and (2) of Corollary~\ref{cor: Bourgain-Tzafriri}.

For the random matrix $B$, define a random matrix $\widetilde V$ measurable with respect to $B$,
constructed as follows: whenever for a given realization of $B$ there is a matrix $V \in \mathcal{V}$
with $\norm{B V^{\top}} \le s$, choose $\widetilde V$ to be such a matrix; otherwise, let $\widetilde V$
be any matrix from $\mathcal V$.
To avoid measurability problems, we can assume that $B$ takes finitely many values.
This assumption can be easily removed after the proof of the proposition is complete.
 Denote
 \[
  q:= \P_{B,J} \left\{\big\|B \widetilde V^{\top} \big\| \le s \ \& \ J \in \mathcal{I}_{\widetilde V} \ \& \
\csubm{B}{J} \in \mathcal{F}_\ell(J) \ \& \ B \in \mathcal{F}_\ell \right\}.
 \]

We will estimate this probability  in two ways.
First, by Corollary \ref{cor: Bourgain-Tzafriri}, for any matrices $M \in \mathcal{F}_\ell$
and $V\in\mathcal{V}$ satisfying $\norm{MV^\top} \le s$, we have
\begin{align*}
   \P\big\{J \in \mathcal{I}_V  \ \& \ \csubm{M}{J} \in \mathcal{F}_\ell(J)\big\}
   &\ge  \P\big\{J \in \mathcal{I}_V\big\}
   - \P\big\{\csubm{M}{J} \notin \mathcal{F}_\ell(J)\big\}  \\
   &\ge (\hat{c}_{\smallrefer{lem: Bourgain-Tzafriri}} \eta)^\ell-(\hat{c}_{\smallrefer{lem: Bourgain-Tzafriri}} \eta/2)^\ell
   \ge (\hat{c}_{\smallrefer{lem: Bourgain-Tzafriri}} \eta/2)^\ell.
\end{align*}
Hence,
 \begin{align} \label{eq: lower bound  on q}
  q
   &=\E_{B} \left( \P\big\{J \in \mathcal{I}_{\widetilde V}  \ \& \ \csubm{B}{J} \in \mathcal{F}_\ell(J)\,
\mid\; B\big\} \cdot \mathbf{1}_{\|B\widetilde V^\top\| \le s} \cdot \mathbf{1}_{B \in \mathcal{F}_\ell}
   \right) \notag \\
  &\ge (\hat{c}_{\smallrefer{lem: Bourgain-Tzafriri}} \eta/2)^\ell \cdot
\P\big\{\|B\widetilde V^\top\| \le s \ \& \ B \in \mathcal{F}_\ell\big\}.
 \end{align}
 On the other hand,
\begin{align*}
q&\leq\P_{B,J} \left\{\big\|B \widetilde V^{\top} \big\| \le s \ \& \ J \in \mathcal{I}_{\widetilde V} \ \& \
\csubm{B}{J} \in \mathcal{F}_\ell(J)\right\}\\
&=\E_{J} \big[\E_{\csubm{B}{J} } \big( \E_{\csubm{B}{{J}^c}}
(\mathbf{1}_{\|B\widetilde V^\top\| \le s} \cdot \mathbf{1}_{J \in \mathcal{I}_{\widetilde V}}
\cdot \mathbf{1}_{\csubm{B}{J} \in \mathcal{F}_\ell(J)} \mid J, \csubm{B}{J}) \mid J\big)\big]\\
&= \E_{J} \big[\E_{\csubm{B}{J} } \left( \E_{\csubm{B}{{J}^c}}
\big(\mathbf{1}_{\|B\widetilde V^\top\| \le s} \cdot \mathbf{1}_{J \in \mathcal{I}_{\widetilde V}}\mid J, \csubm{B}{J} \big)  \cdot \mathbf{1}_{\csubm{B}{J} \in \mathcal{F}_\ell(J)} \mid J\right)\big].
\end{align*}
Applying Lemma~\ref{l: BT to dist} we get
that the event $\big\{\|B\widetilde V^\top\| \le s\mbox{ and }J \in \mathcal{I}_{\widetilde V}\big\}$
is contained in the event
$$\Big\{\norm{P_J \col_j(B)}_2 \le \frac{\sqrt{2}}{c_{\smallrefer{lem: Bourgain-Tzafriri}} \rho} \sqrt{\frac{n}{\eta k}}s\quad\mbox{ for at least
$\ell/2$ indices }j\in J\Big\}.$$
Hence,
 \begin{multline*}
 \E_{\csubm{B}{J^c}}(\mathbf{1}_{\|B\widetilde V^\top\| \le s} \cdot \mathbf{1}_{J \in \mathcal{I}_{\widetilde V}}
\mid J,\csubm{B}{J} )
 \\
 \le \P \left( \exists \widetilde J \subset J \ |\widetilde J|=\lceil \ell/2\rceil \ \mbox{ such that } \ \forall i \in \widetilde J \
   \norm{P_J \col_i(B)}_2 \le \frac{\sqrt{2} s}{c_{\smallrefer{lem: Bourgain-Tzafriri}} \rho} \sqrt{\frac{n}{\eta k}} \mid J,\csubm{B}{J}  \right).
 \end{multline*}
Note that conditioned on any realization of $J$ and $\csubm{B}{J}$, the projections $P_J \col_i(B)$, $i\in J$,
are jointly independent.
Therefore, on the event $\{\csubm{B}{J} \in \mathcal{F}_\ell(J)\}$ we can apply \eqref{eq: projection for intermediate}
together with the union bound over all $\lceil\ell/2\rceil$-element subsets of $J$
to get
$$
\E_{\csubm{B}{J^c}}(\mathbf{1}_{\|B\widetilde V^\top\| \le s} \cdot \mathbf{1}_{J \in \mathcal{I}_{\widetilde V}}
\mid J,\csubm{B}{J} )\cdot \mathbf{1}_{\csubm{B}{J} \in \mathcal{F}_\ell(J)}
\leq \binom{\ell}{\lfloor \ell/2\rfloor} \cdot t^{\ell/2}\le (C \sqrt{t})^\ell.
$$
In combination with the above inequalities, this yields
 \[
 q \le (C \sqrt{t})^\ell.
 \]
 Combining it with \eqref{eq: lower bound  on q}, we conclude the proof of the proposition.
\end{proof}

\section{The intermediate singular values} \label{sec: intermediate}

In this section we are concerned with bounding intermediate singular values
$s_{n-k}(\widetilde A)$ for $n/\log^C n\leq k\leq n(pn)^{-c}$.
Note that for $pn$ polylogarithmic in $n$, the interval for $k$ is empty,
and the results of this section do not enter into the proof of the circular law.
The estimates obtained here become important when $pn\leq\log n$,
and will be used in the next section to verify uniform integrability of logarithm with respect
to empirical measures of singular values of $\widetilde A_n$.
Estimating the intermediate singular values in the setting of random directed $d$--regular graphs
was an important step in the proof of the circular law for that model in the regime when the degree $d$ is sub-logarithmic in dimension
\cite{LLTTY circ}. We note that in \cite{LLTTY circ} a completely different approach based on bounding distances between matrix columns
and {\it uniform random normals} to certain random subspaces was employed.

Assume that  the matrix $\widetilde{A}$ is such that the event $\Event_{good}$ occurs.
In this section, we will show that  for a random set $J$ of a fixed cardinality, with high probability  the space $\ker((\csubm{\tilde A}{J})^{\top}$ possesses a large orthonormal system of sufficiently spread vectors.
We start with a deterministic statement asserting the existence of an orthonormal basis of spread vectors in any fixed subspace
(see \cite[Lemma~4.3]{LLTTY circ} for a related statement).
\begin{lemma}[Basis of spread vectors]\label{l: small infinity norm}
 Let $E \subset \C^n$ be a linear subspace of dimension $k \ge C \log n$.
 Let
 \[
   1\leq s \le c_{\smallrefer{l: small infinity norm}} \frac{k}{\log (n/k)},
 \]
where $c_{\smallrefer{l: small infinity norm}}>0$ is a sufficiently small universal constant.
Then there exists an orthonormal basis $u_1 \etc u_k$ in $E$ such that $(u_j)_s^* \ge \frac{1}{2\sqrt{n}}$ for all $j \in [k]$.
\end{lemma}
\begin{proof}
 Let $u_E$ be a random vector uniformly distributed on $S^{n-1}(\C) \cap E$.
Let $P_J^{\text{crd}}$ be the coordinate projection on $\C^J, \ J \subset [n]$.
We will show that with large probability for any $s$--element subset $J$ of $[n]$, $u_E$
satisfies $\norm{P_J^{\text{crd}} u_E}_2 <1/2$.
To this end, we  represent $u_E$ as $P_E g/\norm{P_E g}_2$ where $P_E$ is the projection on $E$,
and $g$ is the standard Gaussian vector in $\C^n$. Then by the Gaussian concentration
 \[
 \P\big\{ \norm{P_E g}_2 \le \sqrt{k}/2\big\} \le \exp(-ck).
\]
 Also, $\E \norm{P_J^{\text{crd}} P_E g}_2^2 \le \E \norm{P_J^{\text{crd}}  g}_2^2 =s$ since for any $B\in\C^{n\times n}$,
$\E \norm{Bg}_2^2$ depends only on the singular values of $B$.
Using the Gaussian concentration again, we derive
\[
  \P\big\{\norm{P_J^{\text{crd}} P_E g}_2 > t\big\} \le \exp(-ct^2)
\]
for $t \ge 2\sqrt{s}$. Choosing $t := \sqrt{Cs \log \frac{n}{s}}$ (for a sufficiently large $C>0$), we get
\begin{align*}
  k\binom{n}{s} \cdot \P(\norm{P_J^{\text{crd}} P_E g}_2 > t)
  \le \exp \left(\log k + s \log \frac{en}{s} - C s \log \frac{n}{s} \right)
  \le \frac{1}{4}.
\end{align*}
Hence,
\begin{align*}
   \P \bigg\{&\exists J \in \binom{[n]}{s}: \ \norm{P_J^{\text{crd}} u_E}_2
\ge 2C \sqrt{\frac{s}{k}} \cdot \sqrt{\log \frac{n}{s}} \bigg\} \\
   &\le \P \bigg\{\exists J \in \binom{[n]}{s}: \ \norm{P_J^{\text{crd}} P_E g}_2  \ge C \sqrt{s} \cdot \sqrt{\log \frac{n}{s}}
\ \& \ \norm{P_E g}_2 \ge \frac{\sqrt{k}}{2} \bigg\} \\
   &\hspace{3cm}+ \P \bigg\{ \norm{P_E g}_2 \le \frac{\sqrt{k}}{2} \bigg\}
  \le \frac{1}{2k},
\end{align*}
and by the union bound, a Haar--uniformly distributed random orthonormal basis $u_1 \etc u_k$ in $E$ satisfies
\[
 \forall i \in [k] \  \forall J \in \binom{[n]}{s}: \ \norm{P_J^{\text{crd}} u_i}_2 \le 2C \sqrt{\frac{s}{k}} \cdot \sqrt{\log \frac{n}{s}} \le \frac{1}{2}
\]
with probability at least $1/2$.
For any such realization, we have $\norm{(u_i^*)_{[s:n]}}_2 \ge 1/2$, which implies the lemma.
\end{proof}

Lemma~\ref{l: small infinity norm} above allows to construct an orthonormal basis with a good control of
the $\ell_\infty$--norm of the vectors. Yet, it does not give sufficiently strong information
on the size of the vector support. On the other hand, Proposition~\ref{prop: almost proportional balance}
which we proved earlier in this paper, provides lower bounds on the $cn/\log(pn)$--th order statistics of almost null vectors
but does not imply a strong upper bound on the $\ell_\infty$ norm.

We would like to combine Lemma \ref{l: small infinity norm} with Proposition \ref{prop: almost proportional balance} to improve
the ``spreadness'' property of the vectors in the basis.
Yet, this is not always possible since $E:=\ker((\csubm{\widetilde A}{J})^{\top})$
can be a coordinate subspace for some choice of $J$, as we discussed in the previous section.
Fortunately, even if constructing a good orthonormal system in $\ker((\csubm{\widetilde A}{J})^{\top})$
is impossible for all sets $J$, it is possible for a random set $J$ with high probability.
In view of Proposition \ref{prop: via projection}, this would be enough.
To show this property of $\ker((\csubm{\widetilde A}{J})^{\top})$,
we will utilize the concept of the matrix compression: for any realization of $J$,
we replace $\ker((\csubm{\widetilde A}{J})^{\top})$ with its subspace having the form
$\ker(\phi(\widetilde A^{\top}))$, for a specially chosen compression $\phi$.
On the one hand, existence of the compression is guaranteed with high probability by Lemma~\ref{l: phi construction}.
On the other hand, the required structural properties of $\ker(\phi(\widetilde A^{\top}))$
can be verified by combining Proposition~\ref{prop: almost proportional balance}
and Lemma~\ref{l: small infinity norm}.
In this respect, the matrix compression allows to replace the problem of describing the geometry of $\ker((\csubm{\widetilde A}{J})^{\top})$
(which turns out to be a complex mixture of spread and sparse vectors) with studying
a relatively simple subspace $\ker(\phi(\widetilde A^{\top}))$ which typically contains only spread vectors.

\begin{prop} \label{prop: othonormal system}
Let $n$, $p$, $z$ and the matrix $A$
satisfy assumptions \eqref{Asmp on p weak}--\eqref{Asmp on A}--\eqref{Asmp on z}.
Fix a realization of $A$ in $\Event_{good}$.
Let $\ell\geq n^{1/2}$
be a natural number. Let $J$ be a random subset of $[n]$ of cardinality $\ell$ uniformly chosen from the sets of this cardinality.
Let
\[
 M:=\Big\lfloor\tilde{c}_{\smallrefer{prop: almost proportional}} \frac{n}{\log pn}\Big\rfloor,
\]
and for every fixed $I\subset[n]$, $|I|=\ell$, let $\FF_\ell(I)$ be the set of all $(n-\ell) \times n$
matrices $B$ (with rows indexed over $I^c$) such that the kernel of $B$ contains $\lfloor c_{\smallrefer{prop: othonormal system}}\ell\rfloor$
orthonormal vectors $v_1 \etc v_{\lfloor c_{\smallrefer{prop: othonormal system}}\ell\rfloor}$ with
\begin{equation} \label{eq: M order statistic}
  (v_j)_M^*
  \ge \frac{1}{\sqrt{n}} \exp \left(- C_{\smallrefer{prop: othonormal system}} \log^4(pn)
\log^4 \left( \frac{n}{\ell} \right) \right),
\quad\quad j=1,2,\dots,\lfloor c_{\smallrefer{prop: othonormal system}}\ell\rfloor.
\end{equation}
Then
\[
 \P \left\{ (\csubm{\widetilde A}{J})^{\top} \notin \FF_\ell(J)\right\}
 \le \left( \frac{\hat{c}_{\smallrefer{lem: Bourgain-Tzafriri}}/2}{\log pn} \right)^\ell.
\]

\end{prop}
We will use this proposition to show that $\EE_{\text{good}}$
can play the role of $\FF_\ell$ in Proposition \ref{prop: via projection}.

\begin{proof}
Define $K_0:=\frac{pn}{2\alpha}$, $K:=K_0/2$ and $\varepsilon:=\frac{1}{2^{15}\alpha}$.
Since $A$ belongs to $\EE_{\smallrefer{p: supports}}\cap \Event_{\smallrefer{p: chains combined}}$, we can use Lemma~\ref{l: phi construction}.
Namely, let $\Event$ be the event (with respect to the randomness of $J$) defined in the lemma.
It is sufficient to show that for any realization of $J$ from $\Event$, we have
$(\csubm{\widetilde A}{J})^{\top} \in \FF_\ell(J)$.

Fix any realization of $J$ from $\Event$ and let $\phi$ be the mapping defined in Lemma~\ref{l: phi construction}.
Observe that the kernel of $(\csubm{\widetilde A}{J})^{\top}$ contains the kernel of the matrix
$\phi(\widetilde A^{\top})$. Further, by Lemma~\ref{l: small infinity norm},
there is an orthonormal basis $v_1,v_2,\dots,v_{\lfloor \varepsilon \ell\rfloor}$ in the kernel of $\phi((\widetilde A)^{\top})$,
such that for
$$q:=\min\Big(\Big\lfloor c_{\smallrefer{l: small infinity norm}}\lfloor \varepsilon \ell\rfloor/\log \frac{n}{\lfloor \varepsilon \ell\rfloor}\Big\rfloor,
\lfloor c_{\smallrefer{prop: no very sparse null vectors}}/p\rfloor\Big)$$
we have $(v_i)^*_q\geq \frac{1}{2\sqrt{n}}$ for all vectors from the basis.

Observe that the matrix $\widetilde A^{\top}$ and the mapping $\phi$ satisfy conditions
of Proposition~\ref{prop: almost proportional balance}.
Hence, we have for all $i$:
$$(v_i)^*_M\geq (2\alpha)^{-C_{\smallrefer{prop: almost proportional balance}}\log^2\frac{4n}{q}\,\log^2(pn+\log\frac{4n}{q})}.$$
This implies that
$(\csubm{\widetilde A}{J})^{\top} \in \FF_\ell(J)$, and the statement follows.
\end{proof}


We will now establish a bound for intermediate singular values necessary to derive the circular law.
Our main tool is Proposition \ref{prop: via projection} reducing the singular value bound to the bound
for the distances between a random set of rows of the matrix and one row from the complement of this set.
To apply it, we will construct a special projection matrix $P_J$ for any set $J \subset [n], \ |J|=\ell$. Note
that any such projection matrix can be represented as $P_J=Q_J Q_J^{\top}$, where $Q_J$ is an $n \times \ell$
matrix whose columns form an orthonormal basis of the space $H_J=(\Span(\col_j(\widetilde{A}), \ j \notin J))^{\perp}$.
Any vector in such basis is in the kernel of the $(n-\ell) \times n$ matrix $B$ which is obtained from the matrix $\widetilde{A}^{\top}$
by deleting the rows from $J$. We will use Proposition \ref{prop: othonormal system} to construct  an orthonormal basis of spread vectors.

\begin{theor}  \label{thm: intermediate singular}
Let $\beta, \d \in (0,1)$.
Let $n$, $p$, $z$ and the matrix $A$
satisfy assumptions \eqref{Asmp on p weak}--\eqref{Asmp on A}--\eqref{Asmp on z}.
 Then
 \begin{align*}
   \P \Big\{&\exists k\geq \frac{n}{\log^{100}n} \mbox{ such that}\;\;
    s_{n-k}(\widetilde{A}) \le  e^{-C_{\smallrefer{prop: almost proportional balance}}\log^{100}\frac{4n}{k}\,\log^{100}(pn)}
    \ \& \  \EE_{\text{good}}  \Big\}
   \le \left(\frac{1}{ pn} \right)^{-c_{\smallrefer{thm: intermediate singular}} \sqrt{n}}.
 \end{align*}
\end{theor}


\begin{proof}
 Fix an integer $k\geq \frac{n}{\log^{100}n}$.
 Denote
 \[
  \tau :=  e^{-C_{\smallrefer{prop: almost proportional balance}}\log^{100}\frac{4n}{k}\,\log^{100}(pn)}.
 \]
Fix for a moment any realization of $\widetilde A$ such that $\EE_{\text{good}}$ occurs and such that $s_{n-k}(\widetilde{A}) \le \t$.
Let $E$ be the subspace spanned by the $k$ singular vectors of $\widetilde{A}$ corresponding to the smallest singular values.
Then for any $x \in S^{n-1}(\C) \cap E$, we have $\|\widetilde{A}x\|_2 \le \t$.
Choose an orthonormal basis $v_1 \etc v_k$ of $E$ as in Lemma \ref{l: small infinity norm}.
 Then for
 \[
  s:=\Big\lfloor\frac{c_{\smallrefer{l: small infinity norm}} k}{\log(n/k)}\Big\rfloor
 \]
 and for any $j \in [k]$, we have $(v_j)_s^* \ge \frac{1}{2 \sqrt{n}}$.
 Hence, assuming that $C_{\smallrefer{thm: intermediate singular}}$ is large enough, we have
\[
\|\widetilde{A} v_j\|_2
\le \t
\le \sqrt{n} (v_j)_s^*\, \frac{1}{2\alpha}\,(2\alpha)^{-C_{\smallrefer{prop: almost proportional balance}}\log^4\frac{4n}{s}\,\log^4(pn)},
\]
and, as long as $s\leq c/p$, the assumptions of Proposition \ref{prop: almost proportional balance} are satisfied. By this proposition,
\[
(v_j)_M^* \ge (v_j)_s^* \exp \Big( -C'\log^4\frac{4n}{s}\,\log^4(pn)\Big)
\ge \frac{1}{\sqrt{n}} \rho
\]
with
\[
M:= \Big\lfloor\frac{\tilde{c}_{\smallrefer{prop: almost proportional balance}} n}{\log pn}\Big\rfloor =: \lfloor\eta n\rfloor
\quad \text{and} \quad \rho:=
\exp \Big( -C''\log^4\frac{4n}{k}\,\log^4(pn)\Big),
\]
where we used $\log \left( \frac{n}{s} \right) \le 2 \log \left( \frac{n}{k} \right)$ to estimate $\rho$.
On the other hand, for $s\geq c/p$, the bound $(v_j)_M^*\geq \frac{1}{2 \sqrt{n}}$ follows immediately from Lemma~\ref{l: small infinity norm}.
This means that for these $\eta$ and $\rho$, the $k \times n$ matrix $V$ with rows $v_1 \etc v_k$
belongs to the set $\mathcal{V}$ defined in Proposition \ref{prop: via projection}. Thus,
\begin{equation}  \label{eq: norm of AV^T}
   \P \Big\{
    s_{n-k}(\widetilde{A}) \le  \tau
    \ \& \  \EE_{\text{good}} \Big\}
  \le
   \P \Big\{ \exists V \in \mathcal{V} \
    \|\widetilde A V^{\top}\| \le \tau
    \ \& \  \EE_{\text{good}}  \Big\}.
\end{equation}
To apply this proposition, we will construct a projection $P_J$ for a set $J \subset [n]$ with
 \[
   |J|=\ell:=\lfloor \tilde{c} \eta^3 \rho^2 k\rfloor,
 \]
for which $(\csubm{\widetilde A}{J})^{\top} \in \FF_\ell(J)$, where $\FF_\ell(J)$ is defined in Proposition \ref{prop: othonormal system}.
This requires checking that \eqref{eq: projection for intermediate} holds for $P_J$.
By the assumption on $k$,
 $
   4 n^{1/2} 
   \le \ell.
  $

Fix an $\ell$--element subset $J$ of $[n]$, and condition on a realization
of $(\csubm{\widetilde A}{J})^{\top}$ from $\FF_\ell(J)$.
We define projection $P_J$ as $P_J:=Q_J Q_J^{\top}$, where $Q_J$ is an $n \times \ell$ matrix whose columns $Q_1 \etc Q_\ell$
form an orthonormal basis of $\ker((\csubm{\widetilde A}{J})^{\top})$.
We will choose a special orthonormal basis. Namely, we will choose
$\lfloor c_{\smallrefer{prop: othonormal system}}\ell\rfloor$
orthonormal vectors $Q_1 \etc Q_{\lfloor c_{\smallrefer{prop: othonormal system}}\ell\rfloor}$
satisfying  the condition \eqref{eq: M order statistic} and complete them (arbitrarily)
to an orthonormal basis of $\ker((\csubm{\widetilde A}{J})^{\top})$.

By \eqref{eq: M order statistic}, for any $j \le c_{\smallrefer{prop: othonormal system}}\ell$, we have
 \[
   (Q_j)_M^*
   \ge \frac{1}{\sqrt{n}} \exp \left(- C_{\smallrefer{prop: othonormal system}} \log^4 \left( \frac{n}{\ell} \right)\log^4(pn) \right),
 \]
 where
 \[
  \log \left( \frac{n}{\ell} \right) \le C \left[ \log \left( \frac{1}{\log pn} \right) + \log^4 \left( \frac{4n}{k} \right)\log^4(pn) \right].
 \]
 Therefore,
 \begin{equation} \label{eq: M coord}
   (Q_j)_M^*
   \ge \frac{1}{\sqrt{n}} \exp \left(-\widetilde C \left[ \log^{20}(pn)\log^{16}\left( \frac{4n}{k} \right) \right] \right)
   =: \frac{\bar{\rho}}{\sqrt{n}}.
 \end{equation}
This estimate will be instrumental in obtaining the small ball probability bound for
$\norm{P_J \col_i(\widetilde{A})}_2, \ i \in J$ which is needed to apply Proposition~\ref{prop: via projection}.

Take $j\leq \lfloor c_{\smallrefer{prop: othonormal system}}\ell\rfloor, \ i \in J$,
and apply the Lemma~\ref{l: rogozin} to the random variable $Y_j=\langle Q_j ,\col_i(\widetilde{A})\rangle$.
In combination with \eqref{eq: M coord}, it yields
 \begin{align*}
   \LL(Y_j, \bar{\rho}/(\alpha\sqrt{n}))
  &= \LL\Big((\alpha\sqrt{n}/\bar{\rho})\sum_{m=1}^n (Q_j)_m \widetilde{a}_{mi}, 1\Big) \\
  &\le \frac{c}{\sqrt{\sum_{m=1}^n [ 1- \LL((\alpha\sqrt{n}/\bar{\rho}) (Q_j)_m \widetilde{a}_{mi}, 1)] }}\\
  &\le \frac{c}{\sqrt{Mp/\alpha}}
  \le c \sqrt{\frac{\alpha\log pn}{pn}}
  \le (pn)^{-1/4}
  =:t/8.
 \end{align*}

Now, we have to turn the estimates for individual inner products $Y_j$
into the bound for $\|P_J \col_i(\widetilde{A})\|_2^2\geq\sum_{j\leq c_{\smallrefer{prop: othonormal system}}\ell}|Y_j|^2$.
We have to deal with dependencies between $Y_j$'s.
Set $Z:=\sum_{j=1}^{\lfloor c_{\smallrefer{prop: othonormal system}}\ell\rfloor}
\mathbf{1}_{[0,\bar{\rho} /\sqrt{n}]} (|\langle Q_j, \col_i(\widetilde{A})\rangle|)$.
Then $\E Z \le (t/8) \lfloor c_{\smallrefer{prop: othonormal system}}\ell\rfloor$, and so the probability that
$Z>\lfloor c_{\smallrefer{prop: othonormal system}}\ell\rfloor/8$ does not exceed $t$.
This means that conditionally on $J$ and $\csubm{\widetilde A}{J}$,
 \[
  \P\left\{ \norm{P_J \col_i(\widetilde{A})}_2 \le \bar{\rho} \frac{\sqrt{c_{\smallrefer{prop: othonormal system}}\ell}}{ \sqrt{2 n}} \right\}
  \leq\P\bigg\{ \sum_{j=1}^\ell |\langle Q_j,
 \col_i(\widetilde{A})\rangle|^2 \le \frac{ \bar{\rho}^2 c_{\smallrefer{prop: othonormal system}}\ell}{2n}
\bigg\} \le t.
 \]

 Define $s'$ via the relation
 \[
  \frac{\sqrt{2}}{c_{\smallrefer{lem: Bourgain-Tzafriri}} \rho} \sqrt{\frac{n}{\eta k}}s'
  = \bar{\rho} \frac{\sqrt{c_{\smallrefer{prop: othonormal system}}\ell}}{ \sqrt{2 n}} .
 \]
 We have checked that $\widetilde{A}$  satisfies \eqref{eq: projection for intermediate} with $s'$ playing the role of $s$.
By Proposition \ref{prop: othonormal system}, $\P_J (\csubm{\widetilde A}{J}) \le (\hat{c}_{\smallrefer{lem: Bourgain-Tzafriri}} \eta/2)^\ell$,
so we can use $\EE_{\text{good}}$ as $\FF_\ell$ in Proposition~\ref{prop: via projection}.

Applying Proposition \ref{prop: via projection}, we conclude that
   \[
   \P \left\{\exists V \in \mathcal{V}: \ \|\widetilde{A}V^\top\| \le s'
   \ \& \  \EE_{\text{good}}   \right\}
   \le \left(\frac{Ct}{\eta}\right)^\ell
   \le \left(\frac{1}{pn} \right)^{-\sqrt{n}}
  \]
since $\ell> 4 \sqrt{n}$.
Substituting the values of $\eta, \rho$, and $\bar{\rho}$, we see that $s' \ge \tau$
.
This inequality, in combination with \eqref{eq: norm of AV^T}, implies the desired estimate for a fixed $k$.
Taking the union bound, we obtain a similar estimate for all $s_{n-k} (\widetilde{A})$
simultaneously.
The proof is complete.
\end{proof}

\section{Proof of the circular law} \label{sec: proof circular}

In this section, we apply the previously obtained singular value estimates to prove the main result of this paper, Theorem \ref{th: main}.
The derivation of the circular law relies on  \cite[Lemma 4.3]{BC circ} (see also \cite{Tao-Vu circ}), which we restate below.
\begin{lemma} \label{lem: sufficient circular}
 Let $M_n$ be a sequence of $n \times n$ random matrices.
 Denote by $\mu_{n,z}$ the empirical measure of the eigenvalues of $M_n$ and by $\nu_{n,z}$ the empirical measure of the singular values of $M_n -z \Id_n$.
 Assume that for a.e. $z \in \C$,
 \begin{enumerate}
    \item \label{it: uniform int}  the function $f(x)=\log x$ is uniformly integrable with respect to the measures $\nu_{n,z}$, i.e.,
     for any $\e>0$, there exists $T>0$ (determined  by $\e$ and $z$) such that
 \[
  \limsup_{n \in \N} \P \left\{ \int_{|\log s| >T} |\log s | \, d \nu_{n,z}(s) > \e \right\} < \e;
 \]

  \item \label{it: weak conv}  the measures $\nu_{n,z}$ converge vaguely in probability to a deterministic measure $\nu_z$ supported on $(0,\infty)$, i.e.,
 for any compactly supported function $h \in C((0,\infty))$,
 \[
  \int_0^\infty h(s) \, d \nu_{n,z}(s) \to \int_0^\infty h(s) \, d \nu_z(s) \quad \text{in probability}.
 \]
 \end{enumerate}
  Then $\mu_n$ converges weakly in probability to the unique probability measure $\mu$ on $\C$ satisfying the equation
  \[
   \int_{\C} \log |\l-z| \, d \mu(z)= \int_0^\infty \log s \, d \nu_z(s) \quad \text{ for all } z \in \C.
  \]
  \end{lemma}

This lemma was employed in recent works \cite{Cook circ,LLTTY circ} on the spectrum of $d$--regular directed graphs.
Note that  assumption  \eqref{it: weak conv} in \cite[Lemma 4.3]{BC circ} required the weak convergence. However, once the uniform integrability is established, the weak and the vague convergence become equivalent.

  We will apply this lemma with $M_n= \frac{1}{\sqrt{p_n n}}A_n$.
  In our case, $\mu$ will be the uniform measure on the unit disc.
 The derivation of \eqref{it: weak conv} is standard and will be sketched at the end of this section.
We will not calculate the measures $\nu_z$ explicitly. Instead, it will be enough to show that $\nu_{n,z}-\nu_{n,z}^G$ converges vaguely to $0$ in probability for a.e.  $z \in \C$. Here $\nu_{n,z}^G$ is the empirical measure of singular values of $\frac{1}{\sqrt{n}}G_n$, and $G_n$ is the $n \times n$ matrix with i.i.d. $N(0,1)$ entries.
Since the circular law for the Gaussian matrices is known, this  uniquely defines the measures $\nu_z$.

The main effort will be devoted to proving \eqref{it: uniform int}. The logarithmic function has singularities at $0$ and $\infty$. Establishing uniform integrability at $\infty$ is very simple and relies on the fact that $\E \norm{\frac{1}{\sqrt{p_n n}}A_n}_{HS}^2$ is bounded.
The proof of uniform integrability at $0$ uses the estimates for the smallest and the smallish singular values derived in Sections \ref{sec: smallest} and \ref{sec: intermediate} respectively.
Yet, the bound for the singular values $s_{n-k}( \frac{1}{\sqrt{p_n n}}A_n- z \Id_n)$ obtained in Theorem \ref{thm: intermediate singular} is too loose to be applied for all $k$. We will be able to use it only for $k < \frac{n}{\log^C(p_n n)}$. For larger $k$, we need a tighter bound.
To this end, we use the idea of \cite{Cook circ}  based on the comparison of $\nu_{n,z}([0,s])$ and $\nu_{n,z}^G([0,s])$ for sufficiently large $s$.
In our case, the comparison with the Gaussian matrix does not seem to be feasible. Instead, we compare $\nu_{n,z}([0,s])$ with the empirical measure of the singular values of a new random matrix obtained by replacing relatively small values of $ \frac{1}{\sqrt{p_n n}}A_n- z \Id_n$ by i.i.d. $N(0,1)$ variables.
This will require bounding the Stieltjes transform of 
Gaussian matrices with partially frozen entries. Such bound is obtained in Subsection \ref{sec: shifted}. The uniform integrability is established in Subsection \ref{sec: uniform}. Finally, in Subsection \ref{sec: completion}, we establish the vague convergence and complete the proof of Theorem \ref{th: main}.

\subsection{
Gaussian matrices with partially frozen entries} \label{sec: shifted}

\begin{lemma}
Let $n>k\geq 1$, and let $E\subset\C^n$ be a linear subspace of co-dimension at least $2k$.
Let $X=(X_1,X_2,\dots,X_n)$ be a random vector in $\C^n$ with mutually independent coordinates,
and assume that at least $n-k$ coordinates are real Gaussian variables of unit variance and possibly different means.
Then
$$\Prob\big\{\dist(X,E)\leq c\sqrt{k}\big\}\leq e^{-ck},$$
where $c>0$ is a universal constant.
\end{lemma}
\begin{proof}
Denote by $I$ the set of all indices corresponding to the Gaussian variables, so that $|I|\geq n-k$.
Further, condition on any realization of coordinates $X_i$, $i\in I^c$.

Let $\Proj$ be the coordinate projection onto the span of $e_i$, $i\in I$.
Obviously, we have
$$\dist(X,E)\geq \dist(\Proj(X), \Proj(E)),$$
where $\Proj(E)$ has co-dimension at least $k$, when viewed as a subspace of $\C^I$.
On the other hand, $\Proj(X)$ is a real Gaussian vector in $\C^I$ with identity covariance matrix,
and the statement of the lemma follows as a consequence of standard concentration inequalities.
\end{proof}

A combination of the above lemma with the negative second moment identity yields
\begin{prop}\label{prop: shifted gauss}
Let $n>k\geq 1$, let $V=(v_{ij})$ be an $n\times n$ random matrix with mutually independent entries such that
for any $j\leq n$, at least $n-k$ components of the $j$-th column of $V$
are real Gaussian variables of unit variance. Then
$$\Prob\big\{s_{n-3k+1}(V)\leq ck/\sqrt{n}\big\}\leq ne^{-ck},$$
where $c>0$ is a universal constant.
\end{prop}
\begin{proof}
Let $\widetilde V$ be the $n\times (n-2k)$ matrix obtained from $V$ by removing the last $2k$ columns.
Obviously, we have
$$s_{n-3k+1}(V)\geq s_{n-3k+1}(\widetilde V).$$
Further, to estimate $s_{n-3k+1}(\widetilde V)$, observe that, by the negative second moment identity,
$$k\,s_{n-3k+1}(\widetilde V)^{-2}\leq
\sum_{j=1}^{n-2k}s_{j}^{-2}(\widetilde V)=\sum\limits_{j=1}^{n-2k}\dist(\col_j(\widetilde V),\spn\{\col_i(\widetilde V):\,i\neq j\})^{-2}.$$
By the above Lemma, we have
$$\Prob\big\{\dist(\col_j(\widetilde V),\spn\{\col_i(\widetilde V):\,i\neq j\})\leq c\sqrt{k}\big\}\leq e^{-ck}$$
for any $j\leq n-2k$. Taking the union bound,
we get the result.
\end{proof}

 Given an $n\times n$ matrix $M$ and a complex number $z$, denote by
 $H_z(M)$ the $2n \times 2n$ matrix of the form
 \[
  H_z(M)=
  \begin{pmatrix}
   0 & B_z(M) \\ B_z^*(M) & 0
  \end{pmatrix},
 \]
 where $B_z=\frac{1}{\sqrt{n}} M- z\,\Id_n$.
 The eigenvalues of $H_z(M)$ are the singular values of $B_z$ and their negatives.
 Further, given $w\in\C$, denote by $m_w(M)=m_{w,z}(M)$
 the Stieltjes transform of the empirical measure of the eigenvalues of $H_z(M)$:
 \[
  m_w(M):= \frac{1}{2n} \cdot \tr (H_z(M)-w\,\Id_{2n})^{-1}.
 \]
As an immediate corollary of Proposition~\ref{prop: shifted gauss}, we get
\begin{cor}\label{cor: shifted gauss}
Let $n>k\geq 1$,
and let $V$ be an $n\times n$ matrix
with mutually independent entries such that
for every $j\leq n$, at least $n-k$ coordinates of the $j$-th column of $V$ are real Gaussian variables of unit variance.
Let $z\in\C$ and let $w\in\C$ be such that $\Re(w)=0$ and $\Im(w)\geq k/n$.
Then with probability at least $1-n^2\,e^{-ck}$ we have
$$\Im(m_w(V))\leq C,$$
for some universal constants $C,c>0$.
In particular, if $k\geq C'\log n$ for a sufficiently large constant $C'>0$ then necessarily
$$\Exp\,\Im(m_w(V))\leq C'.$$
\end{cor}
\begin{proof}
Without loss of generality, we can assume that $\Im(w)=k/n$.
First, by applying Proposition~\ref{prop: shifted gauss} to matrix $\sqrt{n} B_z(V)=V-z\sqrt{n}\,\Id_n$,
we get with probability at least $1-n^2\,e^{-ck}$:
$$s_{n-i}(B_z(V))\geq \frac{ci}{n}\;\;\mbox{ for all }i\geq Ck.$$
On this event we have
\begin{align*}
\Im(m_w(V))&= \frac{1}{2n}\sum\limits_{i=1}^n \Im\bigg(\frac{1}{s_i(B_z(V))-w}+\frac{1}{-s_i(B_z(V))-w}\bigg)\\
&= \frac{1}{2n}\sum\limits_{i=1}^n \frac{2\Im(w)}{|w|^2+s_i^2(B_z(V))}
\leq \tilde C
+
\tilde Ck\sum\limits_{i=k+1}^n \frac{1}{i^2}\leq \bar C.
\end{align*}
On the complement of this event we can use the trivial bound $\Im(m_w(V))\leq \frac{1}{\Im(w)}\leq n$.
Hence, if $k\leq C_1\log n$ for a sufficiently large constant $C_1$, the combination of the two bounds
gives $\Exp\,\Im(m_w(V))\leq C_2$.
\end{proof}

\subsection{Uniform integrability of the logarithm} \label{sec: uniform}
Here is the main result of the subsection:
\begin{prop}[Uniform Integrability]  \label{prop: uniform integrability}
Let $A_n$ be a sequence of random matrices as in Theorem~\ref{th: main}.
For any $z \in \C$ denote by $\nu_{n,z}$ the empirical measure of the singular values of the matrix
$\frac{1}{\sqrt{p_n n}}A_n-z\Id_n$.
Then for any $z \in \C$ with $\Im(z)\neq 0$,
the function $f(x)= \log x$ is uniformly integrable with respect to measures
$\nu_{n,z}$, i.e., for any $\e>0$, there exists $T>0$ (determined by $\e$ and $z$) such that
 \[
  \limsup_{n \in \N} \P \left\{ \int_{|\log s| >T} |\log s | \, d \nu_{n,z}(s) > \e \right\} < \e.
 \]
\end{prop}

Before proving the proposition, let us consider some auxiliary lemmas.
The first is an elementary observation on conditional distributions.
\begin{lemma}\label{l: elem condit}
Let $\Lambda=(\xi_{ij})$ be an $n\times n$ random matrix with i.i.d real valued entries of mean $\theta$
and unit variance. Further, for any $L\geq 1$ and any subset $Q\subset[n]\times[n]$,
let $\Event_{L,Q}$ be the event that $|\xi_{ij}-\theta|> L$ for all $(i,j)\in Q$ and
$|\xi_{ij}-\theta|\leq L$ for all $(i,j)\in Q^c$. Then
\begin{itemize}

\item Conditioned on any $\Event_{L,Q}$ with $\Prob(\Event_{L,Q})>0$, the entries of $\Lambda$ are mutually independent;
\item There is $C_\xi>0$ determined by the distribution of $\xi_{ij}$'s
such that, whenever $L\geq C_\xi$ and $Q\subset[n]\times[n]$
satisfy $\Event_{L,Q}\neq \emptyset$, for any $(i,j)\in Q^c$ we have $|\Exp(\xi_{ij}\,|\,\Event_{L,Q})-\theta|\leq \frac{2}{L}$
and $\frac{1}{2}\leq \Var(\xi_{ij}\,|\,\Event_{L,Q})\leq 1$.

\end{itemize}
Moreover, denoting by $\mathcal P_L$ the collection of all subsets $Q\subset[n]\times [n]$
such that $\Prob(\Event_{L,Q})>0$ and
$$|\{i\leq n:\;(i,j)\in Q\}|,|\{i\leq n:\;(j,i)\in Q\}|\leq \frac{2n}{L^2}\;\;\mbox{ for all }\;\;j\in[n],$$
we have for all $L\geq 1$:
$$\Prob\Big(\bigcup\limits_{Q\in\mathcal P_L}\Event_{L,Q}\Big)\geq 1-2n\,e^{-2n/L^4}.$$
\end{lemma}
\begin{proof}
Without loss of generality, $\theta=0$.
The mutual independence of the entries conditioned on $\Event_{L,Q}$ is obvious.
Further, we have for any $(i,j)\in Q^c$:
$$|\Exp(\xi_{ij}\,|\,\Event_{L,Q})|
=\frac{|\Exp(\xi_{ij}\indicator_{|\xi_{ij}|\leq L})|}{\Prob\{|\xi_{ij}|\leq L\}}
=\frac{|\Exp(\xi_{ij}\indicator_{|\xi_{ij}|> L})|}{\Prob\{|\xi_{ij}|\leq L\}}
\leq \frac{1}{L\Prob\{|\xi_{ij}|\leq L\}}\leq \frac{2}{L},$$
if $L\geq \sqrt{2}$, where we used Cauchy--Schwartz' and Markov's inequalities.
Denote $\psi:=\Exp(\xi_{ij}\,|\,\Event_{L,Q})$. Then
$$\Var(\xi_{ij}\,|\,\Event_{L,Q})=\Exp(\xi_{ij}^2\,|\,\Event_{L,Q})-\psi^2
=\frac{\Exp(\xi_{ij}^2\indicator_{|\xi_{ij}|\leq L})}{\Prob\{|\xi_{ij}|\leq L\}}-\psi^2
\geq \frac{1}{2},$$
provided that $L$ is sufficiently large.

Finally, observe that for any $i\leq n$, the event
$$\big\{|\{j\leq n:\;|\xi_{ij}|>L\}|\geq 2n/L^2\big\}$$
has probability at most $e^{-2n/L^4}$ (by applying Bernstein's inequality).
Taking the union bound and combining this with the definition of $\mathcal P_L$, we get the result.
\end{proof}

 In what follows, we will need the next result of Chatterjee \cite[Theorem 1.1]{Cha}.
 \begin{theor} \label{thm: Chatterjee}
  Let $N$ be a natural number, and let $X$ and $W$ be independent random vectors in $\R^N$ with independent components satisfying
  $\E X_j=\E W_j, \ \E X_j^2= \E W_j^2$ for any $j \in [N]$.
  Assume that
  \[
  \gamma_3=: \max_{j \in [N]} \max \left( \E |X_j|^3, \E |W_j|^3 \right) < \infty.
  \]
  Let $f \in C^3(\R^N)$ and denote
  \[
   \l_3(f)= \sup_{x \in \R^N} \max_{r=1,2,3} \max_{J \in [N]^3} |\partial_J^r f(x)|^{3/r}.
  \]
  Then
  \[
  |\E f(X)- \E f(W)| \le C \gamma_3 \l_3(f) N,
  \]
 where $C$ is a universal constant.
 \end{theor}

The next lemma appears as a combination
of results from \cite{Cook circ} and observations made in the previous subsection.
\begin{lemma}\label{l: measure via comparison}
Let $(A_n)$ be a sequence of random matrices from Theorem~\ref{th: main}, and set $\theta:=\Exp\xi$.
For any $z\in\C$ with $\Im(z)\neq 0$ we have
$$\E\,\nu_{n,z}([0,\eta])\leq C_{\smallrefer{l: measure via comparison}}\eta\quad\mbox{for all }\eta\geq (p_n n)^{-c},$$
where $\nu_{n,z}$ is defined as in Proposition~\ref{prop: uniform integrability},
$C_{\smallrefer{l: measure via comparison}}>0$ depends only on $z$ and $\theta$ (and not on $n$) and $c>0$ is a universal constant.
\end{lemma}
\begin{proof}
 An elementary comparison between the indicator function and the Poisson kernel implies that for any $\eta>0$, $n\in\N$
 and $z\in\C$, for any $n\times n$ matrix $\widetilde M$, and for $m_{i\eta}$ defined the same way as in the previous subsection,
 we have
 \begin{equation} \label{eq: Stieltjes}
    \E \nu([0,\eta]) \le \widetilde C \eta\, \E\, \Im(m_{i \eta} (\widetilde M)),
 \end{equation}
 where $\nu$ denotes the normalized counting measure of singular values of $\frac{1}{\sqrt{n}}\widetilde M-z\,\Id_n$
 (see, e.g. \cite[Section~2.4.3]{Tao book}).

 Fix $z\in\C$ with $\Im(z)\neq 0$.
 Let $\eta\in [(p_n n)^{-1/20},c']$, for a small enough constant $c'>0$.
 Set $L:=\eta^{-2}$ and
 fix for a moment any subset $Q\in\mathcal P_L$, where $\mathcal P_L$
 is defined as in Lemma~\ref{l: elem condit}.
 Set $\psi:=\frac{\Exp(\xi\indicator_{|\xi-\theta|\leq L})}{\Prob\{|\xi-\theta|\leq L\}}$ and
 $\tau:=\Var(\xi\;|\;|\xi-\theta|\leq L)$.
 We assume that constant $c'$ is sufficiently small so that all assertions of Lemma~\ref{l: elem condit} hold true for $L$;
 in particular, $|\psi-\theta|\leq \frac{2}{L}$ and $\tau\in[1/2,1]$.
 Observe that the random matrix
 $M:=(p_n\tau)^{-1/2}A_n=(p_n\tau)^{-1/2}(\delta_{ij}\xi_{ij})_{ij}$ has mutually independent entries; moreover, conditioned on $\Event_{L,Q}$,
 for each $(i,j)\in Q^c$ the $(i,j)$--th entry of $M$ has unit variance,
 and all entries corresponding to $(i,j)\in Q^c$
 are uniformly bounded (by absolute value) by $(L+|\theta|)(p_n\tau)^{-1/2}$.

 Let us represent the probability space $\Omega$ as the product space $\Omega:=\Omega_Q\times \Omega_{Q^c}$,
 where the decomposition is generated by partitioning the set of entries of $M$ into the subset indexed over $Q$
 and the subset indexed over $Q^c$. Fix any point $(\omega_Q,\omega_{Q^c})\in\Event_{L,Q}$, and define
 $\widetilde \Event_{L,Q}:=(\{\omega_{Q}\}\times\Omega_{Q^c})\cap \Event_{L,Q}$.
 This way, everywhere on $\widetilde \Event_{L,Q}$ the entries of $M$ indexed over $Q$ are frozen whereas
 the conditional distribution of the entries indexed over $Q^c$ is the same when conditioned on $\widetilde \Event_{L,Q}$
 and when conditioned on $\Event_{L,Q}$.
 Further, let $N$ be the cardinality of $Q^c$, let $X=(X_s)_{s\in Q^c}$ be the random vector of entries of $M$ indexed over $Q^c$,
 and let $W=(W_s)_{s\in Q^c}$ be the vector of independent real Gaussian variables of unit variance and
 mean $\sqrt{\frac{p_n}{\tau}}\psi$, indexed over $Q^c$.

 We will apply Theorem~\ref{thm: Chatterjee} to vectors $X$ and $W$.
 Conditioned on $\widetilde\Event_{L,Q}$, we have
 $$\Exp (X_s\mid \widetilde\Event_{L,Q})
 =\Exp\big((p_n\tau)^{-1/2}\delta\xi \mid |\xi-\theta|\leq L\big)=\sqrt{\frac{p_n}{\tau}}\psi=\Exp W_s,
 \quad s\in Q^c.$$
 Further, $\Exp (X_s^2\mid \widetilde\Event_{L,Q})=1$ and
 \begin{align*}
 \Exp(|X_s|^3\mid \widetilde\Event_{L,Q})
 &=\Exp\big((p_n\tau)^{-3/2}\delta\xi^3 \mid |\xi-\theta|\leq L\big)\\
 &\leq p_n^{-1/2}\tau^{-3/2}(L+|\theta|)\E (\xi^2\;|\;|\xi|\leq L)\leq C'(L+|\theta|)p_n^{-1/2}(1+\theta^2),\;\;s\in Q^c.
 \end{align*}
 Thus,
 \begin{align*}
  \gamma_3 &\le C'(L+|\theta|)p_n^{-1/2}(1+\theta^2)\leq C_\theta L p_n^{-1/2},
 \end{align*}
 where $\gamma_3$ is defined as in Theorem~\ref{thm: Chatterjee} and $C_\theta>0$ may only depend on $\theta$.
 Now, we construct the function $f:\R^{Q^c}\to \R_+$ as follows.
 Take any vector $V=(V_s)_{Q^c}$ indexed over $Q^c$. Then we construct an $n\times n$ matrix $\widetilde V=(\widetilde v_{ij})$
 by setting $\widetilde v_{ij}:=V_{(ij)}$ whenever $(i,j)\in Q^c$, and setting $\widetilde v_{ij}$, $(i,j)\in Q$
 to the values of the entries of $M$ fixed by our choice of event $\widetilde \Event_{L,Q}$.
 Finally, we set $f(V):=\Im(m_{i\eta}(\widetilde V))$, where $m_{i\eta}$ is defined as in the previous subsection.
 The following bound for $\l_3(f)$ can be extracted from \cite[Proof of Proposition 8.2]{Cook circ}:
 \[
  \l_3(f) \le \frac{C}{n^{5/2} \eta^4}.
 \]

 Substituting the two above estimates in Theorem~\ref{thm: Chatterjee}, we obtain
 \begin{align*}
  \big| \E \big(f(X)\;|\;\widetilde\Event_{L,Q}\big)- \E f(W)\big|
  \le C''L p_n^{-1/2} \cdot\frac{1}{n^{5/2} \eta^4}\cdot n^2
  \le C'''\eta,
 \end{align*}
 by our choice of $L$ and since $\eta\geq (p_n n)^{-1/20}$.
 Next, we estimate $\E f(W)$ using Corollary~\ref{cor: shifted gauss}.
 Set $k:=\lfloor 2n/L^2\rfloor$ and observe that, by the definition of $\mathcal P_L$, we have
 $|\{i\leq n:\;(i,j)\in Q^c\}|\geq n-k$ for all $j\in[n]$. Further, by our choice of $L$ and $\eta$ we clearly have $\eta\geq k/n$.
 Thus, by our definition of $f$ and by Corollary~\ref{cor: shifted gauss}, we get
 $\E f(W)\leq \tilde C$. Note that the above estimate does not depend on the particular realization of elements of $M$ indexed over $Q$.
 This implies
 $$\Exp(m_{i \eta}(M)\,|\,\Event_{L,Q})\leq \bar C\eta,$$
 and so, by \eqref{eq: Stieltjes},
 $$\Exp(\nu_{z,M}([0,\eta])\,|\,\Event_{L,Q})\leq C''\eta\;\;\; \mbox{uniformly for all $z\in\C$ and $\eta\in[(p_n n)^{-1/20},c']$},$$
 where $\nu_{z,M}$ denotes the singular value distribution of the matrix $\frac{1}{\sqrt{n}}M-z\,\Id_n$.
 Using that $\tau\geq 1/2$ and in view of the identity $\nu_{n,z}([0,\tau^{1/2} t])=\nu_{z,M}([0,t])$, $t\in\R_+$, we get
 $$\Exp(\nu_{n,z}([0,\eta/2])\,|\,\Event_{L,Q})\leq C''\eta\;\;\; \mbox{uniformly for all $z\in\C$ and $\eta\in[(p_n n)^{-1/20},c']$},$$
 where $C''>0$ may only depend on $\theta$.
 As a final step, note that, by Lemma~\ref{l: elem condit}, the union of the events $\Event_{L,Q}$, with $Q\in\mathcal P_L$,
 has probability at least $2n\,e^{-2n/L^4}$.
 The result follows.
\end{proof}

\begin{proof}[Proof of Proposition~\ref{prop: uniform integrability}]
 For each $n,z$, denote the matrix $\frac{1}{\sqrt{p_n n}}A_n-z\Id_n$ by $V_{n,z}$.
 The function $f(x)=\log(x)$ is unbounded as $x \to \infty$ and $x \to 0$. The first singularity is much easier to handle.
 Let $T$ be such that $|z| <e^{T/2}/2$.
 Assume that $T \ge 1$.
 Since the function $\frac{\log x}{x^2}$ is decreasing for $x \ge e$, we have
 \[
 \int_{s>e^T} |\log s | \, d \nu_{n,z}(s)
 \le  \int_{s>e^T} T e^{-2T} s^2 \, d \nu_{n,z}(s)
 = \frac{ T e^{-2T}}{n} \sum_{s_j(V_{n,z}) >e^T } s_j^2(V_{n,z})
 \leq \frac{T e^{-2T}}{n} \|V_{n,z}\|_{HS}^2.
 \]
 Since
 \[
  \E \frac{1}{n} \|V_{n,z}\|_{HS}^2 \leq 2+2|z|^2\leq 2+e^{T}\leq 2e^T,
 \]
the uniform integrability at $\infty$ follows from Markov's inequality.

\medskip

Let us prove the uniform integrability at $0$. Again, we take a parameter $T\geq 1$.
For any $n$ and $z$, let $\EE_{\min}=\EE_{min}(n,z)$ be the event that
\[
 s_{\min} \left(\frac{1}{\sqrt{p_n n}}A_n-z\Id_n \right)  \ge \exp \left(-C_z \log^3 n \right),
\]
where $C_z>0$ depends only on $z$ and is chosen in such a way that
$\Prob(\EE_{\min})\geq 1-C_z(p_nn)^{-c_{\smallrefer{th: bound on smin}}}$
(this can be done because of Theorem~\ref{th: bound on smin}).
Further, let $\EE_1=\EE_{1}(n,z)$
be the event that for any $k\geq \frac{n}{\log^{100} n}$, we have
\begin{equation}\label{eq: int sv}
 s_{n-k} \left(\frac{1}{\sqrt{p_n n}}A_n-z\Id_n \right)
\ge \exp \left(-C \log^{100} \left(\frac{4n}{k} \right) \cdot  \log^{100} (p_n n) \right),
\end{equation}
where $C>0$ (independent of $n$) is chosen so that $\Prob(\EE_1)\geq 1-(p_n n)^{-c'}$
(this is possible by Theorem~\ref{thm: intermediate singular}).
Furthermore, let $\EE_2=\EE_2(n,z)$ be the event that  for any $\eta \in[ \log^{-300} (p_n n),e^{-T}]$, we have
\begin{equation}\label{eq: sqrt eta}
 \nu_{n,z}([0,\eta])\le C_{\smallrefer{l: measure via comparison}} \sqrt{\eta},
\end{equation}
where the constant $C_{\smallrefer{l: measure via comparison}}$ is taken from Lemma~\ref{l: measure via comparison}.
Observe that
$$\Prob(\EE_2^c)\leq\Prob\big\{\exists i\in [T,300\log\log(p_n n)]\mbox{ such that $\nu_{n,z}([0,e^{-i}])\geq
C_{\smallrefer{l: measure via comparison}} e^{-i/2-1/2}$}\big\}.$$
Combining this with the bound for the expectation of $\nu_{n,z}([0,e^{-i}])$ from Lemma~\ref{l: measure via comparison}
and Markov's inequality, we get
$$\Prob(\EE_2)\geq 1-C''e^{-T/2}.$$

Let us introduce two quantiles of the measure $ \nu_{n,z}$. Set
\begin{align*}
 t_1 &= \sup \left \{ t \ge 0: \  \nu_{n,z}([0,t]) \le \frac{1}{\log^4 n} \right \}, \quad \text{and} \\
  t_2 &= \min\bigg(\sup \left \{ t \ge 0: \  \nu_{n,z}([0,t]) \le \frac{1}{\log^{200} (p_n n)} \right \},\frac{1}{\log^{400} (p_n n)}\bigg).
\end{align*}
Note that on the event $\EE_2$ we have $t_2 \ge   \frac{c}{\log^{400} (p_n n)}$.
Assume that the event $\EE_{min} \cap \EE_1 \cap \EE_2$ occurs.
Then
\[
 \int_0^{t_1} |\log s| \, d \nu_{n,z}(s)
 \le C' \log^3 n \cdot \nu_{n,z}([0,t_1])
 \le \frac{C'}{\log n}.
\]
Assume for a moment that $t_1 \le t_2$.
Denote
\[
 k_1 := \Big\lfloor\frac{n}{\log^4 n}\Big\rfloor  \quad \text{and} \quad
  k_2  :=  \Big\lceil\frac{n}{\log^{200} (p_n n)}\Big\rceil.
\]
Then, by \eqref{eq: int sv},
\begin{align*}
  &\int_{t_1}^{t_2}  |\log s| \, d \nu_{n,z}(s)
 \quad \leq \frac{1}{n} \sum_{k=k_1}^{k_2} \Big|\log s_{n-k}\Big(\frac{1}{\sqrt{p_n n}}A_n-z\,\Id_n\Big)\Big| \\
 &\le \frac{1}{n} \sum_{k=k_1}^{k_2} C \log^{100} \left(\frac{4n}{k} \right) \cdot  \log^{100} (p_n n)
 \le C' \log^{100} (p_n n) \cdot \int_0^{\log^{-200} (p_n n)} \log^{100} \left( \frac{4}{x} \right) \, dx \\
 &\le \frac{1}{\log^{50} (p_n n)}.
\end{align*}
Set $t_3:=\max(t_1,t_2)$.
For $s \in [t_3,e^{-T}]$ we use the bound
\[
 |\log s| \le C \sum_{m=T}^{\log (1/t_3)} \mathbf{1}_{[0,e^{-m}]}(s),
\]
which, by \eqref{eq: sqrt eta}, yields
\[
   \int_{t_3}^{e^{-T}}  |\log s| \, d \nu_{n,z}(s)
   \le C \sum_{m=T}^{\log (1/t_3)} \nu_{n,z} ( [0,e^{-m}])
   \le C'' e^{-T/2}.
\]
Combining the three previous inequalities, we conclude that
\[
    \int_0^{e^{-T}}  |\log s| \, d \nu_{n,z}(s) \le C e^{-T/2} +\beta(n,p_n),
\]
where $\beta(n,p_n)$ is a deterministic term which tends to $0$ as $p_n n \to \infty$.
Since $\P(\EE_{min} \cap \EE_1 \cap \EE_2) \to 1$ as $p_n n,T \to \infty$, the uniform integrability is proved.
\end{proof}

\subsection{Completion of the proof} \label{sec: completion}
Let $\mu$ be the uniform measure on the unit disc in $\C$.
To complete the proof of Theorem \ref{th: main}, we have to check the vague convergence of the measures $\nu_{n,z}$ to some deterministic measures $\nu_z$ such that
  \begin{equation} \label{eq: log potential}
   \int_{\C} \log |\l-z| \, d \mu(z)= \int_0^\infty \log s \, d \nu(s).
  \end{equation}
  As in \cite[Lemma 9.1]{BR circ}, it is enough to prove this convergence, assuming that the random variable $\xi$ (and so all entries of $A_n$) are  bounded. The proof of this fact is standard  and relies on truncation, an application of the Hoffman-Weilandt inequality and the fact that the weak convergence is metrized by the bounded Lipschitz metric.
  We omit the details as they appear in a number of random matrix papers (see e.g., \cite[Proposition 4.1]{BDJ}).

  For the empirical measures $\nu_{n,z}^G$ of singular values of real Gaussian matrices, this convergence and \eqref{eq: log potential} are known, see, e.g. \cite{Edelman}. Thus, it is enough to prove that the measures $\nu_{n,z}-\nu_{n,z}^G$ converge to $0$ vaguely in probability.
  This step closely follows the argument of \cite{Cook circ}, so we will only sketch it.
  Without loss of generality, we can check the vague convergence only for Lipschitz functions.
  By \cite[Lemma 9.2]{BR circ}, which is a variant of \cite[Lemma 9.1]{Cook circ},
  \[
   \int f(s) \, d \nu_{n,z}(s) - \E  \int f(s) \, d \nu_{n,z}(s) \to 0 \quad \text{in probability}
  \]
 for any Lipschitz $f:(0,\infty) \to \R$ with compact support. By the same lemma it also holds for the measures $\nu_{n,z}^G$. Therefore, it is enough to prove that
 \[
   \E \int f(s) \, d \nu_{n,z}(s) - \E  \int f(s) \, d \nu_{n,z}^G(s) \to 0.
 \]
 This convergence would follow if we prove the convergence of the expectations of Stieltjes transforms, more precisely from
\[
 \E m_w \left( \frac{1}{\sqrt{p_n}} A_n \right) - \E m_w \left( \frac{1}{\sqrt{n}} G_n \right) \to 0
 \quad \text{for all } w \in \C \text{ with } \Im(w)>0,
\]
where $G_n$ is the standard $n \times n$ Gaussian matrix. The convergence above follows in turn from \cite[Lemma 9.4]{BR circ}, which is an extension of \cite[Lemma 8.2]{Cook circ} to general random matrices with bounded entries.

This completes the proof of Theorem \ref{th: main}.

\end{document}